\documentclass[12pt, reqno]{amsart}
\usepackage[normalem]{ulem}
\usepackage{amsmath, amsthm, amssymb}
\usepackage[shortlabels]{enumitem}
\usepackage[margin=1.0in]{geometry}
\usepackage{xcolor}
\definecolor{cite}{rgb}{0.30,0.60,1.00}
\definecolor{url}{rgb}{0.00,0.00,0.80}
\definecolor{link}{rgb}{0.40,0.10,0.20}

\usepackage{zref-clever}

\AddToHook{env/theorem/begin}{%
	\zcsetup{countertype={theorem=theorem}}}
\AddToHook{env/claim/begin}{%
	\zcsetup{countertype={theorem=theorem}}}  
\AddToHook{env/proposition/begin}{%
	\zcsetup{countertype={theorem=proposition}}}
\AddToHook{env/lemma/begin}{%
	\zcsetup{countertype={theorem=lemma}}}
\AddToHook{env/conjecture/begin}{%
	\zcsetup{countertype={theorem=theorem}}}  
\AddToHook{env/corollary/begin}{%
	\zcsetup{countertype={theorem=corollary}}}
\AddToHook{env/definition/begin}{%
	\zcsetup{countertype={theorem=definition}}}
\AddToHook{env/remark/begin}{%
	\zcsetup{countertype={theorem=remark}}}
\AddToHook{env/example/begin}{%
	\zcsetup{countertype={theorem=example}}}

\zcsetup{nameinlink=false}

\usepackage[pdfusetitle,colorlinks,linkcolor=link,urlcolor=url,citecolor=cite,pagebackref,breaklinks]{hyperref}
\usepackage{graphicx}
\usepackage{mathdots}
\usepackage{tikz-cd}
\usepackage{comment}
\usepackage{xypic}
\usepackage{mathtools}
\usepackage{float}

\newtheorem{theorem}{Theorem}[section]

\newtheorem{proposition}[theorem]{Proposition}
\newtheorem{lemma}[theorem]{Lemma}

\newtheorem{corollary}[theorem]{Corollary}
\theoremstyle{definition}
\newtheorem{definition}[theorem]{Definition}
\theoremstyle{definition}
\newtheorem{remark}[theorem]{Remark}
\theoremstyle{definition}
\newtheorem{example}[theorem]{Example}

\newcommand{\cref}[1]{\zcref{#1}}
\newcommand{\Cref}[1]{\zcref[S]{#1}}

\newcommand{\setItemnumber}[1]{%
	\setcounter{enumi}{\numexpr\value{enumi}-1\relax}%
	\renewcommand{\labelenumi}{#1)}%
}

\newcommand{\zIntegers}{\mathbb{Z}}
\newcommand{\cComplex}{\mathbb{C}}
\newcommand{\multiplicativegroup}[1]{#1^{\times}}
\newcommand{\detQuadratic}{{\det}_{\slash \quadraticExtension}}
\newcommand{\Hom}{\mathrm{Hom}}
\newcommand{\EndomorphismRing}{\operatorname{End}}
\newcommand{\idmap}{\mathrm{id}}
\newcommand{\conjugate}[1]{\overline{#1}}
\newcommand{\lengthof}{\ell}
\newcommand{\sizeof}[1]{\left|#1\right|}
\newcommand{\hermitianSpace}{\mathrm{V}}
\newcommand{\xIsotropic}{\mathrm{X}}
\newcommand{\yIsotropic}{\mathrm{Y}}
\newcommand{\similitudeCharacter}{\operatorname{sim}}
\newcommand{\etale}{\'etale }
\newcommand{\Erdelyi}{Erd{\'e}lyi}
\newcommand{\Toth}{T{\'o}th}

\newcommand{\innerproduct}[2]{\left\langle #1,#2\right\rangle}

\newcommand{\fieldCharacter}{\psi}
\newcommand{\centralCharacter}[1]{\omega_{#1}}
\newcommand{\Ind}[3]{\mathrm{Ind}_{#1}^{#2}\left(#3\right)}
\newcommand{\Contragradient}[1]{#1^{\vee}}

\newcommand{\grpIndex}[2]{\left[#1:#2\right]}

\newcommand{\transpose}[1]{\, {}^{t}#1}
\newcommand{\involution}[1]{#1^{c}}
\newcommand{\minusInvolution}[1]{#1^{-c}}
\newcommand{\involutionPlusOne}[1]{#1^{1+c}}
\newcommand{\minusInvolutionMinusOne}[1]{#1^{-1-c}}
\newcommand{\IdentityMatrix}[1]{I_{#1}}
\newcommand{\diag}{\mathrm{diag}}
\newcommand{\trace}{\operatorname{tr}}
\newcommand{\GL}{\mathrm{GL}}
\newcommand{\SO}{\mathrm{SO}}
\newcommand{\GSO}{\mathrm{GSO}}
\newcommand{\Sp}{\mathrm{Sp}}
\newcommand{\GSp}{\mathrm{GSp}}
\newcommand{\UnitaryGroup}{\mathrm{U}}
\newcommand{\UnipotentRadical}{N}
\newcommand{\GroupExtension}[1]{\widetilde{#1}}

\newcommand{\FieldNorm}[2]{\mathrm{N}_{#1:#2}}
\newcommand{\aFieldNorm}{\mathrm{N}}
\newcommand{\finiteField}{\mathbb{F}}
\newcommand{\quadraticExtension}{\mathbb{E}}
\newcommand{\finiteFieldExtension}[1]{\finiteField_{#1}}
\newcommand{\quadraticFieldExtension}[1]{\quadraticExtension_{#1}}
\newcommand{\NormOneGroup}[1]{\finiteFieldExtension{#1}^{\aFieldNorm = 1}}
\newcommand{\algebraicClosure}[1]{\overline{#1}}
\newcommand{\Galois}{\operatorname{Gal}}
\newcommand{\Frobenius}{\operatorname{Fr}}
\newcommand{\restrictionOfScalars}[3]{\operatorname{Res}_{#1 \slash #2}{#3}}
\newcommand{\multiplcativeScheme}{\algebraicGroup{G}_m}
\newcommand{\squareMatrix}{\operatorname{Mat}}
\newcommand{\Mat}[2]{\operatorname{Mat}_{#1 \times #2}}
\newcommand{\GaussSum}[2]{\mathcal{G}\left(#1, #2\right)}
\newcommand{\GaussSumSingleCharacter}[2]{\tau\left(#1, #2\right)}
\newcommand{\GaussSumSingleTorusCharacter}[3]{\tau_{#1}\left(#2, #3\right)}
\newcommand{\posDblJacobiSum}[2]{\mathcal{J}^{\mathrm{dbl}}\left(#1, #2\right)}
\newcommand{\JacobiKernel}[1]{\Phi_{#1}}
\newcommand{\posJacobiKernel}[1]{\Phi_{#1}}
\newcommand{\posHermitianJacobiKernel}[2]{\Phi_{#1,#2}}

\newcommand{\GaussSumScalar}[2]{\mathrm{G}\left(#1, #2\right)}
\newcommand{\posDblJacobiSumScalar}[2]{\mathrm{J}^{\mathrm{dbl}}\left(#1, #2\right)}
\newcommand{\dblPosVirtualJacobiSumScalar}[2]{\mathrm{j}^{\mathrm{dbl}}\left(#1, #2\right)}
\newcommand{\posDblVirtualJacobiSumScalar}[2]{\mathrm{j}_{+}^{\mathrm{dbl}}\left(#1, #2\right)}
\newcommand{\dblGammaFactor}[3]{\Gamma^{\mathrm{dbl}}\left(#1 \times #2, #3\right)}
\newcommand{\dblGammaFactorSpace}[4]{\Gamma^{\mathrm{dbl}}_{#1}\left(#2 \times #3, #4\right)}
\newcommand{\dblLanglandsGammaFactor}[3]{\gamma^{\mathrm{dbl}}\left(#1 \times #2, #3\right)}
\newcommand{\GaussSumCharacter}[3]{\tau\left(#1 \times #2, #3\right)}
\newcommand{\GaussSumTorusCharacter}[4]{\tau_{#1}\left(#2 \times #3, #4\right)}
\newcommand{\ladicnumbers}{\algebraicClosure{\mathbb{Q}_{\ell}}}
\newcommand{\IsometryGroup}{\mathrm{Isom}}
\newcommand{\lieAlgebra}{\mathfrak{g}}
\newcommand{\DeligneLusztigInduction}[2]{\mathrm{R}_{#1}^{#2}}
\newcommand{\algebraicGroup}[1]{\boldsymbol{\mathrm{#1}}}
\newcommand{\LusztigSeries}[2]{\mathcal{E}\left(#1, (#2)\right)}
\newcommand{\FrobeniusFixedPoints}[2][\Frobenius]{\algebraicGroup{#2}^{#1}}
\newcommand{\CharacterLattice}[1]{X^{\ast}\left(#1\right)}
\newcommand{\CocharacterLattice}[1]{X_{\ast}\left(#1\right)}
\newcommand{\SymmetricGroup}{\mathfrak{S}}
\newcommand{\Kernel}{\operatorname{Ker}}
\newcommand{\RTTheta}[2]{R_{#1, #2}}
\newcommand{\RTGTheta}[3]{R_{#1}^{#2}(#3)}
\newcommand{\RTThetaVirtual}[3]{\mathcal{R}_{#1}^{#2}(#3)}
\newcommand{\Cocharacter}{\mathrm{cochar}}
\newcommand{\Character}{\mathrm{char}}
\newcommand{\FlagOfSpaces}{\mathcal{F}}
\newcommand{\transfer}[1]{{#1}^{\natural}}
\newcommand{\rank}{\mathrm{rk}}

\hypersetup{pdfauthor={Calvin Yost-Wolff, Elad Zelingher},
	pdfsubject={Number theory, Representation theory},
	pdfkeywords={Gauss sums, Jacobi sums, Exponential sums}}

\title[Doubling method Jacobi sums]{On Jacobi sums arising from the classical doubling method}

\author{Calvin Yost-Wolff}
\address{Department of Mathematics, University of Michigan, 3848 East Hall, 530 Church Street, Ann Arbor, MI 48109-1043 USA}
\email{calvinyw@umich.edu}

\author{Elad Zelingher}
\address{Department of Mathematics, University of Michigan, 1844 East Hall, 530 Church Street, Ann Arbor, MI 48109-1043 USA}
\email{eladz@umich.edu}


\begin{document}

\begin{abstract}
	We define the notion of a non-abelian Jacobi sum $\posDblJacobiSum{\pi}{\chi}$ attached to an irreducible representation $\pi$ of a general linear group or a classical group over a finite field and a character $\chi$ of the multiplicative group of the finite field or its quadratic extension. These sums emerge in the study of the doubling method of Piatetski-Shapiro--Rallis \cite{PiatetskiShapiroRallis1986, GelbartPiatetskiShapiroRallis1987} and Lapid--Rallis \cite{LapidRallis2005}. For general linear groups, we express these non-abelian Jacobi sums in terms of Kondo's non-abelian Gauss sums \cite{Kondo1963}. For classical groups and for characters that are not conjugate-dual, we give an explicit formula for these non-abelian Jacobi sums in terms of Gauss sums attached to the Deligne--Lusztig data of the representation, and we prove that these Jacobi sums are constant on geometric Lusztig series. Our results rely on a multiplicativity result of non-abelian Jacobi sums obtained by Girsch--Zelingher \cite{GirschZelingher2026}.
\end{abstract}
	\dedicatory{\bf To David Soudry, with admiration}
\maketitle

\tableofcontents

\section{Introduction}

Gauss sums are prominent objects in number theory. They are used to prove reciprocity laws, appear in functional equations of $L$-functions, and arise as local factors and other coefficients in representation theory of $p$-adic groups and finite groups of Lie type.

Let $\finiteField$ be a finite field with $q$ elements and let $\fieldCharacter \colon \finiteField \to \multiplicativegroup{\cComplex}$ be a non-trivial additive character. Given a multiplicative character $\chi \colon \multiplicativegroup{\finiteField} \to \multiplicativegroup{\cComplex}$ the normalized Gauss sum corresponding to the data $\left(\chi, \fieldCharacter\right)$ is defined by the formula
$$\GaussSumSingleCharacter{\chi}{\fieldCharacter} = -\frac{1}{\sqrt{q}}\sum_{x \in \multiplicativegroup{\finiteField}} \chi^{-1}\left(x\right) \fieldCharacter\left(x\right).$$

In the 1960s, Kondo \cite{Kondo1963} introduced an interesting exponential sum, which can be thought of as a higher dimensional version of the Gauss sum defined above. Let $\pi$ be an irreducible representation of $\GL_n\left(\finiteField\right)$. Then the sum $$\GaussSum{\pi}{\fieldCharacter} = q^{-\frac{n^2}{2}} \sum_{g \in \GL_n\left(\finiteField\right)} \fieldCharacter\left(\trace g^{-1}\right) \pi\left(g\right)$$
defines an operator in $\Hom_{\GL_n\left(\finiteField\right)}\left(\pi, \pi\right)$. Since $\pi$ is irreducible, by Schur's lemma there exists a scalar $\GaussSumScalar{\pi}{\fieldCharacter} \in \cComplex$ such that $$\GaussSum{\pi}{\fieldCharacter} = \GaussSumScalar{\pi}{\fieldCharacter} \cdot \idmap_{\pi}.$$

In \cite{Kondo1963}, Kondo explicitly computed $\GaussSumScalar{\pi}{\fieldCharacter}$. His result shows that $\GaussSumScalar{\pi}{\fieldCharacter}$ can be expressed as a product of Gauss sums as above associated to the multiplicative characters that correspond to the Deligne--Lusztig parameterization of $\pi$.

In \cite{Zelingher2024}, the second author utilized Kondo's result to explicitly express twisted matrix Kloosterman sums in terms of Hall--Littlewood polynomials evaluated at the eigenvalues of the Frobenius action on the corresponding twisted Kloosterman sheaf. As a result, he resolved a conjecture of \Erdelyi{}--\Toth{} \cite{ErdelyiToth2024}.

It is natural to ask whether the results of Kondo in \cite{Kondo1963} and the results of the second author in \cite{Zelingher2024} mentioned above have generalizations to other groups besides $\GL_n\left(\finiteField\right)$. An attempt to generalize Kondo's result was given by Saito--Shinoda in \cite{SaitoShinoda2000}. Given a classical group $G$, embedded naturally as a subgroup of a general linear group, and an irreducible representation $\pi$ of $G$, they considered the analogous Gauss sum
$$\sum_{g \in G} \pi\left(g\right) \fieldCharacter\left(\trace g\right).$$

As before, by Schur's lemma, this sum is a scalar multiple of $\idmap_{\pi}$. However, the computations of Saito--Shinoda show that this scalar is not well behaved, even for the case $\Sp_{4}\left(\finiteField\right)$. The main issue with this naive Gauss sum is that it is \emph{not constant on Lusztig series}.

An alternative approach for defining an analog of Kondo's Gauss sum comes from the theory of automorphic representations. Not long after Kondo's result was published, Godement and Jacquet \cite{GodementJacquet1972} defined a theory for standard $L$-functions for $\GL_n$. Nowadays it is known that the finite field analog of the Godement--Jacquet local gamma factor corresponding to an irreducible representation $\pi$ of $\GL_n\left(\finiteField\right)$ coincides with Kondo's Gauss sum of $\pi$, see \cite[Section 2]{Macdonald80}.

Since Godement--Jacquet developed their theory for the standard $L$-functions, many other constructions for $L$-functions attached to different representations have appeared in the literature. Unlike the construction of Godement--Jacquet, most of these constructions require the representations in question to satisfy certain genericity conditions. In the 1980s, Piatetski-Shapiro and Rallis \cite{PiatetskiShapiroRallis1986, GelbartPiatetskiShapiroRallis1987} discovered a new construction of (twisted) standard $L$-functions for classical groups. Their construction, known as ``the doubling method'', does not require the representations in question to satisfy any genericity condition. The local theory of the doubling method construction was developed by Lapid--Rallis \cite{LapidRallis2005}. It is closely related to the local construction of Godement--Jacquet, which shows up as certain inner integrals when considering parabolically induced representations of the classical group.

Thus, a promising way to obtain a well behaved exponential sum for classical groups is to consider the finite field analog of the local doubling method construction and to analyze the formula for the gamma factor that this theory yields. This was done for $\Sp_{2n}\left(\finiteField\right)$ by Jun Chang \cite{Chang1997} in her doctoral thesis. In the upcoming work of Girsch and the second author \cite{GirschZelingher2026}, the finite field analog of the doubling method construction is developed for all classical groups.

The computations of \cite{Chang1997} and \cite{GirschZelingher2026} show that, generically speaking, the finite field doubling method gamma factor for classical groups is given by exponential sums that resemble Jacobi sums. Let $G \subset \GL\left(\hermitianSpace\right)$ be the connected component of a classical group, corresponding to the space $\left(\hermitianSpace, \innerproduct{\cdot}{\cdot}\right)$, and let $\lieAlgebra \subset \EndomorphismRing\left(\hermitianSpace\right)$ be the Lie algebra of $G$. Let $\pi$ be an irreducible representation of $G$ and let $\chi \colon Z\left(\GL\left(\hermitianSpace\right)\right) \to \multiplicativegroup{\cComplex}$ be a character that is not conjugate-dual, that is, $\chi^{1+q}$ is not the trivial character. The exponential sum arising from the finite field doubling method gamma factor of $\pi$ is given by the formula
$$\posDblJacobiSum{\pi}{\chi} = \frac{1}{\sizeof{\lieAlgebra}^{1 \slash 2}} \sum_{\substack{g \in G\\
\det\left(\idmap_{\hermitianSpace} + g\right) \ne 0}} \pi\left(g\right) \chi\left(\det\left(\idmap_{\hermitianSpace} + g\right)\right).$$

Notice that if $-\idmap_{\hermitianSpace} \in G$ then  $$\posDblJacobiSum{\pi}{\chi} = \frac{\centralCharacter{\pi}\left(-1\right)}{\sizeof{\lieAlgebra}^{1/2}}\sum_{\substack{g \in G, h \in \GL\left(\hermitianSpace\right)\\
g + h = \idmap_{\hermitianSpace}}} \pi\left(g\right) \chi\left(\det h\right),$$
which resembles a classical Jacobi sum attached to characters $\chi_1, \chi_2 \colon \multiplicativegroup{\finiteField} \to \multiplicativegroup{\cComplex}$:
$$\sum_{\substack{x_1,x_2 \in \multiplicativegroup{\finiteField}\\
		x_1 + x_2 = 1}} \chi_1\left(x_1\right) \chi_2\left(x_2\right).$$

As before, we have that $\posDblJacobiSum{\pi}{\chi}$ lies in $\Hom_G\left(\pi, \pi\right)$. Since $\pi$ is irreducible, by Schur's lemma there exists a complex number $\posDblJacobiSumScalar{\pi}{\chi}$ such that $$\posDblJacobiSum{\pi}{\chi} = \posDblJacobiSumScalar{\pi}{\chi} \cdot \idmap_{\pi}.$$

In this paper, we study the doubling method Jacobi sum $\posDblJacobiSum{\pi}{\chi}$. We compute the scalar $\posDblJacobiSumScalar{\pi}{\chi}$ in terms of the Deligne--Lusztig data corresponding to $\pi$ and show that, up to some well understood factors, this scalar matches the Kondo Gauss sum of the conjectural functorial transfer of $\pi$ to a general linear group. Specifically, we show the following theorem.
\begin{theorem}[\Cref{thm:doubling-method-gamma-factor-for-deligne-lusztig}, \Cref{thm:computation-of-doubling-gauss-sum-scalar-for-deligne-lusztig-characters}]
	For a classical group $G$ and a multiplicative character $\chi$ as above, the kernel function $$\posHermitianJacobiKernel{\hermitianSpace}{\chi}\left(g\right) = \begin{dcases}
		\chi\left(\det\left(\idmap_{\hermitianSpace} + g\right)\right) & \det\left(\idmap_{\hermitianSpace} + g\right) \ne 0,\\
		0 & \text{otherwise}
	\end{dcases}$$ is a stable function. If $\pi$ appears as a constituent of the Deligne--Lusztig induction $\RTThetaVirtual{T}{G}{\theta}$, we have that
	$$\posDblJacobiSumScalar{\pi}{\chi} = \left(-1\right)^{\mathrm{rel.rank} \transfer{T}} q^{-\left\lfloor\frac{\dim \hermitianSpace}{2}\right\rfloor} c_{\hermitianSpace}\left(\chi, \fieldCharacter\right) \sum_{\transfer{t} \in \transfer{T}} \left(\transfer{\theta}\right)^{-1}\left(\transfer{t}\right) \chi^{-1}\left(\det \transfer{t}\right) \fieldCharacter\left(\trace \transfer{t}\right).$$
	Here $c_{\hermitianSpace}\left(\chi, \fieldCharacter\right)$ is an explicit factor and the torus character pair $\left(\transfer{\algebraicGroup{T}}, \transfer{\theta}\right)$ is the functorial transfer of $\left(\algebraicGroup{T}, \theta\right)$ to a corresponding general linear group. See \Cref{subsec:normalization-factor}, \Cref{subsec:some-notation} and \Cref{thm:computation-of-doubling-gauss-sum-scalar-for-deligne-lusztig-characters} for the details.
\end{theorem}
This is the first explicit result for classical groups we are aware of that allows one to extract information about the Deligne--Lusztig data associated to an irreducible representation $\pi$ via a character sum computation. This result is used in \cite{GirschZelingher2026} to give an explicit expression for the doubling method gamma factor of certain depth-zero representations.

Our proof combines a couple of ingredients. The first ingredient is a lemma of Saito--Shinoda \cite{SaitoShinoda2000} that allows us to explicitly compute the analogous character sums for Deligne--Lusztig virtual representations $\RTThetaVirtual{T}{G}{\theta}$.  The second ingredient, which is a key ingredient for this computation, is that these Jacobi sums are well behaved, in the sense that they are constant on Lusztig series. To show this, we utilize a multiplicativity property of these Jacobi sums from \cite{GirschZelingher2026} and combine it with a careful analysis of the Deligne--Lusztig data corresponding to $\pi$ after functorial transfer. 

We only deal with ``generic'' Jacobi sums. These are Jacobi sums that correspond to characters $\chi$ that are not conjugate-dual, i.e., characters such that $\chi^{1+q}$ is not the trivial character. In a future work we plan to study Jacobi sums that correspond to conjugate-dual characters. These Jacobi sums seem to be more involved, as their kernel function does not only depend on the semisimple part of a given conjugacy class. See \cite{Kim2000} for more information for the case $\chi = 1$.

Since our Jacobi sums are constant on Lusztig series, it is natural to ask whether Lusztig series are determined by their associated Jacobi sums, or in other words, to ask whether our Jacobi sums satisfy a certain converse theorem. Using \cite[Appendix B]{nien2021converse} it is possible to deduce a converse theorem, but this requires to have a definition for Jacobi sums for all characters, including conjugate-dual characters. We plan to include this converse theorem in our future work mentioned above.

Let us also mention the Braverman--Kazhdan program which conjectures that every $L$-function should have a construction analogous to the Godement--Jacquet zeta integral \cite{BravermanKazhdan2003}. In the context of finite fields, this conjecture was entirely proved in recent years \cite{ChengNgo2018, Chen2022, LaumonLetellier2023}. However, it seems difficult to make the construction explicit. It would be interesting to show that the construction given by the Braverman--Kazhdan framework unfolds to the Jacobi sums we study in this work.

The paper is structured as follows. \Cref{sec:preliminaries} consists of preliminaries. In \Cref{subsec:classical-groups} we review the definitions of classical groups and their counterpart similitude groups over finite fields. In \Cref{subsec:harish-chandra-series} we review the notions of cuspidal representations and Harish-Chandra series for these groups. In \Cref{subsec:rational-tori} we give an explicit overview of maximal rational tori in these groups. In \Cref{subsec:deligne-lusztig-theory} we review the Deligne--Lusztig theory and important properties regarding geometric Lusztig series for classical groups and their counterpart similitude groups. In \Cref{subsec:gauss-sums} we introduce notation for Gauss sums and review results about non-abelian Gauss sums for general linear groups due to Kondo. We also prove a vanishing result for certain degenerate Kondo Gauss sums (\Cref{lem:sum-vanishes-for-singular-matrices}).

\Cref{sec:doubling-method-jacobi-sums} is the main body of the paper. In \Cref{subsec:definition-of-jacobi-sums} we define the notion of non-abelian Jacobi sums for general linear groups, and for classical groups and their similitude counterparts. For general linear groups we express these non-abelian Jacobi sums in terms of Kondo's non-abelian Gauss sums (\Cref{subsubsec:jacobi-sums-for-general-linear-groups}). In \Cref{subsec:multiplicativity}, we state the multiplicativity property \cite{GirschZelingher2026} these Jacobi sums satisfy. In \Cref{subsubsec:pairing-with-deligne-lusztig} we explicitly compute the analogous non-abelian Jacobi character sums for Deligne--Lusztig characters, formed by taking the inner product of the kernel function $\posHermitianJacobiKernel{\hermitianSpace}{\chi}$ with a Deligne--Lusztig character. In \Cref{subsec:invariance-under-geometric-conjugacy}, we show that the ratio of the expression from \Cref{subsubsec:pairing-with-deligne-lusztig} and the virtual dimension of the Deligne--Lusztig character is independent of the representative of the geometric conjugacy class of the Deligne--Lusztig character. Finally, in \Cref{subsec:stability} we use the multiplicativity property from \Cref{subsec:multiplicativity} and the results from \Cref{subsec:invariance-under-geometric-conjugacy} to prove that our Jacobi sums are constant on geometric Lusztig series, leading to an explicit expression for the Jacobi sums.

The paper contains three appendices. In \Cref{appendix:evaluation-of-some-exponential-sums} we establish two simple identities of exponential sums needed for the computations in \Cref{subsubsec:pairing-with-deligne-lusztig}. In \Cref{appendix:multiplicativity-of-jacobi-sums} we explain how to obtain a multiplicativity result for similitude groups from the result established in \cite{GirschZelingher2026} for classical groups. Finally, in \Cref{appendix:extension-of-identities} we perform analogous computations to the ones in \Cref{subsubsec:pairing-with-deligne-lusztig} for the conjugate-dual character case.

\subsection*{Acknowledgements}
We would like to thank Charlotte Chan for discussions on this project and for her encouragement. C.Y. was supported by the National Science Foundation Graduate
Research Fellowship Program under Grant No. DGE-2241144. Any opinions, findings, and conclusions or recommendations expressed in this material are those of the authors and do not necessarily reflect the views of the National Science Foundation.

\section{Preliminaries}\label{sec:preliminaries}
Let $\finiteField$ be a finite field with $q$ elements, and let $p$ be the characteristic of $\finiteField$. We will always assume that $p \ne 2$. Fix an algebraic closure $\algebraicClosure{\finiteField}$ of $\finiteField$. For any $k \ge 1$, let $\finiteFieldExtension{k} \slash \finiteField$ be the unique field extension of degree $k$ contained in $\algebraicClosure{\finiteField}$. If $k' \mid k$, let $\FieldNorm{k}{k'} \colon \multiplicativegroup{\finiteFieldExtension{k}} \to \multiplicativegroup{\finiteFieldExtension{k'}}$ be the norm map.

Given an integer $m \ge 1$ we define $$\NormOneGroup{2m} = \Kernel \FieldNorm{2m}{m} = \left\{ x \in \multiplicativegroup{\finiteFieldExtension{2m}} \mid \FieldNorm{2m}{m}\left(x\right) = 1\right\}.$$

Given a partition $\lambda = \left(\lambda_1, \dots, \lambda_\ell\right)$, we denote by $\lengthof\left(\lambda\right) = \ell$ the \emph{length} of $\lambda$. Given a partition $\lambda$ as above and a finite field extension $\quadraticExtension \slash \finiteField$, we denote the \etale $\quadraticExtension$-algebra $$\quadraticFieldExtension{\lambda} = \prod_{i=1}^{\ell} \finiteFieldExtension{\grpIndex{\quadraticExtension}{\finiteField} \lambda_i}$$ and define
$$\multiplicativegroup{\quadraticFieldExtension{\lambda}} = \prod_{i=1}^{\ell} \multiplicativegroup{\finiteFieldExtension{\grpIndex{\quadraticExtension}{\finiteField} \lambda_i}}.$$

We also define $2 \lambda = \left(2 \lambda_1, \dots, 2 \lambda_{\ell}\right)$ and $\aFieldNorm_{\finiteFieldExtension{2\lambda} \slash \finiteFieldExtension{\lambda}} \colon \multiplicativegroup{\finiteFieldExtension{2 \lambda}} \to \multiplicativegroup{\finiteFieldExtension{\lambda}}$ by $$\aFieldNorm_{\finiteFieldExtension{2\lambda} \slash \finiteFieldExtension{\lambda}}\left(x_1,\dots,x_{\ell}\right) = \left(\FieldNorm{2\lambda_1}{\lambda_1}\left(x_1\right), \dots, \FieldNorm{2\lambda_{\ell}}{\lambda_{\ell}}\left(x_{\ell}\right)\right).$$
This is the restriction to $\multiplicativegroup{\finiteFieldExtension{2 \lambda}}$ of the norm map given by viewing $\finiteFieldExtension{2\lambda}$ as an algebra over $\finiteFieldExtension{\lambda}$. We define $$\NormOneGroup{2\lambda} = \Kernel \aFieldNorm_{\finiteFieldExtension{2\lambda} \slash \finiteFieldExtension{\lambda}} = \left\{x \in \multiplicativegroup{\finiteFieldExtension{2\lambda}} \mid \aFieldNorm_{\finiteFieldExtension{2\lambda} \slash \finiteFieldExtension{\lambda}}\left(x\right) = 1\right\},$$
where $\multiplicativegroup{\finiteField} \to \multiplicativegroup{\finiteFieldExtension{\lambda}}$ is embedded diagonally. It is clear that $$\NormOneGroup{2\lambda} = \prod_{i=1}^{\ell} \NormOneGroup{2 \lambda_i}.$$

\subsection{Classical groups}\label{subsec:classical-groups}

In this section, we recall the notion of classical groups and their corresponding similitude groups. These are the main groups which we will study in this paper.

\subsubsection{Algebraic groups over finite fields}
Each of our classical groups $G$ in the proceeding sections are defined by polynomial equations in entries of a matrix which yield an \emph{affine algebraic group}, that is, an affine algebraic variety with a group structure (see \cite[Section 0]{DigneMichel1991}). 
We will denote the algebraic group by the bold symbol $\algebraicGroup{G}$. 
It can be thought of as its $\algebraicClosure{\finiteField}$ points together with an action of $\Frobenius$ on these points whose fixed points is $\algebraicGroup{G}(\finiteField)$. 
Thus we will often denote $\algebraicGroup{G}(\finiteField)$ by $\algebraicGroup{G}^{\Frobenius}$ or just by the non-bold symbol $G$.

\subsubsection{$\epsilon$-sesquilinear spaces}
Let $\quadraticExtension \slash \finiteField$ be a field extension of degree $1$ or $2$ contained in $\algebraicClosure{\finiteField}$. For any $k \ge 1$, let $\quadraticFieldExtension{k} \slash \quadraticExtension$ be the unique field extension of degree $k$ contained in $\algebraicClosure{\finiteField}$. Let $x \mapsto \involution{x}$ be the generator of $\Galois\left(\quadraticExtension \slash \finiteField\right)$.

Let $\hermitianSpace$ be a vector space of dimension $n$ over $\quadraticExtension$, equipped with a non-degenerate sesquilinear form $\innerproduct{\cdot}{\cdot} \colon \hermitianSpace \times \hermitianSpace \to \quadraticExtension$ which is $\epsilon_{\hermitianSpace}$-symmetric for $\epsilon_{\hermitianSpace} \in \left\{\pm 1\right\}$. By this we mean that:
\begin{enumerate}
	\item For every $x_1,x_2,y \in \hermitianSpace$, $$\innerproduct{x_1 + x_2}{y} = \innerproduct{x_1}{y} + \innerproduct{x_2}{y}.$$
	\item For every $x,y \in \hermitianSpace$ and $t \in \quadraticExtension$, $$\innerproduct{tx}{y} = t\innerproduct{x}{y}.$$
	\item ($\epsilon_{\hermitianSpace}$-symmetric) For every $x,y \in \hermitianSpace$, $$\innerproduct{x}{y} = \epsilon_{\hermitianSpace} \involution{\innerproduct{y}{x}}.$$
	\item (non-degenerate) For every $0 \ne x \in \hermitianSpace$ there exists $y \in \hermitianSpace$ such that $$\innerproduct{x}{y} \ne 0.$$
\end{enumerate}

A vector $v \in \hermitianSpace$ is called \emph{anisotropic} if $\innerproduct{v}{v} \ne 0$. A subspace $\xIsotropic \subset \hermitianSpace$ is called \emph{totally isotropic} if for every $x,y \in \xIsotropic$ we have $\innerproduct{x}{y} = 0$. If $\xIsotropic, \yIsotropic \subset \hermitianSpace$ are totally isotropic subspaces, we say that $\xIsotropic$ and $\yIsotropic$ are \emph{in duality} (with respect to $\innerproduct{\cdot}{\cdot}$) if the map $\yIsotropic \to \Hom_{\quadraticExtension}\left(\xIsotropic, \quadraticExtension\right)$ given by $$y \mapsto \left(x \mapsto \innerproduct{x}{y}\right)$$ is an (anti-)isomorphism.

\subsubsection{Isometry groups}\label{subsec:isometry-groups}
Let $\IsometryGroup \left(\hermitianSpace\right)$ be the isometry group of $\hermitianSpace$, consisting of all elements $g \in \GL_{\quadraticExtension}\left(\hermitianSpace\right)$ satisfying $\innerproduct{gx}{gy} = \innerproduct{x}{y}$ for every $x,y \in \hermitianSpace$. We denote by $G$ the identity component of $\IsometryGroup\left(\hermitianSpace\right)$. The two above conditions define an algebraic group $\algebraicGroup{G}$ whose $\finiteField$ points are $G$. We explain the possible options for $G$ and the corresponding group $\algebraicGroup{G}$ in the following list.

Let $$w_n = \begin{pmatrix}
	& & & 1\\
	& & 1\\
	& \iddots\\
	1
\end{pmatrix} \in \GL_n\left(\finiteField\right)$$ be the longest Weyl element.
\begin{enumerate}
	\setItemnumber{A}
	\item (Unitary groups): $\quadraticExtension \ne \finiteField$. Denote $\dim_{\quadraticExtension} \hermitianSpace = n$. In this case, the group $G$ is isomorphic to the group $\UnitaryGroup_n\left(\finiteField\right) \subset\GL_{n}\left(\quadraticExtension\right)$ defined as the $\finiteField$-points of $$\algebraicGroup{\UnitaryGroup}_n = \left\{ g \in \restrictionOfScalars{\quadraticExtension}{\finiteField}{\algebraicGroup{\GL}_n} \mid \involution{\left(\transpose{g}\right)} w_n g = w_n \right\}.$$
	Here $\restrictionOfScalars{\quadraticExtension}{\finiteField} {\algebraicGroup{\GL}}_n = \algebraicGroup{\GL}_n \times \algebraicGroup{\GL}_n$, equipped with the involution $\involution{\left(g_1,g_2\right)} = \left(g_2, g_1\right)$ and the Frobenius root $\Frobenius \left(g_1, g_2\right) = \involution{\left(\Frobenius g_1, \Frobenius g_2\right)}$, and we realize $w_n$ embedded diagonally as $\left(w_n, w_n\right)$.
	\setItemnumber{B/D}
	\item (Special orthogonal groups): $\quadraticExtension = \finiteField$ and $\epsilon_{\hermitianSpace} = 1$. It follows by definition that every element $g \in \IsometryGroup\left(\hermitianSpace\right)$ satisfies $\left(\det g\right)^2 = 1$.  In this case, we obtain the connected components of the identity by adding a restriction on the determinant: We have that $G$ consists of all elements $g \in \IsometryGroup\left(\hermitianSpace\right)$ with $\det g = 1$. There are two cases here.
	\begin{enumerate}
		\item Split case: in this case, we can write $$\hermitianSpace = \begin{dcases}
		\xIsotropic \oplus \yIsotropic & \dim_{\finiteField} \hermitianSpace = 2n\\
		\xIsotropic \oplus \finiteField v \oplus \yIsotropic & \dim_{\finiteField}\hermitianSpace = 2n + 1
		\end{dcases}$$ where $\xIsotropic$ and $\yIsotropic$ are totally isotropic in duality. In the odd-dimensional case, $v \in \hermitianSpace$ is an anisotropic vector orthogonal to $\xIsotropic$ and $\yIsotropic$.
		
		In either case, denote $N = \dim_{\finiteField} \hermitianSpace$. We have that $G$ is isomorphic to $\SO^{+}_{N}\left(\finiteField\right) \subset  \GL_{N}\left(\finiteField\right)$, defined as the $\finiteField$-points of $$\algebraicGroup{\SO}^{+}_{N} = \left\{g \in \algebraicGroup{\GL}_{N} \mid \transpose{g} w_{N} g  = w_{N},\,\, \det g = 1\right\}.$$
		
		In the odd-dimensional case $N = 2n+1$, $V$ is always split, and we simply denote $\SO_{2n+1}\left(\finiteField\right) = \SO_{2n+1}^{+}\left(\finiteField\right)$ and $\algebraicGroup{\SO}_{2n+1} = \algebraicGroup{\SO}_{2n+1}^{+}$.
		\item Non-split case: In this case, $\dim_{\finiteField} \hermitianSpace = 2n$ and there exists a decomposition $$\hermitianSpace = \xIsotropic \oplus \hermitianSpace' \oplus \yIsotropic,$$
		where $\xIsotropic$ and $\yIsotropic$ are totally isotropic spaces in duality and $\hermitianSpace'$ is an anisotropic space (that is, for every $0 \ne v' \in \hermitianSpace'$, we have $\innerproduct{v'}{v'} \ne 0$) of dimension 2. Let $d$ be an element of $\multiplicativegroup{\finiteField}$ that is not a square. Denote $$B_{n,d} = \begin{pmatrix}
			\IdentityMatrix{n-1}\\
			& & 1 &\\
			& -d & &\\
			& & & \IdentityMatrix{n-1}
		\end{pmatrix} \cdot w_{2n}.$$
		In this case, we have that $G$ is isomorphic to $\SO^{-}_{2n}\left(\finiteField\right) \subset  \GL_{2n}\left(\finiteField\right)$, defined as the $\finiteField$-points of $$\algebraicGroup{\SO}^{-}_{2n} = \left\{g \in \algebraicGroup{\GL}_{2n} \mid \transpose{g} B_{n,d} g = B_{n,d},\,\, \det g = 1\right\}.$$
	\end{enumerate}
	\setItemnumber{C}
	\item (Symplectic groups): $\quadraticExtension = \finiteField$ and $\epsilon_{\hermitianSpace} = -1$. In this case, $\dim_{\finiteField} \hermitianSpace = 2n$ for some positive integer $n$ and $G$ is isomorphic to the group $\Sp_{2n}\left(\finiteField\right) \subset  \GL_{2n}\left(\finiteField\right)$, defined as the $\finiteField$-points of $$\algebraicGroup{\Sp}_{2n} = \left\{g \in \algebraicGroup{\GL}_{2n} \mid \transpose{g} \begin{pmatrix}
		& w_n\\
		-w_n
	\end{pmatrix} g = \begin{pmatrix}
		& w_n\\
		-w_n
	\end{pmatrix}\right\}.$$
\end{enumerate}

We will denote by $\lieAlgebra$ be the Lie algebra of $G$, consisting of all elements $A \in \EndomorphismRing_{\quadraticExtension}\left(\hermitianSpace\right)$ satisfying $\innerproduct{Ax}{y} + \innerproduct{x}{Ay} = 0$ for every $x, y \in \hermitianSpace$.

\subsubsection{The group $\tilde{G}$ and similitude groups}\label{subsec:similitude-groups}
In this section, we will introduce a group $\tilde{\algebraicGroup{G}}$ containing $\algebraicGroup{G}$ that has connected center. This will be important for later because the Deligne--Lusztig theory is easier  for $\tilde{\algebraicGroup{G}}$.

In the special odd orthogonal group case (that is, $\epsilon_{\hermitianSpace} = 1$, $\quadraticExtension = \finiteField$ and $\dim_{\finiteField} \hermitianSpace$ is odd) and in the unitary case (that is, $\quadraticExtension \ne \finiteField$), we have that $\algebraicGroup{G}$ has connected center and we denote $\tilde{\algebraicGroup{G}} = \algebraicGroup{G}$.

In the remaining cases $\quadraticExtension = \finiteField$, $\dim_{\finiteField} \hermitianSpace$ is even, and $\algebraicGroup{G}$ does not have a connected center. We extend $\algebraicGroup{G}$ to a group $\tilde{\algebraicGroup{G}}$ with connected center by considering the similitude group corresponding to $\hermitianSpace$.

First, let $\GroupExtension{\IsometryGroup\left(\hermitianSpace\right)}$ be the similitude group version of $\IsometryGroup\left(\hermitianSpace\right)$, consisting of the elements $g \in \GL_{\quadraticExtension}\left(\hermitianSpace\right)$ with the property that there exists a constant $\similitudeCharacter\left(g\right) \in \multiplicativegroup{\quadraticExtension}$ such that $\innerproduct{gx}{gy} = \similitudeCharacter\left(g\right) \innerproduct{x}{y}$ for every $x,y \in \hermitianSpace$. Let $\GroupExtension{G}$ be the identity component of $\GroupExtension{\IsometryGroup\left(\hermitianSpace\right)}$. The two above conditions define an algebraic group $\tilde{\algebraicGroup{G}}$ with connected center whose $\finiteField$ points are $\GroupExtension{G}$.

We describe $\tilde{G}$ and $\tilde{\algebraicGroup{G}}$ more explicitly.
\begin{enumerate}
	\setItemnumber{D}
	\item (Even special orthogonal groups): Denote $\dim_{\finiteField} \hermitianSpace = 2n$. Every element $g \in \widetilde{\IsometryGroup\left(\hermitianSpace\right)}$ satisfies $\left(\det g\right)^2 = \similitudeCharacter\left(g\right)^{2n}$. Similarly to before, we obtain the connected components of the identity by adding a restriction on the determinant. The group $\GroupExtension{G}$ consists of all elements $g \in \GroupExtension{\IsometryGroup\left(\hermitianSpace\right)}$ such that $\det g = \similitudeCharacter\left(g\right)^n$. We have the two cases from before here:
	\begin{enumerate}
		\item (Split case): The group $\GroupExtension{G}$ is isomorphic to the group $\GSO^{+}_{2n}\left(\finiteField\right) \subset \GL_{2n}\left(\finiteField\right)$, defined as the projection to the first coordinate of the $\finiteField$-points of $$\algebraicGroup{\GSO}^{+}_{2n} = \left\{\left(g, \lambda\right) \in \algebraicGroup{\GL}_{2n} \times \multiplcativeScheme \mid \transpose{g} w_{2n} g = \lambda w_{2n},\,\, \det g = \lambda^n\right\}.$$
		\item (Non-split case): The group $\GroupExtension{G}$ is isomorphic to the group $\GSO^{-}_{2n}\left(\finiteField\right) \subset \GL_{2n}\left(\finiteField\right)$, defined as the projection to the first coordinate of the $\finiteField$-points of $$\algebraicGroup{\GSO}^{-}_{2n} = \left\{\left(g, \lambda\right) \in \algebraicGroup{\GL}_{2n} \times \multiplcativeScheme \mid \transpose{g} B_{n,d} g = \lambda B_{n,d},\,\, \det g = \lambda^n\right\}.$$
	\end{enumerate}
	\setItemnumber{C}
	\item (Symplectic groups): The group $\GroupExtension{G}$ is isomorphic to the group $\GSp_{2n}\left(\finiteField\right) \subset \GL_{2n}\left(\finiteField\right)$, defined as the projection to the first coordinate of the $\finiteField$-points of $$\algebraicGroup{\GSp}_{2n} = \left\{\left(g, \lambda\right) \in \algebraicGroup{\GL}_{2n} \times \multiplcativeScheme \mid \transpose{g} \begin{pmatrix}
		& w_n\\
		-w_n
	\end{pmatrix} g = \lambda \begin{pmatrix}
		& w_n\\
		-w_n
	\end{pmatrix}\right\}.$$
\end{enumerate}

\subsubsection{Parabolic subgroups}\label{subsec:parabolic-subgroups}
Given a finite dimensional vector space $\hermitianSpace$ over $\quadraticExtension$, a \emph{flag} $\FlagOfSpaces$ in $\hermitianSpace$ of \emph{length} $M \le \dim_{\quadraticExtension}\hermitianSpace$ is a sequence of vector spaces
$$\FlagOfSpaces \colon \left\{0\right\} = \hermitianSpace_0 \subsetneq \hermitianSpace_1 \subsetneq \hermitianSpace_2 \subsetneq \dots \subsetneq \hermitianSpace_M = \hermitianSpace.$$
The group $\GL_{\quadraticExtension}\left(\hermitianSpace\right)$ acts on the set of flags of $\hermitianSpace$ by $$g \FlagOfSpaces \colon \left\{0\right\} = g \hermitianSpace_0 \subsetneq g V_1 \subsetneq g V_2 \subsetneq \dots \subsetneq g \hermitianSpace_M = \hermitianSpace,$$
where $g \in \GL_{\quadraticExtension}\left(\hermitianSpace\right)$.

The \emph{parabolic subgroup} of $\GL_{\quadraticExtension}\left(\hermitianSpace\right)$ stabilizing $\FlagOfSpaces$ is the subgroup $P = P_{\GL_{\quadraticExtension}\left(\hermitianSpace\right)}\left(\FlagOfSpaces\right)$ consisting of all elements $p$ satisfying $p \FlagOfSpaces = \FlagOfSpaces$. It admits a Levi decomposition $P = L \ltimes N$ as follows. The subgroup $N = N_{\GL_{\quadraticExtension}\left(\hermitianSpace\right)}\left(\FlagOfSpaces\right)$ is the unipotent radical of $P$, consisting of elements $n \in P$ satisfying that $n$ acts as identity on the quotient space $\hermitianSpace_i \slash \hermitianSpace_{i-1}$ for every $1 \le i \le M$. Fix a decomposition $\hermitianSpace_i = \hermitianSpace'_1 \oplus \dots \oplus \hermitianSpace'_i$ for every $1 \le i \le M$. The subgroup $L = L_{\GL_{\quadraticExtension}\left(\hermitianSpace\right)}\left(\FlagOfSpaces\right)$ is a \emph{Levi subgroup} consisting of elements $l$ satisfying $l \hermitianSpace'_i = \hermitianSpace'_i$ for every $i$. We have that $L$ is isomorphic to the product $\prod_{i=1}^M \GL_{\quadraticExtension}\left(\hermitianSpace'_i\right)$.

If $\hermitianSpace = \quadraticExtension^k$ and $k_1 + \dots + k_r = k$ where $k_1, \dots, k_r > 0$, we denote by $P_{(k_1,\dots,k_r)}$ the parabolic subgroup of $\GL_{k}\left(\quadraticExtension\right)$ stabilizing the flag
$$0 \subset \quadraticExtension^{k_1} \subsetneq \quadraticExtension^{k_1 + k_2} \subsetneq \dots \subsetneq \quadraticExtension^{k_1 + \dots + k_r} = \quadraticExtension^k,$$
and call $P_{(k_1,\dots,k_r)}$ the \emph{standard parabolic subgroup} corresponding to the composition $\left(k_1,\dots,k_r\right)$, where we realize $\quadraticExtension^j$ as a subspace of $\quadraticExtension^k$ for $j \le k$ by $\quadraticExtension^j \times \left\{0\right\}^{k-j}$. The standard Levi subgroup corresponding to $P_{(k_1,\dots,k_r)}$ is given by $$L_{(k_1,\dots,k_r)} = \left\{\diag\left(g_1,\dots,g_r\right) \mid g_i \in \GL_{k_i}\left(\quadraticExtension\right) \right\} \cong \prod_{i=1}^r \GL_{k_i}\left(\quadraticExtension\right).$$

Given a non-degenerate $\epsilon$-sesquilinear hermitian space $\left(\hermitianSpace, \innerproduct{\cdot}{\cdot}\right)$, and a flag $$\FlagOfSpaces\colon \left\{0\right\} = \xIsotropic_0 \subsetneq \xIsotropic_1 \subsetneq \xIsotropic_2 \subsetneq \dots \subsetneq \xIsotropic_M \subsetneq \hermitianSpace,$$
where $\xIsotropic_1,\dots,\xIsotropic_M \subset \hermitianSpace$ are totally isotropic subspaces, the parabolic subgroup of $H \in \{G, \tilde{G}\}$ (see \Cref{subsec:isometry-groups} and \Cref{subsec:similitude-groups}) stabilizing $\FlagOfSpaces$ is $P_{H}\left(\FlagOfSpaces\right) = P_{\GL_{\quadraticExtension}\left(\hermitianSpace\right)}\left(\FlagOfSpaces\right) \cap H$. It admits a Levi decomposition $P_{H}\left(\FlagOfSpaces\right) = L_{H}\left(\FlagOfSpaces\right) \ltimes N_{H}\left(\FlagOfSpaces\right)$ as follows. The subgroup $\UnipotentRadical_{H}\left(\FlagOfSpaces\right) = N_{\GL_{\quadraticExtension}\left(\hermitianSpace\right)}\left(\FlagOfSpaces\right) \cap H$ is the \emph{unipotent radical} of $P_{H}\left(\FlagOfSpaces\right)$. Choose a decomposition $\hermitianSpace = \xIsotropic_M \oplus \hermitianSpace' \oplus \yIsotropic_M$ where $\yIsotropic_M$ is a totally isotropic subspace such that $\xIsotropic_M$ and $\yIsotropic_M$ are in duality, and such that $\hermitianSpace' \subset \hermitianSpace$ is a non-degenerate subspace orthogonal to $\xIsotropic_M$ and $\yIsotropic_M$. Choose a decomposition such that $\xIsotropic_i = \xIsotropic'_1 \oplus \dots \oplus \xIsotropic'_i$ for every $1 \le i \le M$ and let $\yIsotropic'_i$ be the subspace of $\yIsotropic_M$ dual to $\xIsotropic'_i$ such that $\yIsotropic'_i$ is orthogonal to $\xIsotropic'_j$ for every $j \ne i$. Then $L_{H}\left(\FlagOfSpaces\right)$ is the subgroup of $P_{H}\left(\FlagOfSpaces\right)$ consisting of elements $l$ such that $l \xIsotropic'_i = \xIsotropic'_i$ and $l \yIsotropic'_i = \yIsotropic'_i$ for every $i$ and such that $\l \hermitianSpace' = \hermitianSpace'$. We have that the Levi part $L_{H}\left(\FlagOfSpaces\right)$ is isomorphic to the product $\left(\prod_{i=1}^M \GL_{\quadraticExtension}\left(\xIsotropic'_i\right)\right) \times H'$ where $H' \in \{G', \widetilde{G'}\}$ is of the same type as $G$, where $G'$ and $\widetilde{G'}$ are the groups from Sections \ref{subsec:isometry-groups} and \ref{subsec:similitude-groups} corresponding to $\hermitianSpace'$.

\subsection{Harish-Chandra series}\label{subsec:harish-chandra-series}
Let $\hermitianSpace$ be a finite-dimensional vector space over $\quadraticExtension$. Given a finite-dimensional representation $\tau$ of $\GL_{\quadraticExtension}\left(\hermitianSpace\right)$, and a unipotent radical $\left\{\idmap_{\hermitianSpace}\right\} \ne N \subset \GL_{\quadraticExtension}\left(\hermitianSpace\right)$ we say that $0 \ne v \in \tau$ is an \emph{$N$-Jacquet vector} if $\tau\left(n\right)v = v$ for every $n \in N$. We say that $\tau$ is \emph{cuspidal} if $\tau$ does not admit Jacquet vectors for $N \ne \left\{\idmap_{\hermitianSpace}\right\}$.

Let $$\FlagOfSpaces \colon \left\{0\right\} = \hermitianSpace_0 \subsetneq \hermitianSpace_1 \subsetneq \dots \subsetneq \hermitianSpace_M = \hermitianSpace$$ be a flag and $\hermitianSpace_i = \hermitianSpace'_1 \oplus \dots \oplus \hermitianSpace'_i$ be a decomposition as before. Any irreducible representation $\sigma$ of the Levi subgroup $L_{\GL_{\quadraticExtension}\left(\hermitianSpace\right)}\left(\FlagOfSpaces\right) \cong \prod_{i=1}^M \GL_{\quadraticExtension}\left(\hermitianSpace'_i\right)$ is given by $\sigma = \sigma_1 \otimes \dots \otimes \sigma_M$, where $\sigma_i$ is an irreducible representation of $\GL_{\quadraticExtension}\left(\hermitianSpace'_i\right)$ for every $i$. We say that $\sigma$ is \emph{cuspidal} if $\sigma_i$ is cuspidal for every $i$. We inflate $\sigma$ to an irreducible representation of $P_{\GL_{\quadraticExtension}\left(\hermitianSpace\right)}\left(\FlagOfSpaces\right)$, which we denote $\inf \sigma = \sigma_1 \overline{\otimes} \dots \overline{\otimes} \sigma_M$, as follows. The representation $\inf \sigma$ acts on the space of $\sigma$ and its action is given by $\inf \sigma\left(ln\right) = \sigma\left(l\right)$ for every $l \in L_{\GL_{\quadraticExtension}\left(\hermitianSpace\right)}\left(\FlagOfSpaces\right)$ and $n \in N_{\GL_{\quadraticExtension}\left(\hermitianSpace\right)}\left(\FlagOfSpaces\right)$. We define a \emph{parabolically induced representation} by the formula $\Ind{P}{\GL_{\quadraticExtension}\left(\hermitianSpace\right)}{\inf \sigma}$. A fundamental result in the representation theory of general linear groups is that any irreducible representation $\tau$ of $\GL_{\quadraticExtension}\left(\hermitianSpace\right)$ can be realized as a subrepresentation of a parabolically induced representation $\Ind{P}{\GL_{\quadraticExtension}\left(\hermitianSpace\right)}{\inf \sigma}$ for some parabolic subgroup $P$ and some cuspidal irreducible representation $\sigma$ of $L$. Moreover, the Levi subgroup $L$ and the cuspidal irreducible representation $\sigma$ are determined by $\tau$ up to isomorphism. We say that $\tau$ belongs the \emph{Harish--Chandra series} $\left(P, \sigma\right)$. Notice that if $\tau$ is cuspidal, we have that the Harish--Chandra series $\left(\GL_{\quadraticExtension}\left(\hermitianSpace\right), \tau\right)$ equals $\{\tau\}$.

Now let $\left(\hermitianSpace, \innerproduct{\cdot}{\cdot}\right)$ be a non-degenerate $\epsilon$-sesquilinear hermitian space and let $H \in \{G, \tilde{G}\}$ (where we use the notations of \Cref{subsec:isometry-groups} and \Cref{subsec:similitude-groups}). Similarly to above, given a finite-dimensional representation $\pi$ of $H$ and a unipotent radical $N \subset H$, we say that $0 \ne v \in \pi$ is an $N$-Jacquet vector if $\pi\left(n\right)v = v$ for every $n \in N$. We say that $\pi$ is \emph{cuspidal} if $\pi$ does not admit $N$-Jacquet vectors for every unipotent radical $\left\{\idmap_{\hermitianSpace}\right\} \ne N \subset H$.

Let $$\FlagOfSpaces\colon 0 = \xIsotropic_0 \subsetneq \xIsotropic_1 \subsetneq \dots \subsetneq \xIsotropic_M \subsetneq \hermitianSpace,$$
be a flag where $\xIsotropic_i$ is a totally isotropic subspace for every $i$. As at the end of \Cref{subsec:parabolic-subgroups}, choose decompositions $\xIsotropic_i = \xIsotropic'_1 \oplus \dots \oplus \xIsotropic'_i$ and $\hermitianSpace = \xIsotropic_M \oplus \hermitianSpace' \oplus \yIsotropic_M$. Any irreducible representation $\sigma$ of the Levi group $L_{H}\left(\FlagOfSpaces\right) \cong \left(\prod_{i=1}^M \GL_{\quadraticExtension}\left(\xIsotropic'_i\right)\right) \times H'$ is of the form $\sigma = \tau_1 \otimes \dots \otimes \tau_M \otimes \pi'$, where $\tau_i$ is an irreducible representation of $\GL_{\quadraticExtension}\left(\xIsotropic'_i\right)$ for every $1 \le i \le M$ and $\pi'$ is an irreducible representation of $H'$. We say that $\sigma$ is \emph{cuspidal} if $\tau_i$ is cuspidal for all $i$ and if $\pi'$ is cuspidal. As above, we inflate $\sigma$ to an irreducible representation of $P_{H}\left(\FlagOfSpaces\right)$ denoted $\inf \sigma = \sigma_1 \overline{\otimes} \dots \overline{\otimes} \sigma_M \overline{\otimes} \pi'$ by letting $\inf \sigma$ act on the space of $\sigma$ via the action $\inf \sigma\left(l n\right) = \sigma\left(l\right)$ for every $l \in L_H\left(\FlagOfSpaces\right)$ and $n \in N_H\left(\FlagOfSpaces\right)$. Similarly to above, we have that for every irreducible representation $\pi$ of $H$, there exists a parabolic subgroup $P = P_H\left(\FlagOfSpaces\right)$ with Levi part $L_H\left(\FlagOfSpaces\right)$ and a cuspidal irreducible representation $\sigma$ of $L_H\left(\FlagOfSpaces\right)$ such that $\pi$ can be realized as a subrepresentation of the parabolically induced representation $\Ind{P_H\left(\FlagOfSpaces\right)}{H}{\inf \sigma}$. As before, the Levi subgroup $L_H\left(\FlagOfSpaces\right)$ and the cuspidal irreducible representation $\sigma$ of $L_H\left(\FlagOfSpaces\right)$ are determined (up to isomorphism) by $\pi$. We say that $\pi$ belongs to the \emph{Harish-Chandra series} $\left(P_H\left(\FlagOfSpaces\right), \sigma\right)$. Notice again that if $\pi$ is cuspidal then the Harish-Chandra series $\left(H, \pi\right)$ simply equals $\{\pi\}$.

\subsection{Rational Tori}\label{subsec:rational-tori}
A \emph{rational torus} $\algebraicGroup{T} \subset \algebraicGroup{G}$ over $\finiteField$ is a torus defined over $\finiteField$. 

Recall that if $K$ is a group equipped with a Frobenius action $\Frobenius$, we say that $x,y \in K$ are \emph{Frobenius twisted conjugate} if there exists $k \in K$ such that $y = k x \left(\Frobenius k\right)^{-1}$.

Suppose that $\algebraicGroup{G}$ is a connected reductive group. Over a finite field, $\algebraicGroup{G}\left(\finiteField\right)$-conjugacy classes of maximal rank rational tori are parameterized by Frobenius twisted conjugacy classes in the Weyl group. In particular if we fix a rational torus $\algebraicGroup{T}_0$ with rank $\mathrm{rk} \algebraicGroup{G}$ and an element $g \in \algebraicGroup{G}(\algebraicClosure{\finiteField})$ such that the image of the Lang map $L(g):\algebraicGroup{G}\left(\algebraicClosure{\finiteField}\right) \to \algebraicGroup{G}\left(\algebraicClosure{\finiteField}\right)$, $L\left(g\right) = g^{-1}\Frobenius(g)$ lies in the normalizer $\mathrm{N}_{\algebraicGroup{G}}(\algebraicGroup{T}_0)$, then the torus $g\algebraicGroup{T}_0g^{-1}$ is rational, and its $\algebraicGroup{G}(\finiteField)$-conjugacy class is determined  by the Frobenius-twisted conjugacy class of $L\left(g\right) \in W(\algebraicGroup{T}_0)(\algebraicClosure{\finiteField})$, see \cite[Section 3.3]{Carter1985} or \cite[Proposition 4.2.22]{DigneMichel2020}.

Let us denote $$\left(\quadraticExtension, x \mapsto \involution{x}\right) = \begin{dcases}
	\left(\finiteFieldExtension{2}, x \mapsto x^q\right)  & \algebraicGroup{G} \text{ is a unitary group},\\
	\left(\finiteField, x \mapsto x\right) & \text{otherwise.}
\end{dcases}$$

Let $\algebraicGroup{A}_n \subset \algebraicGroup{\GL}_n$ be the diagonal torus. We denote for an element $g \in \algebraicGroup{\GL}_n$, $g^{\ast} = w_n \transpose{g}^{-1} w_n$, where $w_n$ is the long Weyl element (see \Cref{subsec:isometry-groups}). Using $\algebraicGroup{A}_n$ we define standard (mostly split) tori in our groups as follows.
\begin{center}
\begin{tabular}{|c|c|}
	\hline 
		$\algebraicGroup{G}$ & $\algebraicGroup{T}_0$ \tabularnewline
		\hline 
		\hline 
		$\algebraicGroup{\UnitaryGroup}_{2n}$ & $\left\{ \begin{pmatrix}
			a\\
			& \involution{\left(a^{\ast}\right)}
		\end{pmatrix}  \mid a \in \restrictionOfScalars{\quadraticExtension}{\finiteField}{\algebraicGroup{A}_n} \right\}$ \tabularnewline
		\hline 
		$\algebraicGroup{\UnitaryGroup}_{2n+1}$ & $\left\{ \begin{pmatrix}
	a \\
	& t\\
	& & \involution{\left(a^{\ast}\right)}
\end{pmatrix}  \mid a \in \restrictionOfScalars{\quadraticExtension}{\finiteField}{\algebraicGroup{A}_n},\,\, t \in \algebraicGroup{U}_1 \right\}$ \tabularnewline
\hline 
	$\algebraicGroup{\SO}_{2n+1}$ & $\left\{ \begin{pmatrix}
		a\\
		& 1\\
		& & a^{\ast}
	\end{pmatrix} \mid a \in \algebraicGroup{A}_n \right\}$ \tabularnewline
	\hline 
	$\algebraicGroup{\Sp}_{2n}$ or $\algebraicGroup{\SO}^{+}_{2n}$ & $\left\{ \begin{pmatrix}
		a\\
		& a^{\ast}
	\end{pmatrix} \mid a \in \algebraicGroup{A}_n \right\}$ \tabularnewline
	\hline 
	$\algebraicGroup{\SO}_{2n}^{-}$ & $\left\{ \begin{pmatrix}
		a\\
		& h\\
		& & a^{\ast}
	\end{pmatrix} \mid a \in \algebraicGroup{A}_{n-1}\,\, h \in \algebraicGroup{\SO}_2^{-} \right\}$ \tabularnewline
	\hline		
	\end{tabular}
\end{center}
Here, $\algebraicGroup{T}_{0}$ is split for all rows except for the last row. We have that $\mathrm{rel.rank} \algebraicGroup{G} = \mathrm{rel.rank} \algebraicGroup{T}_0$, which equals $n$ for all cases, except for $\algebraicGroup{G} = \algebraicGroup{\SO}_{2n}^{-}$ for which $\mathrm{rel.rank} \algebraicGroup{\SO}_{2n}^{-} = n-1$.

We also write down the standard tori for the relevant corresponding similitude groups.
\begin{center}
	\begin{tabular}{|c|c|}
		\hline 
		$\algebraicGroup{G}$ & $\algebraicGroup{T}_0$ \tabularnewline
		\hline 
		\hline 
		$\algebraicGroup{\GSp}_{2n}$ or $\algebraicGroup{\GSO}^{+}_{2n}$ & $\left\{\left(\begin{pmatrix}
			\lambda a\\
			& a^{\ast}
		\end{pmatrix}, \lambda\right) \mid a \in \algebraicGroup{A}_n,\,\, \lambda \in \multiplcativeScheme \right\}$ \tabularnewline
		\hline 
		$\algebraicGroup{\GSO}_{2n}^{-}$ & $\left\{ \left(\begin{pmatrix}
			\algebraicGroup{\aFieldNorm}_{\finiteFieldExtension{2} \slash \finiteField}\left(h\right) a\\
			& h\\
			& & a^{\ast}
		\end{pmatrix}, \algebraicGroup{\aFieldNorm}_{\finiteFieldExtension{2} \slash \finiteField}\left(h\right)\right) \mid a \in \algebraicGroup{A}_{n-1}\,\, h \in \restrictionOfScalars{\finiteFieldExtension{2}}{\finiteField}{\multiplcativeScheme} \right\}$ \tabularnewline
		\hline		
	\end{tabular}
\end{center}
Here we realize $\restrictionOfScalars{\finiteFieldExtension{2}}{\finiteField}{\multiplcativeScheme}$ as $$\restrictionOfScalars{\finiteFieldExtension{2}}{\finiteField}{\multiplcativeScheme} = \left\{ \left(\begin{pmatrix}
	a & b\\
	db & a
\end{pmatrix}, \lambda\right) \mid a^2 - db^2 = \lambda \in \multiplcativeScheme \right\},$$
where $d$ is as in \Cref{subsec:isometry-groups}. We have that $\mathrm{rel.rank} \algebraicGroup{G} = \mathrm{rel.rank} \algebraicGroup{T}_0$, which equals $n+1$ for $\algebraicGroup{G} = \algebraicGroup{\GSp}_{2n}, \algebraicGroup{\GSO}_{2n}^+$ and equals $n$ for $\algebraicGroup{G} = \algebraicGroup{\GSO}_{2n}^{-}$.

In the unitary group case, the action of $\Frobenius$ on the Weyl group $W\left(\algebraicGroup{T}_{0}\right)$ is conjugation by $w_n$. In the $\algebraicGroup{\SO}^-_{2n}$ and $\algebraicGroup{\GSO}^-_{2n}$ cases, the action of $\Frobenius$ on the Weyl group $W\left(\algebraicGroup{T}_{0}\right)$ is the outer automorphism on the Dynkin diagram of type $D_n$. In our matrix representation, $\Frobenius$ acts on the Weyl group by conjugation by the matrix $$\left(-n,n\right) \coloneq \begin{pmatrix}\IdentityMatrix{n-1} & 0 & 0 & 0\\
	0 & 1 & 0 & 0\\
	0 & 0 & -1 & 0\\
	0 & 0 & 0 & \IdentityMatrix{n-1}
\end{pmatrix}.$$ In all other cases, the action of $\Frobenius$ on $W(\algebraicGroup{T}_0)$ is trivial. We now parameterize the Weyl group $\Frobenius$-twisted conjugacy classes and their corresponding tori.

\begin{enumerate}
	\setItemnumber{(A)}
    \item (Unitary groups) The Weyl group $W\left(\algebraicGroup{T}_{0}\right)$ is isomorphic to the symmetric group $\SymmetricGroup_n$. If $\dot{w} \in \algebraicGroup{\UnitaryGroup}_n$ is a lift of $w \in W\left(\algebraicGroup{T}_{0}\right)$, the map $w \mapsto \dot{w} \cdot w_n$ sends a $w_n$-twisted $\SymmetricGroup_n$ conjugacy class to a $\SymmetricGroup_n$ conjugacy class. Thus twisted conjugacy classes are determined by the cycle partition $\lambda \vdash n$ associated to $\dot{w} \cdot w_n$. Let us denote by $\lambda^+$ the partition formed by taking the even parts of $\lambda$ and by $\lambda^{-}$ the partition formed by taking the odd parts of $\lambda$. The corresponding algebraic torus is
	$$^{L^{-1}w}\algebraicGroup{T}_0 \cong \restrictionOfScalars{\finiteFieldExtension{\lambda^{+}}}{\finiteField}{\multiplcativeScheme} \times \restrictionOfScalars{\finiteFieldExtension{\lambda^{-}}}{\finiteField}{\algebraicGroup{\UnitaryGroup}_1}.$$
	See also \cite[Section 2.3]{ThiemVinroot2009} and \cite[Section 2.2]{SaitoShinoda2000}.
	\setItemnumber{(B/C)}
    \item (Odd special orthogonal and symplectic groups)
    For $\algebraicGroup{\Sp}_{2n}$ or $\algebraicGroup{\SO}_{2n+1}$, the Weyl group is isomorphic to $$\left\{\pm 1\right\}^n \rtimes \SymmetricGroup_n,$$ which we can identify with the group of signed permutations. The conjugacy classes of the Weyl group are classified by pairs of partitions $(\lambda^+,\lambda^-)$ with $\sizeof{\lambda^+} + \sizeof{\lambda^-} = n$. See \cite[Section 2.2]{KonvalinkaMatjavPfeiffer2011} for more details. Let $(\lambda^+,\lambda^-)$ correspond to the conjugacy class of $w$ in the Weyl group. Then
	$$^{L^{-1}w}\algebraicGroup{T}_0 \cong \restrictionOfScalars{\finiteFieldExtension{\lambda^{+}}}{\finiteField}{\multiplcativeScheme} \times \restrictionOfScalars{\finiteFieldExtension{\lambda^{-}}}{\finiteField}{\algebraicGroup{\UnitaryGroup}_1}.$$
		See also \cite[Section 3.2 Part (B)]{Zalesski2018}.
		\setItemnumber{(D)}
        \item (Even special orthogonal groups) The Weyl group of $\algebraicGroup{SO}_{2n}^{\pm}$ is isomorphic to $$\left\{ \left(x_1,\dots,x_n\right) \in \left\{\pm 1\right\}^n \mid \prod_{i=1}^n x_i = 1\right\} \rtimes \SymmetricGroup_n,$$ which can be realized the subgroup of the group of signed permutations, consisting of all even signed permutations. Using this realization, we may attach to $w$ in the Weyl group a pair of partitions $\left(\lambda^+, \lambda^-\right)$ as before, with $\sizeof{\lambda^+} + \sizeof{\lambda^-} = n$. If $\lambda^{-}$ is not the empty partition, the pair $\left(\lambda^+, \lambda^-\right)$ determines the conjugacy class of $w$. Otherwise, there are two different conjugacy classes of the Weyl group corresponding to the same pair $\left(\lambda^{+}, \lambda^{-}\right)$. These two conjugacy classes correspond to isomorphic tori, and these tori are not conjugate in $\algebraicGroup{\SO}_{2n}^{\pm}$. See \cite[Section 2.3]{KonvalinkaMatjavPfeiffer2011} for more details.
        \begin{enumerate}
        	\item (Split case): For $\algebraicGroup{\SO}_{2n}^+$ we have that
        	$$^{L^{-1}w}\algebraicGroup{T}_0 \cong \restrictionOfScalars{\finiteFieldExtension{\lambda^{+}}}{\finiteField}{\multiplcativeScheme} \times \restrictionOfScalars{\finiteFieldExtension{\lambda^{-}}}{\finiteField}{\algebraicGroup{\UnitaryGroup}_1}.$$
        	See also \cite[Section 3.2 Parts (C) and (D)]{Zalesski2018}.        	
        	\item (Non-split case): As in the unitary group case, for $\algebraicGroup{SO}_{2n}^-$, the map $w \mapsto \dot{w}\circ (-n,n)$ sends a $(-n,n)$-twisted conjugacy class in the even signed permutation group to a conjugacy class in the signed permutation group. Let $(\lambda^+,\lambda^-)$ correspond to the signed permutation group conjugacy class of $\dot{w} \circ (-n,n)$. Then
        	$$^{L^{-1}w}\algebraicGroup{T}_0 \cong \restrictionOfScalars{\finiteFieldExtension{\lambda^{+}}}{\finiteField}{\multiplcativeScheme} \times \restrictionOfScalars{\finiteFieldExtension{\lambda^{-}}}{\finiteField}{\algebraicGroup{\UnitaryGroup}_1}.$$
			See also \cite[Section 3.2 Part (D)]{Zalesski2018}. Notice that we differ from \cite{Zalesski2018} since we do not incorporate the element $\left(-n, n\right)$ in the Frobenius root.
        \end{enumerate}
    \end{enumerate}
	The following theorem summarizes these cases.
	
	\begin{theorem}
		A maximal rational torus $\algebraicGroup{T}$ of $\algebraicGroup{G} = \algebraicGroup{\UnitaryGroup}_{n}, \algebraicGroup{\Sp}_{2n}, \algebraicGroup{\SO}_{2n + 1}, \algebraicGroup{\SO}_{2n}^{\pm}$ is isomorphic to 	$$\algebraicGroup{T} \cong \restrictionOfScalars{\finiteFieldExtension{\lambda^{+}}}{\finiteField}{\multiplcativeScheme} \times \restrictionOfScalars{\finiteFieldExtension{\lambda^{-}}}{\finiteField}{\algebraicGroup{\UnitaryGroup}_1},$$
		where $\left(\lambda^+, \lambda^-\right)$ is a pair of partitions such that $\sizeof{\lambda^+} + \sizeof{\lambda^-} = n$, where in the unitary group case we require that $\lambda^+$ consists only of even parts and that $\lambda^-$ consists only of odd parts. We have that $$\algebraicGroup{T}^{\Frobenius} \cong \multiplicativegroup{\finiteFieldExtension{\lambda^+}} \times \NormOneGroup{2 \lambda^{-}},$$ and that $\mathrm{rel.rank} \algebraicGroup{T} = \lengthof\left(\lambda^+\right)$. The embedding $\algebraicGroup{T}^{\Frobenius} \to \algebraicGroup{G}^{\Frobenius}$ is conjugate to an embedding of the form $$\left(x, y\right) \mapsto
		\begin{pmatrix}
			\iota_+\left(x\right)\\
			& \iota_{-}\left(y\right) &\\
			 & & \involution{\iota_+^{\ast}\left(x\right)}
		\end{pmatrix},$$
		where $\iota_+ \colon \multiplicativegroup{\finiteFieldExtension{\lambda^+}} \to \GL_{\frac{\sizeof{\lambda^+}}{\grpIndex{\quadraticExtension}{\finiteField}}}\left(\quadraticExtension\right)$ and $\iota_- \colon \NormOneGroup{2 \lambda^{-}} \to G'$ are embeddings. Here, $G'$ is a group of the same type as $\algebraicGroup{G}^{\Frobenius}$, with $n$ replaced by $\sizeof{\lambda^{-}}$.
		Moreover, in the odd special orthogonal case, the embedding $\iota_{-}$ can be chosen to fix the middle standard basis vector $e_{n+1}$.
	\end{theorem}

    In addition, we write the tori classification for the similitude groups we consider.
	\begin{itemize}
		\item (split) For $\algebraicGroup{\GSp}_{2n}$ or $\algebraicGroup{\GSO}_{2n}^+$, suppose that $(\lambda^+,\lambda^-)$ corresponds to the conjugacy class of $w$ as an element of the group of signed permutations. Then
			\begin{align*}
				^{L^{-1}w}\algebraicGroup{T}_0 \cong& \left\{\left(x,y,z\right) \in  \restrictionOfScalars{\finiteFieldExtension{\lambda^+}}{\finiteField}{\multiplcativeScheme} \times \restrictionOfScalars{\finiteFieldExtension{2\lambda^-}}{\finiteField}{\multiplcativeScheme}  \times \multiplcativeScheme \mid \FieldNorm{2\lambda^-}{\lambda^-}(y) = z \right\}.
			\end{align*}

			\item (non-split) For $\algebraicGroup{\GSO}_{2n}^-$, let $(\lambda^+,\lambda^-)$ correspond to the conjugacy class of $w\circ (-n,n)$ viewed as an element of the group of signed permutations. Then
			\begin{align*}
				^{L^{-1}w}\algebraicGroup{T}_0 \cong& \left\{\left(x,y,z\right) \in  \restrictionOfScalars{\finiteFieldExtension{\lambda^+}}{\finiteField}{\multiplcativeScheme} \times \restrictionOfScalars{\finiteFieldExtension{2\lambda^-}}{\finiteField}{\multiplcativeScheme}  \times \multiplcativeScheme \mid \FieldNorm{2\lambda^-}{\lambda^-}(y) = z \right\}.
			\end{align*}
	\end{itemize}

	The following theorem summarizes these cases.
	\begin{theorem}
	A maximal rational torus $\algebraicGroup{T}$ of $\algebraicGroup{G} =  \algebraicGroup{\GSp}_{2n}, \algebraicGroup{\GSO}_{2n}^{\pm}$ is isomorphic to \begin{align*}
				^{L^{-1}w}\algebraicGroup{T}_0 \cong& \left\{\left(x,y,z\right) \in  \restrictionOfScalars{\finiteFieldExtension{\lambda^+}}{\finiteField}{\multiplcativeScheme} \times \restrictionOfScalars{\finiteFieldExtension{2\lambda^-}}{\finiteField}{\multiplcativeScheme}  \times \multiplcativeScheme \mid \FieldNorm{2\lambda^-}{\lambda^-}(y) = z \right\}.
	\end{align*}
	where $\left(\lambda^+, \lambda^-\right)$ is a pair of partitions such that $\sizeof{\lambda^+} + \sizeof{\lambda^-} = n$. We have that $$\algebraicGroup{T}^{\Frobenius} \cong \left\{\left(x, y, z\right) \in \multiplicativegroup{\finiteFieldExtension{\lambda^+}} \times \multiplicativegroup{\finiteFieldExtension{2 \lambda^-}} \times \multiplicativegroup{\finiteField} \mid \FieldNorm{2 \lambda^{-}}{\lambda^{-}}\left(y\right) = z \right\},$$ and that $\mathrm{rel.rank} \algebraicGroup{T} = \lengthof\left(\lambda^+\right) + 1$. The embedding $\algebraicGroup{T}^{\Frobenius} \to \algebraicGroup{G}^{\Frobenius}$ is conjugate to an embedding of the form $$\left(x, y, z\right) \mapsto
	\begin{pmatrix}
		z \cdot \iota_+ \left(x\right)\\
		& \iota_{-}\left(y,z\right) &\\
		& & \iota_+^{\ast}\left(x\right)
	\end{pmatrix},$$
	where $\iota_+ \colon \multiplicativegroup{\finiteFieldExtension{\lambda^+}} \to \GL_{\sizeof{\lambda^+}}\left(\finiteField\right)$ and $$\iota_- \colon \left\{ \left(y, z\right)\in \multiplicativegroup{\finiteFieldExtension{2 \lambda^-}} \times \multiplicativegroup{\finiteField} \mid \FieldNorm{2 \lambda^{-}}{\lambda^{-}}\left(y\right) = z \right\} \to G'$$ are embeddings. Here, $G'$ is a group of the same type as $\algebraicGroup{G}^{\Frobenius}$, with $n$ replaced by $\sizeof{\lambda^{-}}$.
\end{theorem}

\subsection{Deligne--Lusztig theory}\label{subsec:deligne-lusztig-theory}
For each maximal rational torus $\algebraicGroup{T}$ in $\algebraicGroup{G}$, Deligne and Lusztig constructed a geometric space with commuting actions of $\algebraicGroup{T}(\finiteField)$ and $\algebraicGroup{G}(\finiteField)$.
The alternating sum of the compactly supported cohomology groups of this space is a $\algebraicGroup{T}(\finiteField) \times \algebraicGroup{G}(\finiteField)$ virtual representation, and taking the $\left(\algebraicGroup{T}(\finiteField),\theta\right)$-equivariant subspace gives a $\algebraicGroup{G}(\finiteField)$ virtual representation $\RTThetaVirtual{T}{G}{\theta}$.
The virtual representation $\RTThetaVirtual{T}{G}{\theta}$ only depends on the $\algebraicGroup{G}(\finiteField)$-conjugacy class of $(\algebraicGroup{T},\theta)$. Every irreducible $\algebraicGroup{G}(\finiteField)$-representation appears as a constituent of some virtual representation $\RTThetaVirtual{T}{G}{\theta}$ for some $T$ and some $\theta$. We denote by $\RTGTheta{T}{G}{\theta}$ (and sometimes by $\RTTheta{T}{\theta}$) the virtual character corresponding to the virtual representation $\RTThetaVirtual{T}{G}{\theta}$.

\begin{proposition}[{\cite[Corollary 12.6]{DigneMichel1991} or \cite[Theorem 1.2]{SaitoShinoda2000}}]
\label{prop:semisimple_pair_with_RTtheta}
For any class function $f \colon \algebraicGroup{G}(\finiteField) \to \cComplex$ 
(i.e., a function that is invariant under conjugation by elements of $\algebraicGroup{G}\left(\finiteField\right)$) with the property that $f(g)$ only depends on the semisimple part of $g$, that is, $f(g) = f(s)$ where $g=su$ is a Jordan decomposition of $g$, we have
\[
	\innerproduct{\RTGTheta{T}{G}{\theta}}{f}_{\algebraicGroup{G}(\finiteField)} = \innerproduct{\theta}{f\restriction_{\algebraicGroup{T}(\finiteField)}}_{\algebraicGroup{T}(\finiteField)}.
\]
i.e.
\[
	\frac{1}{\sizeof{\algebraicGroup{G}(\finiteField)}}\sum_{g \in \algebraicGroup{G}(\finiteField)}\left(\RTGTheta{T}{G}{\theta}\right)(g)\conjugate{f(g)} = \frac{1}{\sizeof{\algebraicGroup{T}(\finiteField)}}\sum_{t \in  \algebraicGroup{T}(\finiteField)} \theta(t)\conjugate{f(t)}.
\]
\end{proposition}

\subsubsection{Geometric Conjugacy and Lusztig Series}
Let $\ell$ be a prime number not equal to the characteristic of $\finiteField$. Fix an isomorphism $\ladicnumbers \cong \cComplex$.

Let $\algebraicGroup{G}$ be a reductive group over $\finiteField$ and fix a maximal $\Frobenius$-stable torus $\algebraicGroup{T} \subset \algebraicGroup{G}$ defined over $\finiteField$. We denote by $$\CharacterLattice{\algebraicGroup{T}} = \Hom_{\algebraicClosure{\finiteField}}\left(\algebraicGroup{T}, \multiplcativeScheme\right)$$ the \emph{character lattice} of $\algebraicGroup{T}$ and by $$\CocharacterLattice{\algebraicGroup{T}} = \Hom_{\algebraicClosure{\finiteField}}\left(\multiplcativeScheme, \algebraicGroup{T}\right)$$ the \emph{cocharacter lattice} of $\algebraicGroup{T}$, where we regard $\algebraicGroup{T}$ and $\multiplcativeScheme$ as algebraic groups over $\finiteField$. Recall that $\CharacterLattice{\algebraicGroup{T}}$ and $\CocharacterLattice{\algebraicGroup{T}}$ are free $\zIntegers$-modules of the same finite rank. The assignment $\CharacterLattice{\algebraicGroup{T}} \times \CocharacterLattice{\algebraicGroup{T}} \to \Hom_{\finiteField}\left(\multiplcativeScheme, \multiplcativeScheme\right)$ given by composition $\left(x,y\right) \mapsto \left(z \mapsto x\left(y\left(z\right)\right)\right)$ defines a pairing $\innerproduct{\cdot}{\cdot} \colon \CharacterLattice{\algebraicGroup{T}} \times \CocharacterLattice{\algebraicGroup{T}} \to \zIntegers$ by the rule $$x\left(y\left(z\right)\right) = z^{\innerproduct{x}{y}}.$$

The \emph{dual group over $\finiteField$} of $\algebraicGroup{G}$ with respect to the torus $\algebraicGroup{T}$ corresponds to the group $\algebraicGroup{G}^*$ with a dual torus $\algebraicGroup{T}^*$ and root datum
\[
(\CharacterLattice{\algebraicGroup{T}^*},\Phi_{\algebraicGroup{T}^*},\CocharacterLattice{\algebraicGroup{T}^*},\Phi_{\algebraicGroup{T}^*}^\vee) \cong (\CocharacterLattice{\algebraicGroup{T}},\Phi_{\algebraicGroup{T}}^\vee,\CharacterLattice{\algebraicGroup{T}},\Phi_{\algebraicGroup{T}})
\]
such that the isomorphism respects the Frobenius actions on $\CharacterLattice{\cdot}$ and $\CocharacterLattice{\cdot}$. Some authors distinguish the Frobenius actions on $\algebraicGroup{G}$ and $\algebraicGroup{G}^{\ast}$ by writing $\Frobenius$ and $\Frobenius^{\ast}$, respectively, but we will not do so and simply denote both actions by $\Frobenius$.
\begin{table}[H]
\begin{center}
\caption{Dual groups for $\algebraicGroup{G}$ and $\algebraicGroup{\GroupExtension{G}}$.}
	\begin{tabular}{|c|c|c|c|c|c|} \hline
		$\algebraicGroup{G}$ & $\algebraicGroup{T} \subset \algebraicGroup{G}$ & $\Frobenius$ action on $\CharacterLattice{\algebraicGroup{T}}$ & $\algebraicGroup{G}^{\ast}$ & $\algebraicGroup{T}^{\ast} \subset \algebraicGroup{G}^{\ast}$ \tabularnewline \hline \hline
		$\algebraicGroup{\UnitaryGroup}_{2n}$ & $\restrictionOfScalars{\quadraticExtension}{\finiteField}{\multiplcativeScheme^{n}}$ & $-1\circ w_{2n}$ & $\algebraicGroup{\UnitaryGroup}_{2n}$ &  $\restrictionOfScalars{\quadraticExtension}{\finiteField}{\multiplcativeScheme^{n}}$ \tabularnewline \hline
		$\algebraicGroup{\UnitaryGroup}_{2n+1}$ & $\restrictionOfScalars{\quadraticExtension}{\finiteField}{\multiplcativeScheme^{n}} \times \algebraicGroup{\UnitaryGroup}_1$ & $-1 \circ w_{2n+1}$ & $\algebraicGroup{\UnitaryGroup}_{2n+1}$ &  $\restrictionOfScalars{\quadraticExtension}{\finiteField}{\multiplcativeScheme^{n} \times \algebraicGroup{\UnitaryGroup}}_1$ \tabularnewline \hline		 				 
		$\algebraicGroup{\SO}_{2n+1}$ & $\multiplcativeScheme^n$ & 1 &$\algebraicGroup{\Sp}_{2n}$ &  $\multiplcativeScheme^n$ \tabularnewline \hline
		$\algebraicGroup{\Sp}_{2n}$ & $\multiplcativeScheme^n$ & 1 &$\algebraicGroup{\SO}_{2n+1}$ & $\multiplcativeScheme^n$ \tabularnewline \hline		 
		$\algebraicGroup{\SO}^{+}_{2n}$ & $\multiplcativeScheme^n$ & 1 &$\algebraicGroup{\SO}^{+}_{2n}$ &  $\multiplcativeScheme^n$ \tabularnewline \hline
		$\algebraicGroup{\SO}^{-}_{2n}$ & $\multiplcativeScheme^{n-1} \times \algebraicGroup{\UnitaryGroup}_1$ & $\left(1\right)^{\times {\left(n-1\right)}} \times \left(-1\right)$ &$\algebraicGroup{\SO}^{-}_{2n}$ &  $\multiplcativeScheme^{n-1} \times \algebraicGroup{\UnitaryGroup}_1$ \tabularnewline \hline
		$\algebraicGroup{\GSp}_{2n}$ & $\multiplcativeScheme \times \multiplcativeScheme^{n}$ & 1 &$\algebraicGroup{GSpin}_{2n+1}$ &  $\multiplcativeScheme \times \multiplcativeScheme^{n}$ \tabularnewline \hline
		$\algebraicGroup{\GSO}^{+}_{2n}$ & $\multiplcativeScheme \times \multiplcativeScheme^{n}$ & 1 & $\algebraicGroup{GSpin}_{2n}^{+}$ &  $\multiplcativeScheme \times \multiplcativeScheme^{n}$ \tabularnewline \hline
		$\algebraicGroup{\GSO}^{-}_{2n}$ & $\multiplcativeScheme^{n-1} \times \restrictionOfScalars{\finiteFieldExtension{2}}{\finiteField}{\multiplcativeScheme}$ & $\left(1\right)^{\times {\left(n-1\right)}} \times \sigma^{\ast}$ & $\algebraicGroup{GSpin}_{2n}^{-}$ &  $\multiplcativeScheme^{n-1} \times \restrictionOfScalars{\finiteFieldExtension{2}}{\finiteField}{\multiplcativeScheme}$ \tabularnewline \hline
	\end{tabular}
\end{center}
\end{table}
In the last row of the table, $\sigma \colon \restrictionOfScalars{\finiteFieldExtension{2}}{\finiteField}{\multiplcativeScheme} \to \restrictionOfScalars{\finiteFieldExtension{2}}{\finiteField}{\multiplcativeScheme}$ is the automorphism $a + b\sqrt{d} \mapsto a - b\sqrt{d}$ where $d \in \multiplicativegroup{\finiteField}$ is a non-square.
The Frobenius root action on the Weyl group is described in \Cref{subsec:rational-tori}. 

Fix an isomorphism $\multiplicativegroup{\algebraicClosure{\finiteField}} \cong \left(\mathbb{Q}\slash\mathbb{Z}\right)_{p'}$ and an embedding $\exp\colon\left(\mathbb{Q}\slash\mathbb{Z}\right)_{p'} \to \overline{\mathbb{Q}_\ell}^\times$. Here, $\left(\mathbb{Q}\slash\mathbb{Z}\right)_{p'}$ is the subgroup of $\left(\mathbb{Q}\slash\mathbb{Z}, +\right)$ consisting of elements of order coprime to $p$.

By \cite[Proposition 13.7 (ii)]{DigneMichel1991}, we have an exact sequence
\begin{equation}
	\label{eq:DL}
	\xymatrix{0 \ar[r] & \CocharacterLattice{\algebraicGroup{T}} \ar[r]^{\Frobenius-\idmap} & \CocharacterLattice{\algebraicGroup{T}} \ar[r] & \algebraicGroup{T}^{\Frobenius} \ar[r] & 0}
\end{equation}
with the last map being $y \mapsto \aFieldNorm_{\algebraicGroup{T}^{\Frobenius^n}:\algebraicGroup{T}^{\Frobenius}}(y(\zeta_n))$
for $\zeta_n$ being a $q^n-1$ root of unity corresponding to $1/(q^n-1) \in (\mathbb{Q}/\mathbb{Z})_{p'}$ and sufficiently large and divisible $n$ (so that $\algebraicGroup{T}\left(\finiteFieldExtension{n}\right)$ is split). Here, $$\aFieldNorm_{\algebraicGroup{T}^{\Frobenius^n}:\algebraicGroup{T}^{\Frobenius}}(t) = t\cdot \Frobenius(t)\cdot \Frobenius^2(t)\cdot  \hdots \cdot \Frobenius^{n-1}(t)$$ is the torus norm map.

We say that $(\algebraicGroup{T},\theta)$ is a \emph{torus character pair} for $\algebraicGroup{G}^{\Frobenius}$ if $\algebraicGroup{T} \subset \algebraicGroup{G}$ is a maximal torus and $\theta \colon \algebraicGroup{T}^{\Frobenius} \to \multiplicativegroup{\ladicnumbers}$ is a character. By the discussion above, given a torus character pair  $(\algebraicGroup{T},\theta)$ for $\algebraicGroup{G}^{\Frobenius}$, we may conjugate $\algebraicGroup{T}$ to a fixed torus $\algebraicGroup{T}'$ and associate to $\theta$ a twisted Frobenius (twisted by an element of $W\left(\algebraicGroup{T}\right)$ in the twisted Weyl group conjugacy class corresponding to $\algebraicGroup{T}$) stable element $x_{\algebraicGroup{T},\theta}$ of $\CharacterLattice{\algebraicGroup{T}'}_{(\mathbb{Q}/\mathbb{Z})_{p'}}$ via the equality
\[
\exp(\innerproduct{x_{\algebraicGroup{T},\theta}}{y}) = \theta(\aFieldNorm_{\algebraicGroup{T}^{\Frobenius^n}:\algebraicGroup{T}^{\Frobenius}}(y(\zeta_n)))
\]
for large and divisible enough $n$. This $x_{\algebraicGroup{T},\theta}$ is well defined up to the Weyl group, and in fact, by \cite[Corollary 13.9]{DigneMichel1991} there is a correspondence between torus character pairs $(\algebraicGroup{T},\theta)$ and $\Frobenius$-stable Weyl group orbits in $\CharacterLattice{\algebraicGroup{T}'}_{(\mathbb{Q}/\mathbb{Z})_{p'}}$.  Using the given the identification $\CharacterLattice{\algebraicGroup{T}'} = \CocharacterLattice{\left(\algebraicGroup{T}'\right)^{\ast}}$, we associate to a torus character pair a Frobenius stable semisimple conjugacy class in $\algebraicGroup{G}^*$ \cite[Corollary 13.12]{DigneMichel1991}.

Denote the Frobenius semisimple conjugacy class associated to a torus character pair $(\algebraicGroup{T},\theta)$ by $(s_{\algebraicGroup{T},\theta})$. Similarly, denote the $\Frobenius$-stable Weyl group orbit in $\CharacterLattice{\algebraicGroup{T}}_{(\mathbb{Q}/\mathbb{Z})_{p'}}$ corresponding to $(\algebraicGroup{T},\theta)$ by $(x_{\algebraicGroup{T},\theta}).$

We say that the torus character pairs $(\algebraicGroup{T},\theta)$ and $(\algebraicGroup{T}',\theta')$ for $\algebraicGroup{G}$ are \emph{geometrically conjugate} if the geometric conjugacy classes $(s_{\algebraicGroup{T},\theta})$ and $(s_{\algebraicGroup{T}',\theta'})$ are equal. This is equivalent to the following: there exists $n \ge 1$ and $h \in {\algebraicGroup{G}}^{\Frobenius^n}$ such that $h \algebraicGroup{T}^{\Frobenius^n} h^{-1} = \left(\algebraicGroup{T}'\right)^{\Frobenius^n}$ and such that $$\theta\left(\aFieldNorm_{\algebraicGroup{T}^{\Frobenius^n}:\algebraicGroup{T}^{\Frobenius}}\left(t\right)\right) = \theta'\left(\aFieldNorm_{\left(\algebraicGroup{T}'\right)^{\Frobenius^n}:\left(\algebraicGroup{T}'\right)^{\Frobenius}}\left(h t h^{-1}\right)\right),$$ for every $t \in \algebraicGroup{T}^{\Frobenius^n}$.

\begin{example}
	Let $\algebraicGroup{T}'_2 = \restrictionOfScalars{\finiteFieldExtension{2}}{\finiteField}{\multiplcativeScheme}$ and $$\algebraicGroup{T}_2 = \left\{x \in \restrictionOfScalars{\finiteFieldExtension{2}}{\finiteField}{\multiplcativeScheme} \mid x \involution{x} = 1 \right\}.$$
	Then $\left(\algebraicGroup{T}'_2\right)^{\Frobenius} = \multiplicativegroup{\finiteFieldExtension{2}}$ and $\algebraicGroup{T}_2^{\Frobenius} = \NormOneGroup{2}$. Let $\theta \colon \NormOneGroup{2} \to \multiplicativegroup{\ladicnumbers}$ be a character, and let $\theta' \colon \multiplicativegroup{\finiteFieldExtension{2}} \to \multiplicativegroup{\ladicnumbers}$ be the character given by $$\theta'\left(x\right) = \theta\left(x^{1-q}\right).$$
	Both the tori $\algebraicGroup{T}_2 \times \algebraicGroup{T}_2$ and $\algebraicGroup{T}_2'$ are tori inside $\algebraicGroup{Sp}_4$.
	We claim that the torus character pairs $\left(\algebraicGroup{T}_2 \times \algebraicGroup{T}_2, \theta \times \theta\right)$ and $\left(\algebraicGroup{T}'_2, \theta'\right)$ are geometrically conjugate. Indeed, if $m \ge 2$ is even, then $\left(\algebraicGroup{T}'_2\right)^{\Frobenius^{m}} = \multiplicativegroup{\finiteFieldExtension{m}} \times \multiplicativegroup{\finiteFieldExtension{m}}$ and $\algebraicGroup{T}_2^{\Frobenius^m} = \multiplicativegroup{\finiteFieldExtension{m}}$. The norm map $\left(\algebraicGroup{T}'_2\right)^{\Frobenius^{m}} \to \left(\algebraicGroup{T}'_2\right)^{\Frobenius}$ is given by $\multiplicativegroup{\finiteFieldExtension{m}} \times \multiplicativegroup{\finiteFieldExtension{m}} \ni \left(x,y\right) \mapsto \FieldNorm{m}{2}\left(x\right) \FieldNorm{m}{2}\left(y\right)$, while the norm map $\algebraicGroup{T}_2^{\Frobenius^{m}} \to \algebraicGroup{T}_2^{\Frobenius}$ is given by $\multiplicativegroup{\finiteFieldExtension{m}} \ni x \mapsto \FieldNorm{m}{2}\left(x\right)^{1-q}$. We thus have
	$$\left(\theta \times \theta\right) \circ \aFieldNorm_{ \left(\algebraicGroup{T}_2 \times \algebraicGroup{T}_2\right)^{\Frobenius^m} : \left(\algebraicGroup{T}_2 \times \algebraicGroup{T}_2\right)^{\Frobenius} } = \theta' \circ \aFieldNorm_{\left(\algebraicGroup{T}'_2\right)^{\Frobenius^m} \colon \left(\algebraicGroup{T}'_2\right)^{\Frobenius}}.$$
\end{example}

The virtual representations $\RTThetaVirtual{T}{G}{\theta}$ and $\RTThetaVirtual{T'}{G}{\theta'}$ share an irreducible subrepresentation only if $(T,\theta),(T',\theta')$ are geometrically conjugate. This yields a natural partitioning of the irreducible representations of $\algebraicGroup{G}(\finiteField)$: the \emph{(geometric) Lusztig series} $\LusztigSeries{\algebraicGroup{G}^{\Frobenius}}{s}$ associated to a Frobenius stable geometric conjugacy class $(s) \subset \algebraicGroup{G}^*$ consists of all (equivalence classes of) irreducible representations $\pi$ which appear with non-zero multiplicity in $\RTThetaVirtual{T}{G}{\theta}$ for some $(T,\theta)$ such that $(s_{\algebraicGroup{T},\theta}) = (s)$.

For later use, we note Lusztig series are a coarsening of Harish-Chandra series.
\begin{proposition}[{\cite[Proof of Corollary 3.3.21]{GeckMalle2020}}]\label{prop:lusztig-series-are-unions-of-harish-chandra-series}
	For any semisimple element of the connected center group $s \in \FrobeniusFixedPoints{\GroupExtension{G}}$, the Lusztig series $\LusztigSeries{\FrobeniusFixedPoints{\GroupExtension{G}}}{s}$ is a union of Harish--Chandra series corresponding to Harish-Chandra data of the form $\left(\FrobeniusFixedPoints{\algebraicGroup{P}},\sigma\right)$ where $\algebraicGroup{P}$ is a parabolic subgroup of $\algebraicGroup{\GroupExtension{G}}$ with Levi part $\algebraicGroup{L}$ with $s \in \FrobeniusFixedPoints{L}$ and $\sigma \in \LusztigSeries{\FrobeniusFixedPoints{L}}{s}$ is an irreducible cuspidal representation.
\end{proposition}

Let us now specialize to the classical groups $\algebraicGroup{G}$ and their extensions $\GroupExtension{\algebraicGroup{G}}$ from Sections \ref{subsec:isometry-groups} and \ref{subsec:similitude-groups}. 

In these cases, the map $i:\algebraicGroup{G} \to \GroupExtension{\algebraicGroup{G}}$ yields a pullback map $i^*:\GroupExtension{\algebraicGroup{G}}^* \to \algebraicGroup{G}^*$ well defined up to composing with an inner automorphism induced by a $\Frobenius$-stable element of the maximal torus. From chasing definitions, one finds the following.

\begin{proposition}[{\cite[Corollary 9.7]{Bonnafe2006}}]
	\label{prop:i*isrestriction}
	$i^*(s_{\algebraicGroup{T},\theta}) = s_{\algebraicGroup{T} \cap G, \theta\restriction_{\algebraicGroup{T} \cap G}}$.
\end{proposition}

This gives a natural surjection from geometric Lusztig series for the similitude group $\GroupExtension{\algebraicGroup{G}}$ to those for its classical counterpart $\algebraicGroup{G}$.

The following proposition explains how $i^*$ relates irreducible representations of $\GroupExtension{\algebraicGroup{G}}$ with their restriction to $\algebraicGroup{G}$.

\begin{proposition}[{\cite[Proposition 11.7]{Bonnafe2006}}]\label{prop:lusztig-series-and-restriction}
	Let $\tilde{s}$ be a semisimple element of $\FrobeniusFixedPoints{\GroupExtension{G}}$ and let $s = i^{\ast}\left(\tilde{s}\right)$. Then
	\begin{enumerate}
		\item If $\tilde{\pi} \in \LusztigSeries{\FrobeniusFixedPoints{\GroupExtension{G}}}{\tilde{s}}$ and if $\pi$ is an irreducible subrepresentation of the restriction of $\tilde{\pi}$ to $\FrobeniusFixedPoints{G}$, then $\pi \in \LusztigSeries{\FrobeniusFixedPoints{G}}{s}$.
		\item Let $\pi \in \LusztigSeries{\FrobeniusFixedPoints{G}}{s}$. Then there exists $\tilde{\pi} \in \LusztigSeries{\FrobeniusFixedPoints{\GroupExtension{G}}}{\tilde{s}}$ such that $\pi$ is an irreducible subrepresentation of the restriction of $\tilde{\pi}$ to $\FrobeniusFixedPoints{G}$.
	\end{enumerate}
\end{proposition}

For our later arguments, it is useful to only have one cuspidal representation in a Lusztig series. Unfortunately, when the center of $\algebraicGroup{G}$ is disconnected, this is not always true. However, it is always true for its connected center group version $\GroupExtension{\algebraicGroup{G}}$.

\begin{proposition}[{\cite[Page 172]{Lusztig1977}}]\label{prop:at-most-one-cuspidal-lusztig-series}
	For any semisimple element $s \in \FrobeniusFixedPoints{\GroupExtension{G}}$, the Lusztig series $\LusztigSeries{\FrobeniusFixedPoints{\GroupExtension{G}}}{s}$ contains at most one cuspidal element.
\end{proposition}

\subsection{Gauss sums}\label{subsec:gauss-sums}
In this section we review Kondo's result about non-abelian Gauss sums \cite{Kondo1963}, and prove a vanishing result for certain degenerate non-abelian Gauss sums.

Let $\fieldCharacter \colon \finiteField \to \multiplicativegroup{\cComplex}$ be a non-trivial character. Let $\fieldCharacter_{\quadraticExtension} = \fieldCharacter \circ \trace_{\quadraticExtension \slash \finiteField}$. For any $k \ge 1$ let $\fieldCharacter_k \colon \finiteFieldExtension{k} \to \multiplicativegroup{\cComplex}$ be the character $\fieldCharacter_k = \fieldCharacter \circ \trace_{\finiteFieldExtension{k} \slash \finiteField}$, and similarly let $\fieldCharacter_{\quadraticFieldExtension{k}} \colon \quadraticFieldExtension{k} \to \multiplicativegroup{\cComplex}$ be given by $\fieldCharacter_{\quadraticFieldExtension{k}} = \fieldCharacter \circ \trace_{\quadraticFieldExtension{k} \slash \finiteField}$.

Given an irreducible representation $\tau$ of $\GL_k\left(\quadraticExtension\right)$ and a character $\chi \colon \multiplicativegroup{\quadraticExtension} \to \multiplicativegroup{\cComplex}$, we denote the \emph{twisted Gauss sum}
$$\GaussSum{\tau \times \chi}{\fieldCharacter_{\quadraticExtension}} = q^{-\grpIndex{\quadraticExtension}{\finiteField} k^2 \slash 2} \sum_{g \in \GL_k\left(\quadraticExtension\right)} \tau\left(g\right) \chi\left(\det g\right) \fieldCharacter_{\quadraticExtension}\left(\trace g^{-1}\right).$$
This element lies in $\Hom_{\GL_k\left(\quadraticExtension\right)}\left(\tau, \tau\right)$ and therefore by Schur's lemma there exists a complex number $\GaussSumScalar{\tau \times \chi}{\fieldCharacter_{\quadraticExtension}} \in \cComplex$ such that
$$\GaussSum{\tau \times \chi}{\fieldCharacter_{\quadraticExtension}} = \GaussSumScalar{\tau \times \chi}{\fieldCharacter_{\quadraticExtension}} \cdot \idmap_\tau.$$ The computation of this scalar is known due to Kondo \cite{Kondo1963}.

In order to state Kondo's result, we need to introduce some Gauss sum notation which will be used throughout the paper.

Given a character $\alpha \colon \multiplicativegroup{\quadraticFieldExtension{k}} \to \multiplicativegroup{\cComplex}$, we denote the \emph{Gauss sum} $$\GaussSumSingleCharacter{\alpha}{\fieldCharacter_{\quadraticFieldExtension{k}}} = -q^{-\grpIndex{\quadraticExtension}{\finiteField} k \slash 2}\sum_{x \in \multiplicativegroup{\quadraticFieldExtension{k}}} \alpha^{-1}\left(x\right) \fieldCharacter_{\quadraticFieldExtension{k}}\left(x\right).$$ If $\chi \colon \multiplicativegroup{\quadraticExtension} \to \multiplicativegroup{\cComplex}$ is a character and if $\alpha$ is as above, we denote the \emph{twisted Gauss sum}
$$\GaussSumCharacter{\alpha}{\chi}{\fieldCharacter_{\quadraticFieldExtension{k}}} = -q^{-\grpIndex{\quadraticExtension}{\finiteField} k \slash 2}\sum_{x \in \multiplicativegroup{\quadraticFieldExtension{k}}} \alpha^{-1}\left(x\right) \chi^{-1}\left( \aFieldNorm_{\quadraticFieldExtension{k} \slash \quadraticExtension}\left(x\right)\right) \fieldCharacter_{\quadraticFieldExtension{k}}\left(x\right).$$
Notice that the Gauss sum $\GaussSumCharacter{\alpha}{\chi}{\fieldCharacter_{\quadraticFieldExtension{k}}}$ is invariant under the Frobenius action, that is, if $\left(\Frobenius \alpha\right)\left(x\right) = \alpha\left(x^q\right)$ and $\left(\Frobenius \chi\right)\left(x\right) = \chi\left(x^q\right)$ then $$\GaussSumCharacter{\Frobenius \alpha}{\Frobenius \chi}{\fieldCharacter_{\quadraticFieldExtension{k}}} = \GaussSumCharacter{\alpha}{\chi}{\fieldCharacter_{\quadraticFieldExtension{k}}}.$$

We also denote for characters $\alpha \colon \multiplicativegroup{\finiteFieldExtension{k}} \to \multiplicativegroup{\cComplex}$ and $\chi \colon \multiplicativegroup{\finiteField} \to \multiplicativegroup{\cComplex}$, $\GaussSumSingleCharacter{\alpha}{\fieldCharacter_{k}} = \GaussSumSingleCharacter{\alpha}{\fieldCharacter_{\finiteFieldExtension{k}}}$ and $\GaussSumCharacter{\alpha}{\chi}{\fieldCharacter_{k}} = \GaussSumCharacter{\alpha}{\chi}{\fieldCharacter_{\finiteFieldExtension{k}}}$, where $\quadraticExtension$ is replaced with $\finiteField$ in the definitions above. 

Given a maximal torus $T \subset \GL_k\left(\quadraticExtension\right)$ of the form $T \cong \multiplicativegroup{\quadraticFieldExtension{\lambda}}$ where $\lambda \vdash k$ is a partition and given characters $\alpha \colon T \to \multiplicativegroup{\cComplex}$ and $\chi \colon \multiplicativegroup{\quadraticExtension} \to \multiplicativegroup{\cComplex}$, we define $$\GaussSumTorusCharacter{T}{\alpha}{\chi}{\fieldCharacter_{\quadraticExtension}} = \left(-1\right)^{\lengthof\left(\lambda\right)} q^{-\grpIndex{\quadraticExtension}{\finiteField} k \slash 2} \sum_{t \in T} \alpha^{-1}\left(t\right) \chi^{-1}\left(\det t\right) \fieldCharacter_{\quadraticExtension}\left(\trace_{\slash \quadraticExtension} t\right).$$
We also define $$\GaussSumSingleTorusCharacter{T}{\alpha}{\fieldCharacter_{\quadraticExtension}} = \GaussSumTorusCharacter{T}{\alpha}{1}{\fieldCharacter_{\quadraticExtension}}.$$
Notice that if $\alpha = \alpha_1 \times \dots \times \alpha_{\lengthof\left(\lambda\right)}$, where $\alpha_j \colon \multiplicativegroup{\quadraticFieldExtension{\lambda_j}} \to \multiplicativegroup{\cComplex}$ is a character for every $j$, then
\begin{align}
	\label{eq:Gauss_multiplicitive}
	\GaussSumTorusCharacter{T}{\alpha}{\chi}{\fieldCharacter_{\quadraticExtension}} = \prod_{j=1}^{\lengthof\left(\lambda\right)} \GaussSumCharacter{\alpha_j}{\chi}{\fieldCharacter_{\quadraticFieldExtension{\lambda_j}}}.
\end{align}

Recall that a character $\alpha \colon \multiplicativegroup{\quadraticFieldExtension{k}} \to \multiplicativegroup{\cComplex}$ is called \emph{regular} if the Frobenius orbit of $\alpha$, defined as the set $\left\{\alpha^{q^{\grpIndex{\quadraticExtension}{\finiteField} j}} \mid j \ge 0\right\}$, is of size $k$. By \cite[Section 6]{Gelfand70}, Frobenius orbits of regular characters $\alpha \colon \multiplicativegroup{\quadraticFieldExtension{k}} \to \multiplicativegroup{\cComplex}$ are in bijection with irreducible cuspidal representations of $\GL_k\left(\quadraticExtension\right)$.

We are now ready to state Kondo's result.
\begin{theorem}[Kondo {\cite{Kondo1963}}]\label{thm:KondoClassical}
	\begin{enumerate}
		\item Suppose that $\tau$ is an irreducible representation of $\GL_k\left(\quadraticExtension\right)$ such that $\tau$ is a subrepresentation of the parabolically induced representation $\Ind{P_{(k_1,\dots,k_r)}}{\GL_k\left(\quadraticExtension\right)}{\tau_1 \overline{\otimes} \dots \overline{\otimes} \tau_r}$, where $k_1 + \dots + k_r = k$, and where for every $j$, $\tau_j$ is an irreducible cuspidal representation of $\GL_{k_j}\left(\quadraticExtension\right)$ corresponding to the Frobenius orbit of a regular character $\alpha_j \colon \multiplicativegroup{\quadraticFieldExtension{k_j}} \to \multiplicativegroup{\cComplex}$. Then
		$$\GaussSumScalar{\tau \times \chi}{\fieldCharacter_{\quadraticExtension}} = \left(-1\right)^k \cdot \prod_{j=1}^r \GaussSumCharacter{\alpha_j}{\chi}{\fieldCharacter_{\quadraticFieldExtension{k_j}}}.$$
		\item Suppose that $\tau$ is an irreducible representation of $\GL_k\left(\quadraticExtension\right)$ such that $\innerproduct{\trace \tau}{\RTTheta{T}{\alpha}} \ne 0$ for some maximal torus $T$ of $\GL_k\left(\quadraticExtension\right)$ and some character $\alpha \colon T \to \multiplicativegroup{\cComplex}$. Then $$\GaussSumScalar{\tau \times \chi}{\fieldCharacter_{\quadraticExtension}} = \left(-1\right)^k \GaussSumTorusCharacter{T}{\alpha}{\chi}{\fieldCharacter_{\quadraticExtension}}.$$
	\end{enumerate}
\end{theorem}
\begin{remark}\label{rem:kondo-sum-implies-geometric-conjugacy-invariance}
	By the second part of \Cref{thm:KondoClassical}, we see that the value $\GaussSumTorusCharacter{T}{\alpha}{\chi}{\fieldCharacter_{\quadraticExtension}}$ does not depend on the torus character pair $\left(\algebraicGroup{T}, \alpha\right)$ corresponding to $\pi$. Equivalently, this means that the assignment $\left(\algebraicGroup{T}, \alpha\right) \mapsto \GaussSumTorusCharacter{T}{\alpha}{\chi}{\fieldCharacter_{\quadraticExtension}}$ is invariant under geometric conjugacy. One can also conclude this directly from the Hasse--Davenport relation. 
\end{remark}

We will need the following lemma that allows us to determine whether a degenerate Kondo Gauss sum vanishes.
\begin{lemma}\label{lem:sum-vanishes-for-singular-matrices}
	For any singular matrix $X \in \squareMatrix_k\left(\quadraticExtension\right)$ and any $1 \ne \chi \colon \multiplicativegroup{\quadraticExtension} \to \multiplicativegroup{\cComplex}$, the sum
	$$\sum_{h \in \GL_k\left(\quadraticExtension\right)} \chi\left(\detQuadratic h\right) \fieldCharacter_{\quadraticExtension}\left(\trace_{\slash \quadraticExtension}\left(Xh\right)\right)$$
	is zero.
\end{lemma}
\begin{proof}
	Write $$X = h_1 \begin{pmatrix}
		\IdentityMatrix{k-r}\\
		& 0_r
	\end{pmatrix} h_2,$$
	where $1 \le r \le k$ and $h_1, h_2 \in \GL_k\left(\quadraticExtension\right)$. Then
	$$\fieldCharacter_{\quadraticExtension}\left(\trace Xh\right) = \fieldCharacter_{\quadraticExtension}\left(\trace\left( \begin{pmatrix}
		\IdentityMatrix{k-r}\\
		& 0_r
	\end{pmatrix} h_2 h h_1\right)\right).$$
	Changing variables $h \mapsto h_2^{-1} h h_1^{-1},$
	we obtain the inner sum
	$$\sum_{h \in \GL_k\left(\quadraticExtension\right)} \fieldCharacter_{\quadraticExtension}\left(\trace \begin{pmatrix}
		\IdentityMatrix{k-r}\\
		& 0_r
	\end{pmatrix} h \right) \chi\left(\det h\right),$$
	which equals
	$$\frac{1}{\sizeof{\multiplicativegroup{\quadraticExtension}}} \sum_{a \in \multiplicativegroup{\quadraticExtension}} \sum_{h \in \GL_k\left(\quadraticExtension\right)} \fieldCharacter_{\quadraticExtension}\left(\trace \begin{pmatrix}
		\IdentityMatrix{k-r}\\
		& 0_r
	\end{pmatrix} \begin{pmatrix}
		\IdentityMatrix{k-1}\\
		& a
	\end{pmatrix} h \right) \chi\left(\det h\right).$$
	Changing variables $h \mapsto \left(\begin{smallmatrix}
		\IdentityMatrix{k-1}\\
		& a^{-1}
	\end{smallmatrix}\right) h$, we obtain the inner sum $$\sum_{a \in \multiplicativegroup{\quadraticExtension}} \chi^{-1}\left(a\right),$$
	which vanishes because $\chi \ne 1$.
\end{proof}

\section{Doubling method Jacobi sums}\label{sec:doubling-method-jacobi-sums}
In this section we study non-abelian Jacobi sums. These sums arise from the finite field analog construction of the doubling method \cite{Chang1997, GirschZelingher2026}. We begin with defining these non-abelian Jacobi sums. We then recall the multiplicativity property they satisfy. Then we perform explicit computations for the analogous character sums, and show that they can be expressed in terms of Gauss sums. Next we show that the kernel function $\posHermitianJacobiKernel{\hermitianSpace}{\chi}$ used to define the sum is a stable functions. Finally, combining these results with results of Lusztig, we find an explicit formula for the Jacobi sum associated to an irreducible representation in terms of its Deligne--Lusztig data (\Cref{thm:doubling-method-gamma-factor-for-deligne-lusztig}).

\subsection{Definition}\label{subsec:definition-of-jacobi-sums}
In this section we define Jacobi sums for irreducible representations of finite general linear groups and finite classical groups.

\subsubsection{The case of general linear groups}\label{subsubsec:jacobi-sums-for-general-linear-groups}

Let $\chi \colon \multiplicativegroup{\quadraticExtension} \to \multiplicativegroup{\cComplex}$ be a non-trivial character. Consider the following assignment $\GL_k\left(\quadraticExtension\right) \to \cComplex$
$$\JacobiKernel{\chi}\left(g\right) = \begin{dcases}
	\chi\left(\detQuadratic\left(\IdentityMatrix{k}+g\right)\right) & \text{if }\det\left(\IdentityMatrix{k}+g\right) \ne 0,\\
	0 & \text{otherwise.}
\end{dcases}$$
It is clear that $\JacobiKernel{\chi}$ is a class function of $\GL_k\left(\quadraticExtension\right)$.

Given an irreducible representation $\tau$ of $\GL_k\left(\quadraticExtension\right)$, we consider the following \emph{doubling method Jacobi sum}:
$$\posDblJacobiSum{\tau}{\chi} = q^{-\frac{k^2}{2}} \sum_{g \in \GL_k\left(\quadraticExtension\right)} \JacobiKernel{\chi}\left(g\right) \tau\left(g\right).$$
Then $\posDblJacobiSum{\tau}{\chi}$ defines an element of $\Hom_{\GL_k\left(\quadraticExtension\right)}\left(\tau, \tau\right)$. Therefore, by Schur's lemma there exists a complex number $\posDblJacobiSumScalar{\tau}{\chi} \in \cComplex$ such that $$\posDblJacobiSum{\tau}{\chi} = \posDblJacobiSumScalar{\tau}{\chi} \cdot \idmap_\tau.$$ 

The goal of this section is to express $\posDblJacobiSumScalar{\tau}{\chi}$ in terms of Kondo's Gauss sum.

\begin{proposition}\label{prop:doubling-for-gln-in-terms-of-kondo}
	Let $\chi \colon \multiplicativegroup{\quadraticExtension} \to \multiplicativegroup{\cComplex}$ be a non-trivial character. We have the identity
	$$\JacobiKernel{\chi}\left(g\right) = \chi\left(-1\right)^k q^{-\grpIndex{\quadraticExtension}{\finiteField} k^2\slash 2} \GaussSumScalar{\chi^{-1}_{\GL_k}}{\fieldCharacter_{\quadraticExtension}} \sum_{\substack{x, y \in \GL_k\left(\quadraticExtension\right)\\
			y^{-1} x = g}} \fieldCharacter_{\quadraticExtension}\left(\trace_{\slash \quadraticExtension} x\right) \chi^{-1}\left(\detQuadratic y\right) \fieldCharacter_{\quadraticExtension}\left(\trace_{\slash \quadraticExtension} y\right),$$
	where $\chi_{\GL_k} \colon \GL_k\left(\quadraticExtension\right) \to \multiplicativegroup{\cComplex}$ is the character given by $\chi_{\GL_k}\left(g\right) = \chi\left(\detQuadratic g\right)$ for every $g \in \GL_k\left(\quadraticExtension\right)$.
\end{proposition}
\begin{proof}
	We write $$\JacobiKernel{\chi}\left(g\right) = \frac{1}{\sizeof{\squareMatrix_k\left(\quadraticExtension\right)}}\sum_{X \in \squareMatrix_k\left(\quadraticExtension\right)} \sum_{h \in \GL_k\left(\quadraticExtension\right)} \fieldCharacter_{\quadraticExtension}\left(\trace_{\slash \quadraticExtension} \left(X\left(g+\IdentityMatrix{k}-h\right)\right)\right) \chi\left(\detQuadratic h\right).$$
	By \Cref{lem:sum-vanishes-for-singular-matrices}, we can reduce the sum to $X \in \GL_k\left(\quadraticExtension\right)$. Thus $$\JacobiKernel{\chi}\left(g\right) = \frac{1}{\sizeof{\squareMatrix_k\left(\quadraticExtension\right)}} \sum_{X, h \in \GL_k\left(\quadraticExtension\right)} \fieldCharacter_{\quadraticExtension}\left(\trace_{\slash \quadraticExtension} \left(X\left(g+\IdentityMatrix{k}-h\right)\right)\right) \chi\left(\detQuadratic h\right).$$
	Changing variables $h \mapsto X^{-1} h$, we get
	$$\JacobiKernel{\chi}\left(g\right) = \frac{1}{\sizeof{\squareMatrix_k\left(\quadraticExtension\right)}} \sum_{X, h \in \GL_k\left(\quadraticExtension\right)} \fieldCharacter_{\quadraticExtension}\left(\trace \left(Xg\right)\right) \fieldCharacter_{\quadraticExtension}\left(-\trace h\right) \chi\left(\det h\right) \chi^{-1}\left(\det X\right) \fieldCharacter_{\quadraticExtension}\left(\trace X\right).$$
	Changing variables again $h \mapsto -h$ and setting $Xg = x$ and $X = y$, we get  
	$$\JacobiKernel{\chi}\left(g\right) = \chi\left(-1\right)^k q^{-\grpIndex{\quadraticExtension}{\finiteField}k^2 \slash 2} \GaussSumScalar{\chi_{\GL_k}}{\fieldCharacter_{\quadraticExtension}} \sum_{\substack{x, y \in \GL_k\left(\quadraticExtension\right)\\
			y^{-1} x = g}} \fieldCharacter_{\quadraticExtension}\left(\trace x\right) \chi^{-1}\left(\det y\right) \fieldCharacter_{\quadraticExtension}\left(\trace y\right).$$
\end{proof}

If $\chi = 1$ is the trivial character, we define $\JacobiKernel{\chi}$ by the formula given in \Cref{prop:doubling-for-gln-in-terms-of-kondo}. We have that for every character $\chi \colon \multiplicativegroup{\quadraticExtension} \to \multiplicativegroup{\cComplex}$
\begin{equation}\label{eq:expression-of-jacobi-kernel-as-fourier-transform}
	\JacobiKernel{\chi}\left(g\right) = \chi\left(-1\right)^k q^{-\grpIndex{\quadraticExtension}{\finiteField}k^2 \slash 2} \GaussSumScalar{\chi_{\GL_k}}{\fieldCharacter_{\quadraticExtension}} \sum_{x \in \GL_k\left(\quadraticExtension\right)} \fieldCharacter_{\quadraticExtension}\left(\trace_{\slash \quadraticExtension} \left(x\left(\IdentityMatrix{k}+g\right)\right)\right) \chi^{-1}\left(\detQuadratic x\right).
\end{equation}
It is easy to check that $\JacobiKernel{\chi}$ is a class function from \eqref{eq:expression-of-jacobi-kernel-as-fourier-transform} (which we already knew if $\chi \ne 1$).

When $\chi = 1$, we have the following explicit expression for $\JacobiKernel{\chi}$.
\begin{theorem}[{\cite[Theorem 2.1]{LiHu2012}}]\label{thm:kernel-for-chi-equals-1}
	For any $g \in \GL_k\left(\quadraticExtension\right)$,
	$$\JacobiKernel{1}\left(g\right) = \left(-1\right)^{\rank \left(\IdentityMatrix{k} + g\right)} q^{-\grpIndex{\quadraticExtension}{\finiteField} k} \prod_{i=1}^{k - \rank\left(\IdentityMatrix{k} + g\right)} \left(q^{\grpIndex{\quadraticExtension}{\finiteField} i} - 1\right).$$
\end{theorem}

We can now compute $\posDblJacobiSum{\tau}{\chi}$ regardless whether $\chi$ is trivial or not. By \Cref{prop:doubling-for-gln-in-terms-of-kondo} (or by the definition when $\chi = 1$), we have that $$\posDblJacobiSum{\tau}{\chi} = q^{-\grpIndex{\quadraticExtension}{\finiteField}k^2} \chi\left(-1\right)^k \GaussSumScalar{\chi_{\GL_k}}{\fieldCharacter_{\quadraticExtension}} \sum_{x, y \in \GL_k\left(\quadraticExtension\right)} \fieldCharacter_{\quadraticExtension}\left(\trace x\right) \chi^{-1}\left(\det y\right) \fieldCharacter_{\quadraticExtension}\left(\trace y\right) \tau\left(y^{-1} x\right).$$
The last equality implies the following result.
\begin{theorem}\label{thm:gln-doubling-gauss-sum-in-terms-of-kondo}For any irreducible representation $\tau$ of $\GL_k\left(\quadraticExtension\right)$ and any character $\chi \colon \multiplicativegroup{\quadraticExtension} \to \multiplicativegroup{\cComplex}$, the following identity holds:
	$$\posDblJacobiSumScalar{\tau}{\chi} = \chi\left(-1\right)^k \left(-\GaussSumSingleCharacter{\chi^{-1}}{\fieldCharacter_{\quadraticExtension}}\right)^k \GaussSumScalar{\tau^{\vee}}{\fieldCharacter_{\quadraticExtension}} \GaussSumScalar{\tau \times \chi}{\fieldCharacter_{\quadraticExtension}}.$$
\end{theorem}

\subsubsection{The case of classical groups}

Let $\left(\hermitianSpace, \innerproduct{\cdot}{\cdot}\right)$ be a non-degenerate $\epsilon_{\hermitianSpace}$-sesquilinear hermitian space. Let $G$ and $\tilde{G}$ be as in Sections \ref{subsec:isometry-groups} and \ref{subsec:similitude-groups}.

Let $\chi \colon \multiplicativegroup{\quadraticExtension} \to \multiplicativegroup{\cComplex}$ be a non-trivial character. As before, the assignment $\JacobiKernel{\chi} \colon \GL_{\quadraticExtension}\left(\hermitianSpace\right) \to \cComplex$ given by $$\JacobiKernel{\chi}\left(g\right) = \begin{dcases}
	\chi\left(\detQuadratic\left( \idmap_{\hermitianSpace} + g\right)\right) & \text{if }\detQuadratic\left( \idmap_{\hermitianSpace} + g\right) \ne 0\\
	0 & \text{otherwise,}
\end{dcases}$$
is a class function of $\GL_{\quadraticExtension}\left(\hermitianSpace\right)$.

We define $\posHermitianJacobiKernel{\hermitianSpace}{\chi} \colon \tilde{G} \to \cComplex$ by $$\posHermitianJacobiKernel{\hermitianSpace}{\chi}\left(g\right) =
	\begin{dcases}
		\JacobiKernel{\chi}\left(g\right) & g \in G,\\
		0 & \text{otherwise.}
	\end{dcases}$$
Let $\pi$ be an irreducible representation of $H$ where $H$ is either the classical group $G$ or its similitude group counterpart $\GroupExtension{G}$. Denote $$\posDblJacobiSum{\pi}{\chi} = \frac{1}{\sqrt{\sizeof{\lieAlgebra}}} \sum_{g \in H} \pi\left(g\right) \posHermitianJacobiKernel{\hermitianSpace}{\chi}\left(g\right) = \frac{1}{\sqrt{\sizeof{\lieAlgebra}}} \sum_{g \in G} \pi\left(g\right) \posJacobiKernel{\chi}\left(g\right).$$
Since $\posJacobiKernel{\chi}$ is a class function of $\GL_{\quadraticExtension}\left(\hermitianSpace\right)$, and since the classical group $G$ is a normal subgroup of $H$, we have that $\posHermitianJacobiKernel{\hermitianSpace}{\chi}$ is a class function of $H$. It follows that $\posDblJacobiSum{\pi}{\chi} \in \Hom_{H}\left(\pi, \pi\right)$. By Schur's lemma, there exists a constant $\posDblJacobiSumScalar{\pi}{\chi} \in \cComplex$, such that $$\posDblJacobiSum{\pi}{\chi} = \posDblJacobiSumScalar{\pi}{\chi} \cdot \idmap_\pi.$$
We call $\posDblJacobiSumScalar{\pi}{\chi}$ the \emph{doubling method Jacobi sum}.

\subsection{Multiplicativity}\label{subsec:multiplicativity}
The goal of this section is to explain how doubling method Jacobi sums behave under parabolic induction.

Let $\xIsotropic$ and $\yIsotropic$ be totally isotropic spaces of $\hermitianSpace$ of dimension $k$, such that $\xIsotropic$ and $\yIsotropic$ are in duality with respect to form $\innerproduct{\cdot}{\cdot}$. Let us write $$\hermitianSpace = \xIsotropic \oplus \hermitianSpace' \oplus \yIsotropic,$$
where $\hermitianSpace' \subset \hermitianSpace$ is a non-degenerate subspace, orthogonal to $\xIsotropic$ and $\yIsotropic$. Let $P$ be the parabolic subgroup of $H$, consisting of all elements stabilizing the flag $$0 \subset \xIsotropic \subset \xIsotropic \oplus \hermitianSpace' \subset \xIsotropic \oplus \hermitianSpace' \oplus \yIsotropic = \hermitianSpace.$$
Let $G'$ be the connected component of $\IsometryGroup\left(\hermitianSpace'\right)$. Write $P = L \ltimes N$, where $L$ is the Levi part of $P$ and $N$ is the unipotent radical of $P$. Then $L$ is isomorphic to $\GL_k\left(\quadraticExtension\right) \times H'$, where if $H = G$ is a classical group then $H' = G'$ is also a classical group and otherwise if $H = \GroupExtension{G}$ is the similitude group counterpart then so is $H' = \GroupExtension{G'}$.

Given a multiplicative character $\chi \colon \multiplicativegroup{\quadraticExtension} \to \multiplicativegroup{\cComplex}$, let $\involutionPlusOne{\chi} \colon \multiplicativegroup{\quadraticExtension} \to \multiplicativegroup{\cComplex}$ be the character given by
$$\involutionPlusOne{\chi}\left(x\right) = \chi\left(x \cdot \involution{x}\right).$$ Similarly, let $\minusInvolution{\chi} \colon \multiplicativegroup{\quadraticExtension} \to \multiplicativegroup{\cComplex}$ be the character given by
$$\minusInvolution{\chi}\left(x\right) = \chi^{-1}\left(\involution{x}\right).$$

\begin{theorem}[{\cite{GirschZelingher2026}}]\label{thm:multiplicativity-in-terms-of-gauss-sums}
	Let $\pi'$ be an irreducible representation of $H'$ and let $\tau$ be an irreducible representation of $\GL_k\left(\quadraticExtension\right)$. Then for every irreducible representation $\pi$ of $H$ which appears as a subrepresentation of the parabolic induction $\rho = \Ind{P}{H}{\tau \overline{\otimes} \pi'}$ and any $\chi$ such that $\involutionPlusOne{\chi} \ne 1$, we have
	$$\posDblJacobiSumScalar{\pi}{\chi} = \posDblJacobiSumScalar{\pi'}{\chi} \posDblJacobiSumScalar{\tau \otimes \minusInvolution{\chi}_{\GL_{k}}}{\involutionPlusOne{\chi}}.$$
	Equivalently, using \Cref{thm:gln-doubling-gauss-sum-in-terms-of-kondo}, this can be rewritten as
	$$\posDblJacobiSumScalar{\pi}{\chi} = \left(-\GaussSumSingleCharacter{\minusInvolutionMinusOne{\chi}}{\fieldCharacter_{\quadraticExtension}}\right)^k \GaussSumScalar{\involution{\left(\Contragradient{\tau}\right)} \times \chi}{\fieldCharacter_{\quadraticExtension}} \GaussSumScalar{\tau \times \chi}{\fieldCharacter_{\quadraticExtension}} \posDblJacobiSumScalar{\pi'}{\chi}.$$
\end{theorem}

\subsection{Pairing with Deligne--Lusztig characters}\label{subsubsec:pairing-with-deligne-lusztig}

The goal of the next two sections is to compute $\posDblJacobiSumScalar{\pi}{\chi}$ for $\pi$ that appears in a Deligne--Lusztig virtual representation $\RTThetaVirtual{T_H}{H}{\theta_H}$ for a suitable maximal torus $T_H$ and $\theta_H \colon T_H \to \multiplicativegroup{\cComplex}$.

In this section, we focus on computing the analogous character sum corresponding to $\posDblJacobiSumScalar{\pi}{\chi}$ for the virtual character $\RTGTheta{T_H}{H}{\theta_H}$.

Let us realize $H = \algebraicGroup{H}^{\Frobenius}$, that is, $H$ is the group consisting of the fixed points of the Frobenius map $\Frobenius$ acting on the connected reductive algebraic group $\algebraicGroup{H}$, where $\algebraicGroup{H} = \tilde{\algebraicGroup{G}}$ if $H = \tilde{G}$ and $\algebraicGroup{H} = \algebraicGroup{G}$ if $H = G$.

Given a class function $F \colon H \to \cComplex$, let us define $$\posDblVirtualJacobiSumScalar{F}{\chi} = \frac{1}{\sqrt{\sizeof{\lieAlgebra}}} \sum_{g \in H} F\left(g\right) \posHermitianJacobiKernel{\hermitianSpace}{\chi}\left(g\right) =  \frac{1}{\sqrt{\sizeof{\lieAlgebra}}} \sum_{g \in G} F\left(g\right) \posHermitianJacobiKernel{\hermitianSpace}{\chi}\left(g\right).$$

It is clear that if $\pi$ is an irreducible representation of $H$ then $$\posDblVirtualJacobiSumScalar{\trace \pi}{\chi} = \posDblJacobiSumScalar{\pi}{\chi} \cdot \dim \pi.$$
Notice that the assignment $g \mapsto \posJacobiKernel{\chi}\left(g\right)$ only depends on the semisimple part of $g \in H$. By \Cref{prop:semisimple_pair_with_RTtheta}, we have the following result.
\begin{proposition}\label{prop:reduction-of-gauss-sum-to-torus}
	For any $\Frobenius$-stable rational maximal torus $\algebraicGroup{T}_{\algebraicGroup{H}}$ of $\algebraicGroup{H}$, and any character $\theta_{H} \colon T_{H} \to \multiplicativegroup{\cComplex}$, where $T_H = \algebraicGroup{T}_{\algebraicGroup{H}}^{\Frobenius}$, we have
	$$ \posDblVirtualJacobiSumScalar{\RTTheta{T_H}{\theta_H}}{\chi} = \frac{\grpIndex{H}{T_H}}{\sqrt{\sizeof{\lieAlgebra}}} \sum_{t \in T_H} \theta_H\left(t\right) \posHermitianJacobiKernel{\hermitianSpace}{\chi}\left(t\right).$$
\end{proposition}
In the special case that $\theta_H$ is in general position, the virtual character $\varepsilon_{\algebraicGroup{H}} \varepsilon_{\algebraicGroup{T}_{\algebraicGroup{H}}} \RTTheta{T_H}{\theta_H}$ equals $\trace \pi$ for some irreducible representation $\pi$ of $H$ and we have that for this $\pi$,
$$\posDblJacobiSumScalar{\pi}{\chi} = \varepsilon_{\algebraicGroup{H}} \varepsilon_{\algebraicGroup{T}_{\algebraicGroup{H}}} \frac{\sizeof{H}_p}{\sqrt{\sizeof{\lieAlgebra}}} \sum_{t \in T_H} \theta_H\left(t\right) \posHermitianJacobiKernel{\hermitianSpace}{\chi}\left(t\right).$$
Here $\varepsilon_{\algebraicGroup{H}} = \left(-1\right)^{\mathrm{rel.rank} \algebraicGroup{H}}$ and $\varepsilon_{\algebraicGroup{T}_{\algebraicGroup{H}}} = \left(-1\right)^{\mathrm{rel.rank} \algebraicGroup{T}_{\algebraicGroup{H}}}$, and $\sizeof{H}_p$ is the size of the $p$-Sylow group of $H$. 

\subsubsection{Split torus computation I}\label{subsec:split-torus-computation-1}
Let $\alpha \colon \multiplicativegroup{\quadraticFieldExtension{k}} \to \multiplicativegroup{\cComplex}$ be a character and let $\chi \colon \multiplicativegroup{\quadraticExtension} \to \multiplicativegroup{\cComplex}$ be a non-trivial character.

The goal of this section is to compute $$\sum_{\substack{x \in \multiplicativegroup{\quadraticFieldExtension{k}}\\
x \ne -1}} \alpha\left(x\right) \chi\left(\aFieldNorm_{\quadraticFieldExtension{k} \slash \quadraticExtension}\left(1+x\right)\right).$$
This computation will be useful for the next section, which in turn will be useful for invoking \Cref{prop:reduction-of-gauss-sum-to-torus}.

Let us rewrite our sum as
$$\frac{1}{\sizeof{\quadraticFieldExtension{k}}} \sum_{z \in \quadraticFieldExtension{k}} \sum_{x \in  \multiplicativegroup{\quadraticFieldExtension{k}}} \sum_{y \in \multiplicativegroup{\quadraticFieldExtension{k}}}\alpha\left(x\right) \chi\left(\aFieldNorm_{\quadraticFieldExtension{k} \slash \quadraticExtension}\left(y\right)\right) \fieldCharacter_{\quadraticFieldExtension{k}}\left(\left(y - \left(x + 1\right)\right)z\right).$$
Since $\chi$ is non-trivial, when $z = 0$, the inner sum over $y$ vanishes. Hence, we have that our sum can be rewritten as
$$\frac{1}{\sizeof{\quadraticFieldExtension{k}}} \sum_{z \in \multiplicativegroup{\quadraticFieldExtension{k}}} \sum_{x \in  \multiplicativegroup{\quadraticFieldExtension{k}}} \sum_{y \in \multiplicativegroup{\quadraticFieldExtension{k}}}\alpha\left(x\right) \chi\left(\aFieldNorm_{\quadraticFieldExtension{k} \slash \quadraticExtension}\left(y\right)\right) \fieldCharacter_{\quadraticFieldExtension{k}}\left(\left(y - \left(x + 1\right)\right)z\right).$$
Replacing $x$ with $-z^{-1} x$ and $y$ with $z^{-1} y$ and then $z$ with $-z$, this becomes
$$\frac{\chi\left(-1\right)^k}{\sizeof{\quadraticFieldExtension{k}}} \sum_{z \in \multiplicativegroup{\quadraticFieldExtension{k}}} \alpha^{-1}\left(z\right) \chi^{-1}\left(\aFieldNorm_{\quadraticFieldExtension{k} \slash \quadraticExtension}\left(z\right)\right) \fieldCharacter_{\quadraticFieldExtension{k}}\left(z\right) \sum_{x \in  \multiplicativegroup{\quadraticFieldExtension{k}}} \alpha\left(x\right) \fieldCharacter_{\quadraticFieldExtension{k}}\left(x\right) \sum_{y \in \multiplicativegroup{\quadraticFieldExtension{k}}} \chi\left(\aFieldNorm_{\quadraticFieldExtension{k} \slash \quadraticExtension}\left(y\right)\right) \fieldCharacter_{\quadraticFieldExtension{k}}\left(y\right).$$
Hence, we have the identity
\begin{equation*}
	\begin{split}
		&\sum_{\substack{x \in \multiplicativegroup{\quadraticFieldExtension{k}}\\
				x \ne -1}} \alpha\left(x\right) \chi\left(\aFieldNorm_{\quadraticFieldExtension{k} \slash \quadraticExtension}\left(x+1\right)\right) \\ = &
		- \chi\left(-1\right)^k \sizeof{\quadraticFieldExtension{k}}^{1/2} \GaussSumCharacter{\alpha}{\chi}{\fieldCharacter_{\quadraticFieldExtension{k}}} \GaussSumSingleCharacter{\alpha^{-1}} {\fieldCharacter_{\quadraticFieldExtension{k}}} \GaussSumSingleCharacter{\chi^{-1} \circ \aFieldNorm_{\quadraticFieldExtension{k} \slash \quadraticExtension}}{\fieldCharacter_{\quadraticFieldExtension{k}}}.
	\end{split}
\end{equation*}

\subsubsection{Split torus computation II}
We use the results from the previous section to compute the doubling method Jacobi sum for a torus of the form $\multiplicativegroup{\quadraticFieldExtension{k}}$. Let $\alpha \colon \multiplicativegroup{\quadraticFieldExtension{k}} \to \multiplicativegroup{\cComplex}$ and $\chi \colon \multiplicativegroup{\quadraticExtension} \to \multiplicativegroup{\cComplex}$ be characters such that $\involutionPlusOne{\chi} \ne 1$. Our goal is to compute $$\sum_{\substack{x \in \multiplicativegroup{\quadraticFieldExtension{k}}\\
		x \ne -1}} \alpha \left(x\right) \chi\left(\aFieldNorm_{\quadraticFieldExtension{k} \slash \quadraticExtension}\left(1 + x\right)\right) \chi\left(\aFieldNorm_{\quadraticFieldExtension{k} \slash \quadraticExtension}\left(1 + \minusInvolution{x}\right)\right).$$
This can be rewritten as
$$\sum_{\substack{x \in \multiplicativegroup{\quadraticFieldExtension{k}}\\
		x \ne -1}} \alpha \left(x\right) \minusInvolution{\chi}\left(\aFieldNorm_{\quadraticFieldExtension{k} \slash \quadraticExtension}\left(x\right)\right) \involutionPlusOne{\chi}\left(\aFieldNorm_{\quadraticFieldExtension{k} \slash \quadraticExtension}\left(1 + x\right)\right).$$
By the previous section, this equals
\begin{equation*}
	-\sizeof{\quadraticExtension}^{\frac{k}{2}} \GaussSumCharacter{\alpha}{\chi}{\fieldCharacter_{\quadraticFieldExtension{k}}} \GaussSumCharacter{\alpha^{-1}}{\involution{\chi}}{\fieldCharacter_{\quadraticFieldExtension{k}}} \GaussSumSingleCharacter{\minusInvolutionMinusOne{\chi}}{\fieldCharacter_{\quadraticExtension}}^k,
\end{equation*}
which, using the fact that Gauss sums are invariant under the Frobenius action, equals
\begin{equation*}
	-\sizeof{\quadraticExtension}^{\frac{k}{2}} \GaussSumCharacter{\alpha}{\chi}{\fieldCharacter_{\quadraticFieldExtension{k}}} \GaussSumCharacter{\minusInvolution{\alpha}}{\chi}{\fieldCharacter_{\quadraticFieldExtension{k}}} \GaussSumSingleCharacter{\minusInvolutionMinusOne{\chi}}{\fieldCharacter_{\quadraticExtension}}^k.
\end{equation*}
Thus we obtained the following.
\begin{theorem}\label{thm:split-case-deligne-lusztig-computation}We have the identity
\begin{align*}
	& \sum_{\substack{x \in \multiplicativegroup{\quadraticFieldExtension{k}}\\
			x \ne -1}} \alpha \left(x\right) \chi\left(\aFieldNorm_{\quadraticFieldExtension{k} \slash \quadraticExtension}\left(1 + x\right)\right) \chi\left(\aFieldNorm_{\quadraticFieldExtension{k} \slash \quadraticExtension}\left(1 + \minusInvolution{x}\right)\right) \\
	=& -\sizeof{\quadraticExtension}^{\frac{k}{2}} \GaussSumCharacter{\alpha}{\chi}{\fieldCharacter_{\quadraticFieldExtension{k}}} \GaussSumCharacter{\minusInvolution{\alpha}}{\chi}{\fieldCharacter_{\quadraticFieldExtension{k}}} \GaussSumSingleCharacter{\minusInvolutionMinusOne{\chi}}{\fieldCharacter_{\quadraticExtension}}^k.
\end{align*}
\end{theorem}

\subsubsection{Elliptic torus computation}
In the next two sections we compute the doubling method Jacobi sum for a torus of the form $\NormOneGroup{2m}$.

Let $\theta \colon \NormOneGroup{2m} \to \multiplicativegroup{\cComplex}$ and $\chi \colon \multiplicativegroup{\quadraticExtension} \to \multiplicativegroup{\cComplex}$ be characters, such that $\involutionPlusOne{\chi} \ne 1$. Our goal is to compute $$\sum_{\substack{x \in \NormOneGroup{2m}\\
		x \ne -1}} \theta \left(x\right) \chi\left(\aFieldNorm_{\finiteFieldExtension{2m} \slash \quadraticExtension}\left(1 + x\right)\right).$$

We separate two cases: the case $\quadraticExtension = \finiteField$ and the case $\quadraticExtension \ne \finiteField$.

\subsubsection{Elliptic torus computation: the case $\quadraticExtension = \finiteField$}\label{subsubsec:elliptic-torus-computation-E-equal-F}
In this section, we compute the above exponential sum for the case $\quadraticExtension = \finiteField$. Since $\chi^2 \ne 1$, it follows that $\theta\left(x^{1-q^m}\right) \ne \chi \circ \aFieldNorm_{\finiteFieldExtension{2m} \slash \finiteField}\left(x\right)$ for some $x \in \multiplicativegroup{\finiteFieldExtension{2m}}$.

As usual, we rewrite the sum as follows
$$q^{-2m} \sum_{z \in \finiteFieldExtension{2m}} \sum_{y \in \multiplicativegroup{\finiteFieldExtension{2m}}} \sum_{x \in \NormOneGroup{2m}} \theta \left(x\right) \chi\left(\aFieldNorm_{\finiteFieldExtension{2m} \slash \finiteField}\left(y\right)\right) \fieldCharacter_{2m}\left(z\left(y-\left(1+x\right)\right)\right).$$

When $z=0$, the sum over $y$ vanishes. Hence we can reduce the sum to $z \in \multiplicativegroup{\finiteFieldExtension{2m}}$. Replacing $y$ with $z^{-1} y$ and then $z$ with $-z$, we get the sum
$$-q^{-m} \GaussSumSingleCharacter{\chi^{-1} \circ \aFieldNorm_{\finiteFieldExtension{2m} \slash \finiteField}}{\fieldCharacter_{2m}} \sum_{z \in \multiplicativegroup{\finiteFieldExtension{2m}}} \chi^{-1}\left(\aFieldNorm_{\finiteFieldExtension{2m} \slash \finiteField}\left(z\right)\right) \fieldCharacter_{2m}\left(z\right)  \sum_{x \in \NormOneGroup{2m}} \theta \left(x\right) \fieldCharacter_{2m}\left(xz\right).$$
To proceed, we use the Hilbert 90 map $\multiplicativegroup{\finiteFieldExtension{2m}} \to \NormOneGroup{2m}$ given by $t \mapsto t^{1 - q^m}$. This map is surjective, and its kernel is $\multiplicativegroup{\finiteFieldExtension{m}}$. Define $\transfer{\theta} \colon \multiplicativegroup{\finiteFieldExtension{2m}} \to \multiplicativegroup{\cComplex}$ by $\transfer{\theta}\left(t\right) = \theta\left(t^{1-q^m}\right)$. Replacing the sum over $x \in \NormOneGroup{2m}$ with a sum over $t \in \multiplicativegroup{\finiteFieldExtension{2m}}$, we get \begin{align*}
	& \sum_{z \in \multiplicativegroup{\finiteFieldExtension{2m}}} \chi^{-1}\left(\aFieldNorm_{\finiteFieldExtension{2m} \slash \finiteField}\left(z\right)\right) \fieldCharacter_{2m}\left(z\right) \sum_{x \in \NormOneGroup{2m}} \theta \left(x\right) \fieldCharacter_{2m}\left(xz\right) \\
	= & \frac{1}{q^m-1}\sum_{z \in \multiplicativegroup{\finiteFieldExtension{2m}}} \sum_{t \in \multiplicativegroup{\finiteFieldExtension{2m}}} \chi^{-1}\left(\aFieldNorm_{\finiteFieldExtension{2m} \slash \finiteField}\left(z\right)\right) \fieldCharacter_{2m}\left(z^{q^m}\right) \transfer{\theta} \left(t\right) \fieldCharacter_{2m}\left(t^{1-q^m} z\right).
\end{align*}
Replacing $z$ with $t^{q^m} z$, this becomes
\begin{align*}
	\frac{1}{q^m-1}\sum_{t,z \in \multiplicativegroup{\finiteFieldExtension{2m}}} \chi^{-1}\left(\aFieldNorm_{\finiteFieldExtension{2m} \slash \finiteField}\left(z\right)\right) \chi^{-1}\left(\aFieldNorm_{\finiteFieldExtension{2m} \slash \finiteField}\left(t\right)\right) \transfer{\theta} \left(t\right) \fieldCharacter_{2m}\left(\trace_{\finiteFieldExtension{2m} \slash \finiteFieldExtension{m}}\left(z\right) t\right),
\end{align*}
where we used the fact that $\fieldCharacter_{2m}$ is invariant under the Frobenius action.
Since $\chi \circ \aFieldNorm_{\finiteFieldExtension{2m} \slash \finiteField} \ne \transfer{\theta}$, if $\trace_{\finiteFieldExtension{2m} \slash \finiteFieldExtension{m}}\left(z\right) = 0$, we have that the inner sum over $t$ is zero. Hence, we may reduce the sum over $z$ to $z$ such that $\trace_{\finiteFieldExtension{2m} \slash \finiteFieldExtension{m}}\left(z\right) \ne 0$. Using the fact that $\transfer{\theta}$ is trivial on $\multiplicativegroup{\finiteFieldExtension{m}}$, we get after replacing variables $t \mapsto \frac{t}{\trace_{\finiteFieldExtension{2m} \slash \finiteFieldExtension{m}}\left(z\right)}$,
\begin{align*}
	\frac{1}{q^m-1}\sum_{\substack{t,z \in \multiplicativegroup{\finiteFieldExtension{2m}}\\
			\trace_{\finiteFieldExtension{2m} \slash \finiteFieldExtension{m}}\left(z\right) \ne 0}} \chi^{-1}\left(\aFieldNorm_{\finiteFieldExtension{2m} \slash \finiteField}\left(\frac{z}{\trace_{\finiteFieldExtension{2m} \slash \finiteFieldExtension{m}}\left(z\right) }\right)\right) \chi^{-1}\left(\aFieldNorm_{\finiteFieldExtension{2m} \slash \finiteField}\left(t\right)\right) \transfer{\theta} \left(t\right) \fieldCharacter_{2m}\left(t\right).
\end{align*}
By \Cref{appendix:e-equals-f}, this equals
\begin{align*}
	q^{\frac{m}{2}} \GaussSumSingleCharacter{\chi \circ \aFieldNorm_{\finiteFieldExtension{m} \slash \finiteField}}{\fieldCharacter_m}^2 \GaussSumSingleCharacter{\chi^{-2} \circ \aFieldNorm_{\finiteFieldExtension{m} \slash \finiteField}}{\fieldCharacter_m} \left(\sum_{t \in \multiplicativegroup{\finiteFieldExtension{2m}}} \chi^{-1}\left(\aFieldNorm_{\finiteFieldExtension{2m} \slash \finiteField}\left(t\right)\right) \transfer{\theta} \left(t\right) \fieldCharacter_{2m}\left(t\right)\right).
\end{align*}

To summarize, we have
\begin{align*}
	& \sum_{\substack{x \in \NormOneGroup{2m}\\
			x \ne -1}} \theta \left(x\right) \chi\left(\aFieldNorm_{\finiteFieldExtension{2m} \slash \finiteField}\left(1 + x\right)\right) \\
		=& q^{\frac{m}{2}} \GaussSumSingleCharacter{\chi}{\fieldCharacter}^{2m} \GaussSumSingleCharacter{\chi^{-1}}{\fieldCharacter}^{2m} \GaussSumSingleCharacter{\chi^{-2}}{\fieldCharacter}^m \GaussSumCharacter{\left(\transfer{\theta}\right)^{-1}}{\chi}{\fieldCharacter_{2m}},
\end{align*}
where $\transfer{\theta}\left(t\right) = \theta\left(t^{1-q^m}\right)$.
Since $\tau\left(\chi^{-1}, \fieldCharacter\right) = \chi\left(-1\right) \conjugate{\tau\left(\chi, \fieldCharacter\right)}$ and since $\chi \ne 1$, we have that $\tau\left(\chi, \fieldCharacter\right) \tau\left(\chi^{-1}, \fieldCharacter\right) = \chi\left(-1\right)$. Combining this with the fact that $\left(\transfer{\theta}\right)^{-1} = \left(\transfer{\theta}\right)^{q^m}$ lies in the same Frobenius orbit as $\transfer{\theta}$, we get the following theorem.
\begin{theorem}\label{thm:computation-norm-one-case}
	We have the identity $$\sum_{\substack{x \in \NormOneGroup{2m}\\
			x \ne -1}} \theta \left(x\right) \chi\left(\aFieldNorm_{\finiteFieldExtension{2m} \slash \finiteField}\left(1 + x\right)\right) = q^{\frac{m}{2}} \GaussSumSingleCharacter{\chi^{-2}}{\fieldCharacter}^m \GaussSumCharacter{\transfer{\theta}}{\chi}{\fieldCharacter_{2m}}.$$
\end{theorem}

\subsubsection{Elliptic torus computation: the case $\quadraticExtension \ne \finiteField$}
In this section, we compute the above exponential sum for the case $\quadraticExtension = \finiteFieldExtension{2}$. Let $\chi \colon \multiplicativegroup{\finiteFieldExtension{2}} \to \multiplicativegroup{\cComplex}$. Let $m$ be an odd positive integer. We wish to compute the following sum.
$$\sum_{\substack{x \in \NormOneGroup{2m}\\
		x \ne -1}} \theta \left(x\right) \chi\left(\aFieldNorm_{\finiteFieldExtension{2m} \slash \finiteFieldExtension{2}}\left(1 + x\right)\right).$$
By \Cref{appendix:e-not-equal-f}, this equals
$$q^{\frac{m}{2}} \chi\left(-1\right)^m \GaussSumCharacter{\transfer{\theta}}{\chi}{\fieldCharacter_{2m}} \tau\left(\chi^{-1} \circ \FieldNorm{2m}{2} \restriction_{\multiplicativegroup{\finiteFieldExtension{m}}}, \fieldCharacter_m\right).$$
Since $m$ is odd, $\chi^{-1} \circ \FieldNorm{2m}{2} \restriction_{\multiplicativegroup{\finiteFieldExtension{m}}} = \chi^{-1}\restriction_{\multiplicativegroup{\finiteFieldExtension{m}}} \circ \FieldNorm{m}{1}$ and thus by the Hasse--Davenport relation we get
$$q^{\frac{m}{2}} \chi\left(-1\right) \tau\left(\chi^{-1} \restriction_{\multiplicativegroup{\finiteField}}, \fieldCharacter\right)^m \GaussSumCharacter{\transfer{\theta}}{\chi}{\fieldCharacter_{2m}},$$
where we recall that $\transfer{\theta} \colon \multiplicativegroup{\finiteFieldExtension{2m}} \to \multiplicativegroup{\cComplex}$ is given by $\transfer{\theta}\left(x\right) = \theta\left(x^{1-q^m}\right)$.

Thus, we obtained the following.
\begin{theorem}\label{thm:computation-unitary-norm-one-case}
	We have the following identity $$\sum_{\substack{x \in \NormOneGroup{2m}\\
			x \ne -1}} \theta \left(x\right) \chi\left(\aFieldNorm_{\finiteFieldExtension{2m} \slash \finiteFieldExtension{2}}\left(1 + x\right)\right) = q^{\frac{m}{2}} \chi\left(-1\right) \GaussSumSingleCharacter{\chi^{-1} \restriction_{\multiplicativegroup{\finiteField}}}{\fieldCharacter}^m \GaussSumCharacter{\transfer{\theta}}{\chi}{\fieldCharacter_{2m}}.$$
\end{theorem}

\subsubsection{Gamma factors}\label{subsec:normalization-factor}
The multiplicativity property (\Cref{thm:multiplicativity-in-terms-of-gauss-sums}) and the computations performed above (Theorems \ref{thm:split-case-deligne-lusztig-computation}-\ref{thm:computation-unitary-norm-one-case}) suggest to normalize our Jacobi sums correctly so that they will satisfy a full multiplicativity property.

Let $\chi \colon \multiplicativegroup{\quadraticExtension} \to \multiplicativegroup{\cComplex} $ be a character. We define a normalization factor as follows. $$c_{\hermitianSpace}\left(\chi, \fieldCharacter\right) = \varepsilon_{\algebraicGroup{G}} \cdot \begin{dcases}
	\GaussSumSingleCharacter{\chi^{-2}}{\fieldCharacter}^{\frac{\dim_{\finiteField} \hermitianSpace}{2}} & \quadraticExtension = \finiteField,\,\, \dim_{\finiteField} \hermitianSpace \text{ is even,}\\
	\chi\left(2\right) \GaussSumSingleCharacter{\chi^{-2}}{\fieldCharacter}^{\frac{\dim_{\finiteField} \hermitianSpace - 1}{2}} & \quadraticExtension = \finiteField,\,\, \dim_{\finiteField} \hermitianSpace \text{ is odd,}	\\
	\chi\left(-1\right)^{\dim_{\quadraticExtension} \hermitianSpace} \GaussSumSingleCharacter{\chi^{-1}\restriction_{\multiplicativegroup{\finiteField}}}{\fieldCharacter}^{\dim_{\quadraticExtension} \hermitianSpace} & \quadraticExtension \ne \finiteField.
\end{dcases}$$
Notice that if $\quadraticExtension \ne \finiteField$ then by the Hasse--Davenport relation \begin{equation}\label{eq:hasse-davenport-for-normalizing-factor}
	\GaussSumSingleCharacter{\chi^{-1}\restriction_{\multiplicativegroup{\finiteField}}}{\fieldCharacter}^{2} = \GaussSumSingleCharacter{\minusInvolutionMinusOne{\chi}} {\fieldCharacter_{\quadraticExtension}}.
\end{equation} The factor on the left hand side of \eqref{eq:hasse-davenport-for-normalizing-factor} appears in \Cref{thm:computation-unitary-norm-one-case}, while the one on the right hand side of \eqref{eq:hasse-davenport-for-normalizing-factor} appears in \Cref{thm:multiplicativity-in-terms-of-gauss-sums}. The factor $\chi\left(2\right)$ in the odd-dimensional special orthogonal case will be explained in the proof of \Cref{thm:computation-of-doubling-gauss-sum-scalar-for-deligne-lusztig-characters}.

For an irreducible representation $\tau$ of $\GL_k\left(\quadraticExtension\right)$ define $$\dblGammaFactor{\tau}{\chi}{\fieldCharacter} = \GaussSumScalar{\Contragradient{\tau} \times \involution{\chi}}{\fieldCharacter} \GaussSumScalar{\tau \times \chi}{\fieldCharacter}.$$
For an irreducible representation $\pi$ of $H$, define the \emph{doubling method gamma factor}\footnote{We warn the reader that \emph{this is not} $\Gamma^{\mathrm{dbl}}\left(\pi, \chi\right)$ from \cite{GirschZelingher2026}.} $$\dblGammaFactorSpace{\hermitianSpace}{\pi}{\chi}{\fieldCharacter} = \frac{\posDblJacobiSumScalar{\pi}{\chi}}{c_{\hermitianSpace}\left(\chi, \fieldCharacter\right)}.$$

Using this notation and using the multiplicativity property of Kondo's Gauss sum and \Cref{thm:multiplicativity-in-terms-of-gauss-sums}, we have the following multiplicativity property.

\begin{theorem}\label{thm:multiplicativity-in-terms-of-gamma-factors}
	\begin{enumerate}
		\item If $\tau_1$ and $\tau_2$ are irreducible representations of $\GL_{k_1}\left(\finiteField\right)$ and $\GL_{k_2}\left(\finiteField\right)$, respectively, then for every irreducible subrepresentation $\tau$ of the parabolic induction $\tau_1 \circ \tau_2$, we have
		$$\dblGammaFactor{\tau}{\chi}{\fieldCharacter} = \dblGammaFactor{\tau_1}{\chi}{\fieldCharacter} \dblGammaFactor{\tau_2}{\chi}{\fieldCharacter}.$$
		\item If $\tau$, $\pi'$ and $\pi$ are as in \Cref{thm:multiplicativity-in-terms-of-gauss-sums} and $\involutionPlusOne{\chi} \ne 1$, then
		$$\dblGammaFactorSpace{\hermitianSpace}{\pi}{\chi}{\fieldCharacter} = \dblGammaFactorSpace{\hermitianSpace'}{\pi'}{\chi}{\fieldCharacter} \dblGammaFactor{\tau}{\chi}{\fieldCharacter}.$$
	\end{enumerate}
\end{theorem}

\begin{remark}
	The ``correct'' gamma factor (expected by Langlands functoriality) is given by $$\dblLanglandsGammaFactor{\pi}{\chi}{\fieldCharacter} = \dblGammaFactorSpace{\hermitianSpace}{\pi}{\chi}{\fieldCharacter} \cdot \begin{dcases}
		1 & \quadraticExtension = \finiteField \text{ and } \epsilon_{\hermitianSpace} = 1\\
		-\GaussSumSingleCharacter{\chi} {\fieldCharacter} & \quadraticExtension = \finiteField \text{ and } \epsilon_{\hermitianSpace} = -1 \\
		\left(-1\right)^{\dim_{\quadraticExtension} \hermitianSpace} & \text{otherwise.}
	\end{dcases}$$
\end{remark}

\subsubsection{Some notation}\label{subsec:some-notation}
Recall that by \Cref{subsec:rational-tori} maximal rational tori of $\algebraicGroup{G}$ correspond to certain pairs of partitions $\left(\lambda^+, \lambda^-\right)$ and we have that the $\Frobenius$ fixed points of a maximal torus $\algebraicGroup{T}$ corresponding to the pair $\left(\lambda^+, \lambda^-\right)$ are of the form $$T = \algebraicGroup{T}^{\Frobenius} \cong \prod_{j=1}^{\lengthof(\lambda^+)} \multiplicativegroup{\finiteFieldExtension{\lambda_j^{+}}} \times \prod_{i=1}^{\lengthof(\lambda^-)} \NormOneGroup{2 \lambda_i^{-}}.$$
Recall that if $\algebraicGroup{G}$ is a unitary group then $\lambda_j^{+}$ is even for every $j$, and thus $\finiteFieldExtension{\lambda^+_j}$ is always a field extension of $\quadraticExtension$.

Let us choose a torus $\transfer{\algebraicGroup{T}}$ of $\restrictionOfScalars{\quadraticExtension}{\finiteField}{\algebraicGroup{\GL}_{\frac{2}{\grpIndex{\quadraticExtension}{\finiteField}} \left\lfloor\frac{\dim_{\finiteField} \hermitianSpace}{2}\right\rfloor}}$ such that the torus $\transfer{T} \coloneq \left(\transfer{\algebraicGroup{T}}\right)^{\Frobenius}$ of $\GL_{\frac{2}{\grpIndex{\quadraticExtension}{\finiteField}} \left\lfloor\frac{\dim_{\finiteField} \hermitianSpace}{2}\right\rfloor}\left(\quadraticExtension\right)$ is of the form $$\transfer{T} \cong \prod_{j=1}^{\lengthof(\lambda^+)} \left(\multiplicativegroup{\finiteFieldExtension{\lambda_j^{+}}} \times \multiplicativegroup{\finiteFieldExtension{\lambda_j^{+}}}\right) \times \prod_{i=1}^{\lengthof(\lambda^-)} \multiplicativegroup{\finiteFieldExtension{2 \lambda_i^{-}}}.$$
Given a character $$\theta = \prod_{j=1}^{\lengthof(\lambda^+)} \alpha_j \times \prod_{i=1}^{\lengthof(\lambda^-)} \theta_i \colon \algebraicGroup{T}^{\Frobenius} \to \multiplicativegroup{\cComplex},$$ we define a character $\transfer{\theta} \colon \transfer{T} \to \multiplicativegroup{\cComplex}$ by the formula
$$\transfer{\theta} = \prod_{j=1}^{\lengthof(\lambda^+)} \left(\alpha_j \times \minusInvolution{\alpha}_j\right) \times \prod_{i=1}^{\lengthof(\lambda^-)} \transfer{\theta}_i,$$
where we recall that $\transfer{\theta}_i \colon \multiplicativegroup{\finiteFieldExtension{2 \lambda_i^{-}}} \to \multiplicativegroup{\cComplex}$ is given by $$\transfer{\theta}_i\left(x\right) = \theta_i\left(x^{1-q^{\lambda_i^-}}\right).$$

\subsubsection{Computation for Deligne--Lusztig characters}
Combining the results of the previous sections with \Cref{prop:reduction-of-gauss-sum-to-torus}, we arrive at the following result.

\begin{theorem}\label{thm:computation-of-doubling-gauss-sum-scalar-for-deligne-lusztig-characters}
	Suppose that $\algebraicGroup{T}_{\algebraicGroup{H}}$ is a maximal rational torus of $\algebraicGroup{H}$ that corresponds to the pair of partitions $\left(\lambda^+, \lambda^-\right)$ and let $\algebraicGroup{T}_{\algebraicGroup{G}} \coloneq \algebraicGroup{G} \cap \algebraicGroup{T}_{\algebraicGroup{H}}$. Denote $T_H = \algebraicGroup{T}_{\algebraicGroup{H}}^{\Frobenius}$ and $T = T_G = \algebraicGroup{T}_{\algebraicGroup{G}}^{\Frobenius}$. Suppose that $\theta_{H} \colon T_H \to \multiplicativegroup{\cComplex}$ is a character and let $$\theta = \theta_{G} \coloneq \theta_{H} \restriction_{T_G} \colon T_G \to \multiplicativegroup{\cComplex}.$$ Then
	\begin{align*}
		 \dblPosVirtualJacobiSumScalar{\RTGTheta{T_H}{H}{\theta_H}}{\chi} = \RTGTheta{T_H}{H}{\theta_H}\left(1\right) c_{\hermitianSpace}\left(\chi, \fieldCharacter\right) \GaussSumTorusCharacter{\transfer{T}}{\transfer{\theta}}{\chi}{\fieldCharacter_{\quadraticExtension}}.
	\end{align*}
\end{theorem}
\begin{proof}	
	We have that if $t \in T_G$ corresponds to $\left(x,y\right) \in \multiplicativegroup{\finiteFieldExtension{\lambda^+}} \times \NormOneGroup{2 \lambda^{-}}$ such that $\idmap_{\hermitianSpace} + t$ is invertible then \begin{align*}
		\posJacobiKernel{\chi}\left(t\right) =& \prod_{i=1}^{\lengthof(\lambda^-)} \chi\left(\aFieldNorm_{\finiteFieldExtension{2\lambda_j^-} \slash \quadraticExtension}\left(1 + y_i\right)\right) \cdot \prod_{j=1}^{\lengthof(\lambda^+)} \chi\left(\aFieldNorm_{\finiteFieldExtension{\lambda_j^+} \slash \quadraticExtension}\left(1 + x_j\right)\right) \chi\left(\aFieldNorm_{\finiteFieldExtension{\lambda_j^+} \slash \quadraticExtension}\left(1 + \minusInvolution{x_j}\right)\right) \\
		& \times  \begin{dcases}
			\chi\left(2\right) & \dim_{\finiteField} \hermitianSpace \text{ is odd},\\
			1 & \text{otherwise}
		\end{dcases}.
	\end{align*}

	Hence our theorem follows from the results mentioned above and from the formula $$\RTGTheta{T_H}{H}{\theta_H}\left(1\right) = \varepsilon_{\algebraicGroup{H}} \varepsilon_{\algebraicGroup{T}_{\algebraicGroup{H}}} \frac{\grpIndex{H}{T_H}}{\sizeof{H}_p},$$ the fact that $$ \sizeof{\lieAlgebra} = \sizeof{H}_p^2 \cdot q^{\left\lfloor\frac{\dim_{\finiteField} \hermitianSpace}{2}\right\rfloor},$$
	the fact that \begin{equation*}
		\varepsilon_{\algebraicGroup{H}} \varepsilon_{\algebraicGroup{T}_{\algebraicGroup{H}}} = \varepsilon_{\algebraicGroup{G}} \varepsilon_{\algebraicGroup{T}_{\algebraicGroup{G}}},
	\end{equation*}
	and the fact
	\begin{align*}
		\varepsilon_{\algebraicGroup{T}_{\algebraicGroup{G}}} = \left(-1\right)^{\lengthof\left(\lambda^+\right)}.
	\end{align*}
	In this computation, the factor $\varepsilon_{\algebraicGroup{T}_{\algebraicGroup{G}}}$ comes from \Cref{thm:split-case-deligne-lusztig-computation} while the factor $\varepsilon_{\algebraicGroup{G}}$ comes from the definition of the normalization factor $c_{\hermitianSpace}\left(\chi, \fieldCharacter\right)$.
\end{proof}

\subsection{Invariance of $\innerproduct{\posHermitianJacobiKernel{\hermitianSpace}{\chi}}{\RTTheta{T}{\theta}}$ under geometric conjugacy}\label{subsec:invariance-under-geometric-conjugacy}

Let $f \colon \algebraicGroup{G}^{\Frobenius} \to \cComplex$ be a class function. Then by Schur's lemma for every irreducible representation $\pi$ of $\algebraicGroup{G}^{\Frobenius}$, there exists $\gamma_{f}\left(\pi\right) \in \cComplex$ such that $$\sum_{g \in \algebraicGroup{G}^{\Frobenius}} f(g)\pi(g) = \gamma_{f}\left(\pi\right) \cdot \idmap_{\pi}.$$

\begin{definition}
	We say that the class function $f \colon \algebraicGroup{G}^{\Frobenius} \to \cComplex$ is a \emph{stable function} on $\algebraicGroup{G}^{\Frobenius}$ if it satisfies the following property: if $\pi$ and $\pi'$ are irreducible representations of $\algebraicGroup{G}^{\Frobenius}$ that lie in the same geometric Lusztig series, then $$\gamma_{f}\left(\pi\right) = \gamma_{f}\left(\pi'\right).$$
(See also \cite[Section 4]{LaumonLetellier2023} and \cite[Section 4.1]{ChenBhattacharya2024})
	\end{definition}

Here is a basic example of a stable class function.
\begin{proposition}
\label{prop:central_char}
If $f \colon \algebraicGroup{G}^{\Frobenius} \to \cComplex$ is supported on $\algebraicGroup{Z}^{\Frobenius}$, where $\algebraicGroup{Z}$ is the center of $\algebraicGroup{G}$, then $f$ is a stable function.
\end{proposition}
\begin{proof}
By \cite[Lemma 2.2]{Malle2007}, if $\pi$ and $\pi'$ lie in the same geometric Lusztig series, then the restrictions of their central characters to $\algebraicGroup{Z}^{\Frobenius}$ coincide. The result now immediately follows.
\end{proof}

One of the main results of this paper is the following theorem.
\begin{theorem}
\label{thm:Phi_stable}
	The assignment $H \to \cComplex$ given by $$h \mapsto \posHermitianJacobiKernel{\hermitianSpace}{\chi}\left(h\right)$$ is a stable function. Equivalently, the assignment $$\pi \mapsto \posDblJacobiSumScalar{\pi}{\chi}$$ is constant on geometric Lusztig series.
\end{theorem}
This result will be proved in \Cref{subsec:stability}. Our next goal is to prove the following intermediate result.
\begin{theorem}
\label{thm:indep_geo_conj}
    If $(\algebraicGroup{T}_H,\theta_H)$ and $(\algebraicGroup{T}'_H, \theta'_H)$ are geometrically conjugate, then $$\frac{\posDblVirtualJacobiSumScalar{\RTTheta{T_H}{\theta_H}}{\chi}}{\RTTheta{T_H}{\theta}(1)} = \frac{\posDblVirtualJacobiSumScalar{\RTTheta{T'_H}{\theta'_H}}{\chi}}{\RTTheta{T'_H}{\theta'_H}\left(1\right)}.$$
\end{theorem}
Although will we only use restriction of scalars from $\quadraticExtension$,
we recall in general that $$
	\restrictionOfScalars{\finiteFieldExtension{m}}{\finiteField}{\algebraicGroup{\GL}_N} = \prod_{a \in \mathbb{Z}/m\mathbb{Z}} \left(\algebraicGroup{\GL}_N\right)_a,
$$
with Frobenius root action $$\Frobenius(g_0,g_1, \hdots, g_{m-1}) = (\Frobenius(g_{m-1}), \Frobenius(g_0), \Frobenius(g_1), \hdots, \Frobenius(g_{m-2})).$$ We will freely identify $\restrictionOfScalars{\finiteFieldExtension{m}}{\finiteField}{\algebraicGroup{\GL}_N}(\finiteField) = \GL_N(\finiteFieldExtension{m})$ via the map $\left(g_0,\dots,g_{m-1}\right) \mapsto g_0$.

\begin{remark}
Torus character pairs for $\restrictionOfScalars{\finiteFieldExtension{m}}{\finiteField}{\algebraicGroup{\GL}_N}$ are the same as torus character pairs for $\algebraicGroup{\GL}_N$ as a group over $\finiteFieldExtension{m}$.
\end{remark}

Let $(\algebraicGroup{T},\theta)$ be a torus character pair associated with the $\Frobenius$-twisted conjugacy class $(w) \in W(\algebraicGroup{T}_0)$. Let $\rho:\algebraicGroup{G}^* \to \restrictionOfScalars{\finiteFieldExtension{m}}{\finiteField}{\algebraicGroup{\GL}_N}$ be an embedding and let $\algebraicGroup{T}_{0,\rho} \subset \restrictionOfScalars{\finiteFieldExtension{m}}{\finiteField}{\algebraicGroup{\GL}_N}$ be a Frobenius stable maximal torus such that $\rho(\algebraicGroup{T}_0^*) \subset \algebraicGroup{T}_{0,\rho}^*$. Let $\tilde{\algebraicGroup{T}}$ be the torus of $\restrictionOfScalars{\finiteFieldExtension{m}}{\finiteField}{\algebraicGroup{\GL}_N}$ corresponding to the $\Frobenius$-twisted conjugacy class of $\rho(w) \in W(\algebraicGroup{T}_{0,\rho})$ and let $\tilde{\theta} \colon \tilde{\algebraicGroup{T}}^{\Frobenius} \to \multiplicativegroup{\cComplex}$ be a character such that the equality $(\rho(s_{\algebraicGroup{T},\theta})) = (s_{\tilde{\algebraicGroup{T}},\tilde{\theta}})$ of Frobenius semisimple conjugacy classes holds. Given a character $\chi \colon \multiplicativegroup{\finiteFieldExtension{m}} \to \multiplicativegroup{\cComplex}$, we define
\[
    g_{\algebraicGroup{T},\rho}(\chi,\theta, \fieldCharacter_{m}) := \tau_{\Tilde{\algebraicGroup{T}}}(\tilde{\theta} \times \chi,\fieldCharacter_{m}) = \varepsilon_{\Tilde{\algebraicGroup{T}}} q^{-Nm/2}\sum_{t \in \tilde{\algebraicGroup{T}}^{\Frobenius}} \tilde{\theta}^{-1}(t)\chi^{-1}({\det}_{\slash \finiteFieldExtension{m}}t) \fieldCharacter_m(\trace_{\slash \finiteFieldExtension{m}}t).
\]

In the case that $\mathbb{E} = \finiteField$, $g_{\algebraicGroup{T},\rho}(\chi,\theta, \fieldCharacter_{m})$ is a $\chi$-twisted version of the Braverman--Kazhdan $\gamma$-factors (see \cite[Introduction]{BravermanKazhdan2003}). The argument that $g_{\algebraicGroup{T},\rho}(\chi,\theta, \fieldCharacter_{m})$ is well defined and independent of geometric conjugacy follows analogously to the argument in \cite[Lemma 1.5]{BravermanKazhdan2003}. For the convenience of the reader, we include an alternative proof.

\begin{lemma}
\label{lemma:gl_invariance}
\begin{enumerate}[(i)]
	\item \label{item:gl-n-hasse-davenport-gauss-sum-implies-well-defined} In the case that $\rho$ is the identity map, $g_{\algebraicGroup{T},\rho}(\chi,\theta,\fieldCharacter_{m})$ depends only on the semisimple conjugacy class of $(s_{\algebraicGroup{T},\theta})$ and the character $\chi$.
	\item The value $g_{\algebraicGroup{T},\rho}(\chi,\theta,\fieldCharacter_{m})$ is well defined, i.e., it is independent of the choices $\algebraicGroup{T}_0$, $\algebraicGroup{T}^*$ and Frobenius twisted conjugacy class representative $w$.
\end{enumerate}
\end{lemma}
\begin{proof}
\begin{enumerate}[(i)]
	\item In this case, $T \cong \multiplicativegroup{\finiteFieldExtension{\lambda}}$ for $\lambda \vdash N$, and we recover the Gauss sum
	\[
	g_{T,\rho}(\chi,\theta,\fieldCharacter_{m}) = \GaussSumTorusCharacter{T}{\theta}{\chi}{\fieldCharacter_{m}} = \left(-1\right)^N \GaussSumScalar{\tau \times \chi}{\fieldCharacter_{m}},
	\]
	where $\tau$ is any irreducible subrepresentation of $\RTThetaVirtual{T}{\GL_N\left(\finiteField\right)}{\theta}$. By \Cref{rem:kondo-sum-implies-geometric-conjugacy-invariance}, this only depends on the geometric conjugacy class of $\left(\algebraicGroup{T},\theta\right)$ and hence only on the semisimple conjugacy class of $\left(s_{\algebraicGroup{T},\theta}\right)$.
	\item For a Frobenius stable semisimple conjugacy class $\left(s\right)$ in $\GL_N$, we define
	$$
	g_{(s)}(\chi,\fieldCharacter_{m}) := g_{T,\rho}(\chi,\theta,\fieldCharacter_{m})$$
	where $(\algebraicGroup{T},\theta)$ is a torus character pair such that $\left(s_{\algebraicGroup{T},\theta}\right) = \left(s\right)$.
	By (\ref{item:gl-n-hasse-davenport-gauss-sum-implies-well-defined}) this is well defined. Then
	$$
	g_{T,\rho}(\chi,\theta,\fieldCharacter_m)  = g_{(\rho(s_{\algebraicGroup{T},\theta}))}(\chi,\fieldCharacter_m),
	$$
	which is independent of the choices $\algebraicGroup{T}^* \subset \algebraicGroup{G}^*$ and $\algebraicGroup{T} \subset \algebraicGroup{\GL}_N$.
\end{enumerate}
\end{proof}

\begin{corollary}
\label{cor:geoconj}
Suppose that $(\algebraicGroup{T},\theta)$ and $(\algebraicGroup{T}',\theta')$ are geometrically conjugate torus character pairs in $G$. Then $g_{\algebraicGroup{T},\rho}(\chi,\theta,\fieldCharacter_m) = g_{\algebraicGroup{T}',\rho}(\chi,\theta',\fieldCharacter_m)$.
\end{corollary}
\begin{proof}
Both are equal to $$g_{(\rho(s_{\algebraicGroup{T},\theta}))}\left(\chi, \fieldCharacter_m\right) = g_{(\rho(s_{\algebraicGroup{T}',\theta'}))}\left(\chi, \fieldCharacter_m\right).$$
\end{proof}

The idea behind the proof of \Cref{thm:indep_geo_conj} is that the formula from \Cref{thm:computation-of-doubling-gauss-sum-scalar-for-deligne-lusztig-characters} relates $\posDblVirtualJacobiSumScalar{\RTGTheta{T_H}{H}{\theta_H}}{\chi}/\RTGTheta{T_H}{H}{\theta_H}\left(1\right)$ to $g_{T,\rho}\left(\chi,\theta,\fieldCharacter_m\right)$ for the appropriate standard representation $\rho$ in the table below. 

\begin{center}
	\begin{tabular}{|c|c|} \hline
		$\algebraicGroup{G}$ & $\rho \colon \algebraicGroup{G}^{\ast} \to \restrictionOfScalars{\finiteFieldExtension{m}}{\finiteField}{\algebraicGroup{\GL}_N}$  \tabularnewline \hline \hline
		$\algebraicGroup{GL}_n$ &  $\algebraicGroup{\GL}_n \to \algebraicGroup{GL}_{2n}$ \tabularnewline \hline
		$\algebraicGroup{\UnitaryGroup}_{n}$ &  $\algebraicGroup{\UnitaryGroup}_{n} \to \restrictionOfScalars{\quadraticExtension}{\finiteField}{\algebraicGroup{GL}_n}$\tabularnewline \hline		 				 
		$\algebraicGroup{\SO}_{2n+1}$ &$\algebraicGroup{\Sp}_{2n} \to \algebraicGroup{\GL}_{2n}$ \tabularnewline \hline
		$\algebraicGroup{\Sp}_{2n}$  &$\algebraicGroup{\SO}_{2n+1}\to \algebraicGroup{\GL}_{2n+1}$\tabularnewline \hline		 
		$\algebraicGroup{\SO}^{+}_{2n}$ &$\algebraicGroup{\SO}^{+}_{2n}\to \algebraicGroup{\GL}_{2n}$\tabularnewline \hline
		$\algebraicGroup{\SO}^{-}_{2n}$ &$\algebraicGroup{\SO}^{-}_{2n}\to \algebraicGroup{\GL}_{2n}$\tabularnewline \hline
		$\algebraicGroup{\GSp}_{2n}$  &$\algebraicGroup{GSpin}_{2n+1}\to \algebraicGroup{\GL}_{2n+1}$ \tabularnewline \hline
		$\algebraicGroup{\GSO}^{+}_{2n}$ & $\algebraicGroup{GSpin}_{2n}^{+}\to\algebraicGroup{\GL}_{2n}$ \tabularnewline \hline
		$\algebraicGroup{\GSO}^{-}_{2n}$& $\algebraicGroup{GSpin}_{2n}^{-}\to\algebraicGroup{\GL}_{2n}$\tabularnewline \hline
	\end{tabular}
\end{center}
In each of the cases in the table, $m = 1$, except for the unitary group case where $m=2$. In each of the non-$\operatorname{GSpin}$ cases, the map on cocharacters $\rho_{\ast}:\CocharacterLattice{\algebraicGroup{T}^{\ast}} \cong (\mathbb{Q}/\mathbb{Z})_{p'}^n \to \CocharacterLattice{\tilde{\algebraicGroup{T}}^{\ast}} \cong (\mathbb{Q}/\mathbb{Z})_{p'}^{2n}$ given by $x \to (x,-w_n x)$. 

We will prove \Cref{thm:indep_geo_conj} by checking that if $\rho(s_{\algebraicGroup{T},\theta}) = (s_{\Tilde{\algebraicGroup{T}},\Tilde{\theta}})$ then $g_{T,\rho}(\chi,\theta,\fieldCharacter_m)$ equals to $\GaussSumTorusCharacter{\tilde{T}}{\tilde{\theta}}{\chi}{\fieldCharacter_m}$ up to some factors. We will do so by breaking our tori into irreducible rational factors with the following lemma handling each possible factor.

\begin{lemma}
\label{lem:cochar_computation}
    Let $\phi:X_*(\algebraicGroup{T}^*) \cong (\mathbb{Q}/\mathbb{Z})_{p'}^n \to X_*(\tilde{\algebraicGroup{T}}^*) \cong (\mathbb{Q}/\zIntegers)_{p'}^{2n}$ be the $\Frobenius$ equivariant map sending $x$ to $(x,-w_n x)$. Then:
    \begin{enumerate}
        \item (B/C/D Case 1) $w$ is conjugate to a cycle of length $n$. In this case $\algebraicGroup{T} = \restrictionOfScalars{\finiteFieldExtension{n}}{\finiteField}{\multiplcativeScheme}$ with $\Frobenius$ acting by $w$ on the domain and by the cycle $\sigma=\diag(w,w_n w w_n)$ on the image, and $$x_{\finiteFieldExtension{n}^\times \times \finiteFieldExtension{n}^\times,\alpha \times \alpha^{-1}} = \phi\left(x_{\multiplicativegroup{\finiteFieldExtension{n}},\alpha}\right).$$
		\item (B/C/D Case 2) $w$ is conjugate to a cycle of length $n$ with one negative sign.
 In this case $\algebraicGroup{T} = \restrictionOfScalars{\finiteFieldExtension{n}}{\finiteField}{\algebraicGroup{\UnitaryGroup}_1}$ with $\Frobenius$ acting by $w$ on the domain and by a cycle $\sigma$ of length $2n$ on the image, and $$x_{\multiplicativegroup{\finiteField_{2n}},\transfer{\theta}}=\phi \left(x_{\NormOneGroup{2n},\theta}\right)$$ where $\transfer{\theta}\left(t\right) = \theta\left(t^{1-q^n}\right).$
    \item (A case 1) $\algebraicGroup{G}=\algebraicGroup{\UnitaryGroup}_n$ and $w$ is conjugate to cycle of length $n$, where $n$ is even. In this case $\algebraicGroup{T} = \restrictionOfScalars{\finiteFieldExtension{n}}{\finiteField}{\multiplcativeScheme}$ with $\Frobenius$ acting by $-1\cdot w$ on the domain and by the cycle $\sigma = \left(\begin{smallmatrix} 0 & w_nw \\
        w w_n & 0
        \end{smallmatrix}\right)$ on the image, and 
		$$x_{\multiplicativegroup{\finiteFieldExtension{n}} \times \multiplicativegroup{\finiteFieldExtension{n}},\alpha \times \alpha^{-q}} = \phi\left(x_{\multiplicativegroup{\finiteFieldExtension{n}},\alpha}\right).$$
		Here $\left(\multiplicativegroup{\finiteFieldExtension{n}} \times \multiplicativegroup{\finiteFieldExtension{n}},\alpha \times \alpha^{-q}\right)$ is realized explicitly as follows. For $z = \left(x,y\right) \in \multiplicativegroup{\finiteFieldExtension{n}} \times \multiplicativegroup{\finiteFieldExtension{n}}$ denote $Dz = \left(\begin{smallmatrix}
			x & 0\\
			0 & y
		\end{smallmatrix}\right) \in \GL_2\left(\finiteFieldExtension{n}\right)$and $$\Frobenius\left(z\right) = \left(\Frobenius x, \Frobenius y\right) = \left(x^q,y^q\right).$$ Then $\multiplicativegroup{\finiteFieldExtension{n}} \times \multiplicativegroup{\finiteFieldExtension{n}}$ is realized as the subgroup of $\restrictionOfScalars{\finiteFieldExtension{2}}{\finiteField}{\algebraicGroup{\GL}_n}\left(\finiteFieldExtension{n}\right) = \GL_{n}\left(\finiteFieldExtension{n}\right) \times \GL_{n}\left(\finiteFieldExtension{n}\right)$ given by the image of the map $$z \mapsto \left(\diag\left(Dz,D\Frobenius^2 z,\dots, D\Frobenius^{n-2} z\right), \diag\left(D\Frobenius z, D\Frobenius^3 z, \dots, D\Frobenius^{n-1} z\right)\right),$$
		and the character $\alpha \times \alpha^{-q}$ is given by $$z = \left(x,y\right) \mapsto \alpha\left(x\right) \alpha^{-q}\left(y\right).$$
    \item (A case 2) $\algebraicGroup{G}=\algebraicGroup{\UnitaryGroup}_n$ and $w$ is conjugate to a cycle of length $n$ where $n$ is odd. In this case $\algebraicGroup{T} = \restrictionOfScalars{\finiteFieldExtension{n}}{\finiteField}{\algebraicGroup{\UnitaryGroup}_1}$ with $\Frobenius$ acting by $-1\cdot w$ on the domain and by the cycle $\sigma=\left(\begin{smallmatrix} 0 & w_nw \\
        w w_n& 0
        \end{smallmatrix}\right)$ on the image, and $$x_{\multiplicativegroup{\finiteFieldExtension{2n}},\transfer{\theta}}=\phi\left(x_{\NormOneGroup{2n},\theta}\right)$$ where $\transfer{\theta}\left(t\right) = \theta\left(t^{1-q^n}\right)$.
    \end{enumerate}
\end{lemma}

\begin{remark}\label{rem:gspin-cases}
In the $\operatorname{GSpin}$ cases, the map on cocharacters is given by $(x,\lambda) \to (x+2\lambda,-w_n x)$. We can lift our characters, tori pairs such that our new $(T,\theta)$ satisfies $\theta\restriction_{\algebraicGroup{Z}(\finiteField)} = \pm 1$. In these cases the $\lambda$ component of $x_{(T,\theta)}$ is in $\zIntegers/2\zIntegers$, which allows us to use \Cref{lem:cochar_computation} with $\phi$ given by a restriction of the cocharacter map.
\end{remark}

\begin{proof}
We recall the definition of $x_{\algebraicGroup{T},\theta}$. Let $m$ be sufficiently large and divisible and let $\zeta_m = 1/(q^m-1)$ under the fixed identification $\multiplicativegroup{\algebraicClosure{\finiteField}} \cong (\mathbb{Q}/\zIntegers)_{p'}$. Then $x_{\algebraicGroup{T},\theta}$ is defined by the formula
$$
    \exp\left(\innerproduct{x_{\algebraicGroup{T},\theta}}{y}\right) = \theta\left(\aFieldNorm_{\algebraicGroup{T}(\finiteField_m):\algebraicGroup{T}(\finiteField)}\left(y\left(\zeta_m\right)\right)\right).
$$

In the following computations, we will choose $m$ to be sufficiently large and divisible so that the torus $\algebraicGroup{T}\left(\finiteField_m\right)$ is isomorphic to $\left(\multiplicativegroup{\finiteFieldExtension{m}}\right)^N$. We will then identify the cocharacter lattice $\CocharacterLattice{\algebraicGroup{T}}$ with $\zIntegers^N$ by the map $\Cocharacter \colon \zIntegers^N \to \CocharacterLattice{\algebraicGroup{T}}$ given by $\left(y_1,\dots,y_N\right) \mapsto \left(\Cocharacter_{(y_1,\dots,y_N)}\left(t\right) = \left(t^{y_1},\dots,t^{y_N}\right)\right)$. Similarly using the above identifications denote for $\left(x_1,\dots,x_N\right) \in \zIntegers^N$ the algebra character $\Character_{(x_1,\dots,x_{N})} \in \CharacterLattice{\algebraicGroup{T}}$ given by the formula $$\Character_{(x_1,\dots,x_{N})}\left(\Cocharacter_{(y_1,\dots, y_N)}\left(t\right)\right) = t^{\sum_{i=1}^N x_i y_i}.$$

We have that $$\innerproduct{\phi \Character_{(x_1,\dots,x_{n})}}{\Cocharacter_{(y_1,\dots, y_{2n})}\left(t\right)} = \innerproduct{\Character_{(x_1,\dots,x_n, -x_n,\dots,-x_1)}}{\Cocharacter_{(y_1,\dots, y_{2n})}\left(t\right)},$$
which equals $$\innerproduct{\Character_{(x_1,x_2,\dots,x_{n})}}{\Cocharacter_{(y_1-y_{2n},y_2-y_{2n-1},\dots, y_n-y_{n+1})}\left(t\right)}.$$
We denote $$\phi^{\ast}\left(\Cocharacter_{(y_1,\dots, y_{2n})}\right) = \Cocharacter_{(y_1-y_{2n},y_2-y_{2n-1},\dots, y_n-y_{n+1})}.$$

We write down the norm computations for the different tori and $\Frobenius$ actions. To ease the notation, we will work with the cycle $\left(1,2,\dots,n\right)$. Working with a general cycle $\left(c_1, c_2, \dots, c_n\right)$ is similar and amounts to reindexing $y_i$ on the right hand side.
\begin{enumerate}
    \item \label{item:norm-n-cycle}(B/C/D Case 1: cycle of length $n$) $\algebraicGroup{T}(\finiteField) \cong \finiteField_n^\times$ and $\algebraicGroup{T}\left(\finiteFieldExtension{mn}\right) \cong \left(\multiplicativegroup{\finiteFieldExtension{mn}}\right)^n$ with $$\FieldNorm{\algebraicGroup{T}(\finiteField_{mn})}{\algebraicGroup{T}(\finiteField)}\left(\Cocharacter_{(y_1,\dots,y_n)}(\zeta_{mn})\right) = \prod_{i=0}^{n-1} \left(\zeta_n\right)^{q^{i} y_{i+1}}.$$
    \item (B/C/D Case 2: cycle of length $n$ with one negative sign) $\algebraicGroup{T}(\finiteField) \cong \NormOneGroup{2n}$ and $\algebraicGroup{T}\left(\finiteFieldExtension{2mn}\right) \cong \left(\multiplicativegroup{\finiteFieldExtension{2mn}}\right)^{n}$ with $$\FieldNorm{\algebraicGroup{T}(\finiteField_{2mn})}{\algebraicGroup{T}(\finiteField)}\left(\Cocharacter_{(y_1,\dots, y_n)}\left(\zeta_{2mn}\right)\right) = \prod_{i=0}^{n-1}(\beta_n)^{q^{i} y_{i+1}},$$
    where $\beta_n = \zeta_{2n}^{1-q^n}$.
    \item (A Case 1: cycle of length $n$, $n$ even, $\Frobenius$ acts by $-1 \cdot w$ on the domain) $\algebraicGroup{T}(\finiteField) \cong \finiteField_n^\times$ and $\algebraicGroup{T}\left(\finiteFieldExtension{mn}\right) \cong \left(\multiplicativegroup{\finiteFieldExtension{mn}}\right)^n$ with $$\FieldNorm{\algebraicGroup{T}(\finiteField_{mn})}{\algebraicGroup{T}(\finiteField)}\left(\Cocharacter_{(y_1,\dots,y_n)}(\zeta_{mn})\right) = \prod_{i=0}^{n-1} (\zeta_n)^{\left(-1\right)^{i} q^{i} y_{i+1}}.$$
    \item (A Case 2: cycle of length $n$, $n$ odd, $\Frobenius$ acts by $-1 \cdot w$ on the domain) $\algebraicGroup{T}(\finiteField) \cong \NormOneGroup{2n}$ and $\algebraicGroup{T}\left(\finiteFieldExtension{2mn}\right) \cong \left(\multiplicativegroup{\finiteFieldExtension{2mn}}\right)^{n}$ with $$\FieldNorm{\algebraicGroup{T}(\finiteField_{2mn})}{\algebraicGroup{T}(\finiteField)}\left(\Cocharacter_{(y_1,\dots, y_n)}\left(\zeta_{2mn}\right)\right) = \prod_{i=0}^{n-1} (\beta_n)^{\left(-1\right)^{i}q^{i}y_{i+1}},$$
    where $\beta_n = \zeta_{2n}^{1-q^n}$.
\end{enumerate}
We move to the computation of $$
\exp\left(\innerproduct{\phi \left(x_{\algebraicGroup{T},\theta}\right)}{y}\right) = \theta\left(\aFieldNorm_{\algebraicGroup{T}(\finiteField_m):\algebraicGroup{T}(\finiteField)}\left(\left(\phi^{\ast}y\right)\left(\zeta_m\right)\right)\right).
$$ We have the following equalities:
\begin{enumerate}
	\item \label{item:norm-b-c-d-case-1}(B/C/D Case 1) $$\prod_{i=0}^{n-1} \alpha\left(\zeta_n^{q^i \left(y_{i+1}-y_{2n-i}\right)}\right) = \prod_{i=0}^{n-1} \alpha\left(\zeta_n^{q^i y_{i+1}}\right) \cdot \prod_{i=0}^{n-1} \alpha^{-1}\left(\zeta_n^{q^i y_{2n-i}}\right).$$
	\item \label{item:norm-b-c-d-case-2}(B/C/D Case 2) \begin{align*}
		\prod_{i=0}^{n-1} \theta\left(\beta_n^{ q^i \left(y_{i+1}-y_{2n-i}\right)}\right) &= \prod_{i=0}^{n-1} \theta\left(\beta_n^{ q^i y_{i+1}}\right) \cdot \prod_{i=0}^{n-1} \theta\left(\beta_n^{ q^{i+n} y_{2n-i}}\right) \\
		&= \prod_{i=0}^{n-1} \transfer{\theta}\left(\zeta_n^{q^i y_{i+1}}\right) \cdot \prod_{i=0}^{n-1} \transfer{\theta}\left(\zeta_n^{q^{n+i} y_{2n-i}}\right).
	\end{align*}
	\item \label{item:norm-a-case-1}(A Case 1) \begin{align*}
		& \prod_{i=0}^{n-1} \alpha\left(\zeta_n^{\left(-1\right)^i q^i \left(y_{i+1} - y_{2n-i}\right)}\right)\\
		=& \prod_{\substack{0 \le i \le n-2\\
				i \text{ is even}}} \alpha\left(\zeta_n^{q^{i} y_{i+1} - q^{i} y_{2n-i}}\right) \cdot \prod_{\substack{1 \le i \le n-1\\
				i \text{ is odd}}} \alpha\left(\zeta_n^{-q^{i} y_{i+1} + q^{i} y_{2n-i}}\right) \\
			=& \prod_{\substack{0 \le i \le n-2\\
					i \text{ is even}}} \alpha\left(\zeta_n^{q^{i} y_{i+1} + q^{i + 1} y_{2n-i-1}}\right) \cdot \prod_{\substack{0 \le i \le n-2\\
					i \text{ is even}}} \alpha^{-q}\left(\zeta_n^{q^{i} y_{i+2} + q^{i-1} y_{2n-i}}\right).
	\end{align*}
	\item \label{item:norm-a-case-2} (A Case 2)\begin{align*}
		&\prod_{i=0}^{n-1} \theta\left( \beta_n^{\left(-1\right)^i q^i \left(y_{i+1} - y_{2n - i}\right)}\right)\\
		 = & \prod_{\substack{0 \le i \le n-1\\
				i \text{ is even}}} \theta\left(\beta_n^{q^{i} y_{i+1} + q^{i+n} y_{2n-i}}\right) \cdot \prod_{\substack{1 \le i \le n-2\\
				i \text{ is odd}}} \theta\left(\beta_n^{q^{i+n} y_{i+1} + q^{i} y_{2n-i}}\right)\\
			= &\transfer{\theta}\left(\zeta_n^{y_1 + q^n y_{2n}}\right) \cdot \prod_{\substack{2 \le i \le n-1\\
			i \text{ is even}}} \transfer{\theta}\left(\zeta_n^{q^{i} y_{i+1} + q^{n-i} y_{n+i}}\right) \cdot \prod_{\substack{1 \le i \le n-2\\
		i \text{ is odd}}} \transfer{\theta}\left(\zeta_n^{q^{i+n} y_{i+1} + q^{2n-i} y_{n+i}}\right).
	\end{align*}
\end{enumerate}

The results now follow from these equalities and the norm computation for an $n$-cycle (\ref{item:norm-n-cycle}) by decomposing $\sigma$ into two product of two $n$-cycles or into a $2n$-cycle.

\begin{enumerate}
	\item \label{item:type-b-c-case-1-gl2n}(B/C/D Case 1) Write $y = \left(y_+, y_{-}\right)$ where $y_{+} = \left(y_1,\dots,y_n\right)$ and $y_{-} = \left(y_{2n},y_{2n-1},\dots,y_{n+1}\right)$. Notice $\sigma = \left(1,\dots,n\right)\left(2n,2n-1,\dots,n+1\right)$. Then $\transfer{\theta}\left(\aFieldNorm_{\tilde{\algebraicGroup{T}}(\finiteField_m):\tilde{\algebraicGroup{T}}(\finiteField)}\left(\Cocharacter_{y}\left(\zeta_{mn}\right)\right)\right)$ is given by \begin{align*}& (\alpha \times \alpha^{-1})\left(\aFieldNorm_{\finiteField_{mn}:\finiteField_n}\left(\Cocharacter_{y_{+}}\left(\zeta_{mn}\right)\right),\aFieldNorm_{\finiteField_{mn}:\finiteField_n}\left(\Cocharacter_{y_{-}}\left(\zeta_{mn}\right)\right)\right)\\	
		=& \prod_{i=0}^{n-1} \alpha\left(\zeta_n^{q^i y_{\sigma^i(1)}}\right) \cdot \prod_{i=0}^{n-1} \alpha^{-1}\left(\zeta_n^{q^i y_{\sigma^i(2n)}}\right)\\
	=& \prod_{i=0}^{n-1} \alpha\left(\zeta_n^{q^i y_{i+1}}\right) \cdot \prod_{i=0}^{n-1} \alpha^{-1}\left(\zeta_n^{q^i y_{2n-i}}\right).
	\end{align*}
	\item (B/C/D Case 2) Write $y = \left(y_+, y_{-}\right)$ as in (\ref{item:type-b-c-case-1-gl2n}). Notice $\sigma = \left(1,\dots,n,2n,2n-1,\dots,n+1\right)$. Then $\transfer{\theta}\left(\aFieldNorm_{\tilde{\algebraicGroup{T}}(\finiteField_m):\tilde{\algebraicGroup{T}}(\finiteField)}\left(\Cocharacter_{y}\left(\zeta_{mn}\right)\right)\right)$ is given by \begin{align*} \transfer{\theta}\left(\aFieldNorm_{\finiteField_{mn}:\finiteField_n}\left(\Cocharacter_{y}\left(\zeta_{mn}\right)\right)\right) = & \prod_{i=0}^{2n-1} \transfer{\theta} \left(\zeta_n^{q^i y_{\sigma^i(1)}}\right)\\
		=& \prod_{i=0}^{n-1} \transfer{\theta}\left(\zeta_n^{q^i y_{i+1}}\right) \cdot \prod_{i=0}^{n-1} \transfer{\theta}\left(\zeta_n^{q^{n+i} y_{2n-i}}\right).
		\end{align*}
	\item (A Case 1) Write $y = \left(y_{+}, y_{-}\right)$ where $y_{+} = \left(y_1,y_{2n-1},y_3,y_{2n-3},\dots, y_{n-1},y_{n+1}\right)$ and $y_{-} = \left(y_{2},y_{2n-2},y_4,y_{2n-4}, \dots y_{n},y_{2n}\right)$. Notice $$\sigma = \left(1,2n-1,3,2n-3,\dots,n,2n\right)\left(2,2n-2,4,2n-4,\dots,n,2n\right).$$ Then $\transfer{\theta}\left(\aFieldNorm_{\tilde{\algebraicGroup{T}}(\finiteField_m):\tilde{\algebraicGroup{T}}(\finiteField)}\left(\Cocharacter_{y}\left(\zeta_{mn}\right)\right)\right)$ is given by \begin{align*}& (\alpha \times \alpha^{-q})\left(\aFieldNorm_{\finiteField_{mn}:\finiteField_n}\left(\Cocharacter_{y_{+}}\left(\zeta_{mn}\right)\right),\aFieldNorm_{\finiteField_{mn}:\finiteField_n}\left(\Cocharacter_{y_{-}}\left(\zeta_{mn}\right)\right)\right)\\
		=& \prod_{i=0}^{n-1} \alpha\left(\zeta_n^{q^i y_{\sigma^i(1)}}\right) \cdot \prod_{i=0}^{n-1} \alpha^{-1}\left(\zeta_n^{q^i y_{\sigma^i(2)}}\right)\\
		=& \prod_{\substack{0 \le i \le n-2\\
		i \text{ is even}}} \alpha\left(\zeta_n^{q^{i} y_{i+1} + q^{i + 1} y_{2n-i-1}}\right) \cdot \prod_{\substack{0 \le i \le n-2\\
		i \text{ is even}}} \alpha^{-q}\left(\zeta_n^{q^{i} y_{i+2} + q^{i-1} y_{2n-i}}\right).
		\end{align*}
	\item (A Case 2) Write $y = \left(y_{+}, y_{-}\right)$ where $y_{+} = \left(y_1,y_{2n-1},y_3,y_{2n-3},\dots, y_{n},y_{2n}\right)$ and $y_{-} = \left(y_{2},y_{2n-2},y_4,y_{2n-4}, \dots ,y_{n-1},y_{n+1}\right)$. Notice $$\sigma = (1,2n-1,3,2n-3,\dots,n,2n,2,2n-2,4,2n-4,\dots,n-1,n+1).$$ Then $\transfer{\theta}\left(\aFieldNorm_{\tilde{\algebraicGroup{T}}(\finiteField_m):\tilde{\algebraicGroup{T}}(\finiteField)}\left(\Cocharacter_{y}\left(\zeta_{mn}\right)\right)\right)$ is given by \begin{align*}& \transfer{\theta}\left(\aFieldNorm_{\finiteField_{mn}:\finiteField_n}\left(\Cocharacter_{y}\left(\zeta_{mn}\right)\right)\right) \\
		=& \prod_{i=0}^{2n-1} \transfer{\theta} \left(\zeta_n^{q^i y_{\sigma^i(1)}}\right).\\
		=& \transfer{\theta}\left(\zeta_n^{y_1 + q^n y_{2n}}\right) \cdot \prod_{\substack{2 \le i \le n-1\\
			i \text{ is even}}} \transfer{\theta}\left(\zeta_n^{q^{i} y_{i+1} + q^{n-i} y_{n+i}}\right) \cdot \prod_{\substack{1 \le i \le n-2\\
		i \text{ is odd}}} \transfer{\theta}\left(\zeta_n^{q^{i+n} y_{i+1} + q^{2n-i} y_{n+i}}\right).
		\end{align*}
\end{enumerate}

\end{proof}

\begin{proof}[Proof of \Cref{thm:indep_geo_conj}] 
By \Cref{thm:computation-of-doubling-gauss-sum-scalar-for-deligne-lusztig-characters}, it is enough to show that $\GaussSumTorusCharacter{\transfer{T}}{\transfer{\theta}}{\chi}{\fieldCharacter_{\quadraticExtension}}$ and $\omega_{\pi}(-1) = \theta(-1)$ only depend of the geometric conjugacy of $(T,\theta)$. The latter only depends on geometric conjugacy by \Cref{prop:central_char}.

For the former we will show
\[
g_{\algebraicGroup{T},\rho}(\chi,\theta,\fieldCharacter_{\quadraticExtension}) = \GaussSumTorusCharacter{\transfer{T}}{\transfer{\theta}}{\chi}{\fieldCharacter_{\quadraticExtension}} \cdot \begin{cases}
        \GaussSumSingleCharacter{\chi}{\fieldCharacter} & \algebraicGroup{H} = \algebraicGroup{\Sp}_{2n} \textnormal{ or } \algebraicGroup{H} = \algebraicGroup{GSp}_{2n},\\
        1 & \text{otherwise.}        
    \end{cases}
\]

We prove this by showing that the Frobenius equivariant map $\CharacterLattice{\algebraicGroup{T}} \to \CharacterLattice{\Tilde{\algebraicGroup{T}}}$ decomposes into a direct sum of $\phi$'s appearing in \Cref{lem:cochar_computation}. 

Assume that $\left(T, \theta\right)$ corresponds to the pair of partitions $\left(\lambda^+, \lambda^-\right)$ and to the characters $\alpha_1 \times \dots \times \alpha_{\lengthof\left(\lambda^+\right)} \colon \multiplicativegroup{\finiteFieldExtension{\lambda^+}} \to \multiplicativegroup{\cComplex}$ and $\theta_1 \times \dots \times \theta_{\lengthof\left(\lambda^-\right)} \colon \NormOneGroup{2 \lambda^{-}} \to \multiplicativegroup{\cComplex}$ as in \Cref{subsec:rational-tori} and \Cref{subsec:some-notation}.

Let $\algebraicGroup{T}_{i,+} \subset \algebraicGroup{T}$ be the component corresponding to $\multiplicativegroup{\finiteFieldExtension{\lambda^+_i}}$ in the description above. Similarly, let $\algebraicGroup{T}_{i,-} \subset \algebraicGroup{T}$ be the component corresponding to $\NormOneGroup{2 \lambda_i^{-}}$. Let us denote by $\tilde{\algebraicGroup{T}}_{i,+} \subset \tilde{\algebraicGroup{T}}$ the component corresponding to $\multiplicativegroup{\finiteFieldExtension{\lambda_i^+}} \times \multiplicativegroup{\finiteFieldExtension{\lambda_i^+}}$ and by $\tilde{\algebraicGroup{T}}_{i,-} \subset \tilde{\algebraicGroup{T}}$ the component corresponding to $\multiplicativegroup{\finiteFieldExtension{2 \lambda_i^-}}$. Let $$\phi_{\pm,i} \colon \CocharacterLattice{\algebraicGroup{T}^{\ast}_{i,\pm}} \to \CocharacterLattice{\tilde{\algebraicGroup{T}}^{\ast}_{i,\pm}}$$ be the map as in \Cref{lem:cochar_computation}. 

In the classical group cases where $\algebraicGroup{H} = \algebraicGroup{G}$ except for the case where $\algebraicGroup{H} = \algebraicGroup{\Sp}_{2n}$, we have that $\phi \colon \CocharacterLattice{\algebraicGroup{T^{\ast}}} \to \CocharacterLattice{\algebraicGroup{\tilde{T}^{\ast}}}$ is given by $$\phi = \left(\bigoplus_{i = 1}^{\lengthof\left(\lambda^+\right)} \phi_{+,i}\right) \oplus \left(\bigoplus_{i=1}^{\lengthof\left(\lambda^-\right)} \phi_{-,i}\right).$$
By \Cref{lem:cochar_computation}, we have in these cases $\left(\tilde{\algebraicGroup{T}}, \tilde{\theta}\right) = \left(\transfer{\algebraicGroup{T}}, \transfer{\theta}\right)$ and thus
$$g_{\algebraicGroup{T},\rho}(\chi,\theta,\fieldCharacter_{\quadraticExtension}) = \GaussSumTorusCharacter{\transfer{T}}{\transfer{\theta}}{\chi}{\fieldCharacter_{\quadraticExtension}}.$$

If $\algebraicGroup{H} = \algebraicGroup{\Sp}_{2n}$ we have that $\phi \colon \CocharacterLattice{\algebraicGroup{T^{\ast}}} \to \CocharacterLattice{\algebraicGroup{\tilde{T}^{\ast}}}$ is given by $$\phi = \left(\bigoplus_{i = 1}^{\lengthof\left(\lambda^+\right)} \phi_{+,i}\right) \oplus \left(\bigoplus_{i=1}^{\lengthof\left(\lambda^-\right)} \phi_{-,i}\right) \oplus 0.$$
Here the zero component comes from the fact that the middle component of the maximal split torus of $\algebraicGroup{\SO}_{2n + 1}$ is $1$.
Thus by \Cref{lem:cochar_computation}, we have in this case $\left(\tilde{\algebraicGroup{T}}, \tilde{\theta}\right) = \left(\transfer{\algebraicGroup{T}} \times \multiplcativeScheme, \transfer{\theta} \times 1\right)$ and therefore
$$g_{\algebraicGroup{T},\rho}(\chi,\theta,\fieldCharacter_{\quadraticExtension}) = \GaussSumTorusCharacter{\transfer{T}}{\transfer{\theta}}{\chi}{\fieldCharacter} \GaussSumSingleCharacter{\chi}{\fieldCharacter}.$$

In the similitude group cases $H = \GroupExtension{G}$, $\rho$ factors as
\[
	\GroupExtension{\algebraicGroup{G}}^* \to_{i^*} \algebraicGroup{G}^* \to_{\rho_{\algebraicGroup{G}}} \algebraicGroup{GL}_n.
\]
By \Cref{prop:i*isrestriction}, $\rho(s_{\algebraicGroup{T},\theta}) = \rho_{\algebraicGroup{G}}(s_{\algebraicGroup{T} \cap \algebraicGroup{G},\theta|_{\algebraicGroup{T} \cap \algebraicGroup{G}}})$ and so 
\[
	g_{\algebraicGroup{T},\rho}(\chi,\theta,\fieldCharacter_{\quadraticExtension}) = g_{\algebraicGroup{T} \cap \algebraicGroup{G},\rho}(\chi,\theta|_{\algebraicGroup{T} \cap \algebraicGroup{G}},\fieldCharacter_{\quadraticExtension})
\]
which, by applying our above work for $\algebraicGroup{G}$ and $\rho_{\algebraicGroup{G}}$, is equal to
\begin{align*}
	g_{\algebraicGroup{T} \cap \algebraicGroup{G},\rho_{\algebraicGroup{G}}}(\chi,\theta|_{\algebraicGroup{T} \cap \algebraicGroup{G}},\fieldCharacter_{\quadraticExtension}) = \GaussSumTorusCharacter{\transfer{T}}{\transfer{\theta}}{\chi}{\fieldCharacter_{\quadraticExtension}} \cdot \begin{cases}
        \GaussSumSingleCharacter{\chi}{\fieldCharacter} & \algebraicGroup{G} = \algebraicGroup{\Sp}_{2n} \textnormal{ or } \algebraicGroup{G} = \algebraicGroup{GSp}_{2n},\\
        1 & \text{otherwise.}        
    \end{cases}
\end{align*}
\end{proof}

\subsection{Stability}\label{subsec:stability}
The goal of this section is to compute $\posDblJacobiSumScalar{\pi}{\fieldCharacter}$ via \Cref{thm:Phi_stable}.

\begin{theorem}
\label{thm:doubling-method-gamma-factor-for-deligne-lusztig}
Let $H$ be either a classical group $G$ or its similitude group counterpart $\GroupExtension{G}$. Using the same notation as in \Cref{thm:computation-of-doubling-gauss-sum-scalar-for-deligne-lusztig-characters}, if $\pi$ is an irreducible representation of $H$ that appears as a subrepresentation of the virtual representation $\RTThetaVirtual{T_H}{H}{\theta_{H}}$ then \begin{equation}\label{eq:doubling-method-gamma-factor-is-gauss-sum-of-functorial-lift}
	\dblGammaFactorSpace{\hermitianSpace}{\pi}{\chi}{\fieldCharacter} = \GaussSumTorusCharacter{\transfer{T}}{\transfer{\theta}}{\chi}{\fieldCharacter_{\quadraticExtension}}.
\end{equation}
\end{theorem}
Recall that in \Cref{thm:indep_geo_conj}, we showed that the right hand side of \eqref{eq:doubling-method-gamma-factor-is-gauss-sum-of-functorial-lift} only depends on the geometric conjugacy class of $(T,\theta)$ and hence only on the Lusztig series $\LusztigSeries{G}{s}$ which contains $\pi$. Thus \Cref{thm:doubling-method-gamma-factor-for-deligne-lusztig} implies \Cref{thm:Phi_stable}.

The proof strategy is to prove the following inductive statements
\begin{enumerate}[(a)]
	\item \label{item:follows-for-non-cuspidal}If \Cref{thm:doubling-method-gamma-factor-for-deligne-lusztig} holds for all representations of $H'$ where $H'$ corresponds to $\left(\hermitianSpace', \innerproduct{\cdot}{\cdot}\right)$ with $\dim_{\quadraticExtension} V' < n$, then \Cref{thm:doubling-method-gamma-factor-for-deligne-lusztig} holds for all non-cuspidal irreducible representations of $H$ with $\dim_{\quadraticExtension} V = n$.
	\item \label{item:follows-for-cuspidal}If \Cref{thm:doubling-method-gamma-factor-for-deligne-lusztig} holds for all non-cuspidal irreducible representations of $H$, then it holds for all irreducible cuspidal representations of $H$.
\end{enumerate}
The key input to \ref{item:follows-for-non-cuspidal} is the multiplicativity of the $\dblGammaFactorSpace{\hermitianSpace}{\pi}{\chi}{\fieldCharacter}$ factor along with \Cref{thm:indep_geo_conj}. The key input to \ref{item:follows-for-cuspidal} is a result of Lusztig controlling the cuspidal representations in a Lusztig series.

\subsubsection{Proof of \Cref{thm:doubling-method-gamma-factor-for-deligne-lusztig}}
We first reduce \Cref{thm:doubling-method-gamma-factor-for-deligne-lusztig} to the connected center group case $H = \GroupExtension{G}$. We start with the following observation. Suppose that $\tilde{\pi}$ is an irreducible representation of $\GroupExtension{G}$ and that $\pi$ is an irreducible representation of $G$ that appears as a subrepresentation of the restriction of $\tilde{\pi}$ to $G$. Then since $\posDblJacobiSum{\tilde{\pi}}{\chi}$ and $\posDblJacobiSum{\pi}{\chi}$ are scalar operators and the restriction of $\posDblJacobiSum{\tilde{\pi}}{\chi}$ to the subspace $\pi$ is $\posDblJacobiSum{\pi}{\chi}$, we have the equality $$\posDblJacobiSumScalar{\pi}{\chi} = \posDblJacobiSumScalar{\tilde{\pi}}{\chi}$$ and thus
$$\dblGammaFactorSpace{\hermitianSpace}{\pi}{\chi}{\fieldCharacter} = \dblGammaFactorSpace{\hermitianSpace}{\tilde{\pi}}{\chi}{\fieldCharacter}.$$

Since the map $\algebraicGroup{\GroupExtension{G}}^*(\finiteField) \to \algebraicGroup{G}^*(\finiteField)$ is surjective, by \Cref{prop:lusztig-series-and-restriction} there exists $\tilde{s}$ depending on $s$ such that for every $\pi \in \LusztigSeries{\FrobeniusFixedPoints{G}}{s}$, there exists some $\tilde{\pi} \in \LusztigSeries{\FrobeniusFixedPoints{\GroupExtension{G}}}{\tilde{s}}$ such that $\pi$ is an irreducible subrepresentation of the restriction of $\tilde{\pi}$ to $G$. It thus suffices to prove \Cref{thm:doubling-method-gamma-factor-for-deligne-lusztig} for $\tilde{\pi}$.

We now prove \Cref{thm:doubling-method-gamma-factor-for-deligne-lusztig} for the connected center group $\GroupExtension{G}$. Our proof will be via induction as described above.

\begin{proof}[Proof of \ref{item:follows-for-non-cuspidal}]
	We will show that both sides of \eqref{eq:doubling-method-gamma-factor-is-gauss-sum-of-functorial-lift} satisfy a multiplicativity property. Since every proper parabolic subgroup of $H$ is contained in a proper maximal parabolic subgroup with Levi part $L = \FrobeniusFixedPoints{\algebraicGroup{L}}$ of the form $\algebraicGroup{L} = \restrictionOfScalars{\quadraticExtension}{\finiteField}{\algebraicGroup{\GL}_k}\times \algebraicGroup{H}'$, where $\algebraicGroup{H}'$ is as in \Cref{subsec:multiplicativity}, the multiplicativity of the left hand side of \eqref{eq:doubling-method-gamma-factor-is-gauss-sum-of-functorial-lift} follows from \Cref{thm:multiplicativity-in-terms-of-gamma-factors}.
	
	For the right hand side of \eqref{eq:doubling-method-gamma-factor-is-gauss-sum-of-functorial-lift}, first \Cref{prop:lusztig-series-are-unions-of-harish-chandra-series} tells us that if $\pi \in \LusztigSeries{
		{\algebraicGroup{H}}^{\Frobenius}}{s}$ is a subrepresentation of the parabolically induced representation $\Ind{H}{G}{\inf \sigma}$ where $\sigma$ is an irreducible representation of the Levi part $L$ of $P$, then $\sigma \in \LusztigSeries{
		\algebraicGroup{L}^{\Frobenius}}{s}$.
		
	Thus if $\pi$ is a subrepresentation of a parabolically induced representation as above, then $(\algebraicGroup{T}_H,\theta_H)$ is geometrically conjugate to $(\algebraicGroup{T}_{\GL_k\left(\quadraticExtension\right)} \times \algebraicGroup{T}_{H'}, \theta_{\GL_k\left(\quadraticExtension\right)} \times \theta_{H'})$ for $\algebraicGroup{T}_{\GL_k\left(\quadraticExtension\right)} \times \algebraicGroup{T}_{H'}$ a torus of $\algebraicGroup{L} = \restrictionOfScalars{\quadraticExtension}{\finiteField}{\algebraicGroup{\GL}_k} \times \algebraicGroup{H}'$. By \Cref{thm:indep_geo_conj}, $$g_{T \cap G}\left(\chi,\theta_H\restriction_{T \cap G},\fieldCharacter_{\quadraticExtension}\right) = g_{\left(T_{\GL_k\left(\quadraticExtension\right)} \times T_{H'}\right) \cap G}\left(\chi,\left(\theta_{\GL_k\left(\quadraticExtension\right)}\times \theta_{H'}\right)\restriction_{\left(T_{\GL_k\left(\quadraticExtension\right)} \times T_{H'}\right) \cap G},\fieldCharacter_{\quadraticExtension}\right).$$ The right hand of the last equality is multiplicative since Gauss sums are multiplicative by \eqref{eq:Gauss_multiplicitive}.
\end{proof}
\begin{proof}[Proof of \ref{item:follows-for-cuspidal} for the connected center group case $H = \GroupExtension{G}$]
By \Cref{thm:computation-of-doubling-gauss-sum-scalar-for-deligne-lusztig-characters}, we have that if $$\RTThetaVirtual{T_H}{H}{\theta_H} = \sum_{\pi} c_{\pi} \pi,$$ where $\pi$ goes over all (equivalence classes of) irreducible representations of $H$ and $c_{\pi} \in \zIntegers$, then $\frac{\posDblVirtualJacobiSumScalar{\DeligneLusztigInduction{T_H}{H} \theta_H}{\chi}}{c_{\hermitianSpace}\left(\chi, \fieldCharacter\right)}$ is given by the formula
$$\sum_{\pi} c_{\pi} \dblGammaFactorSpace{\hermitianSpace}{\pi}{\chi}{\fieldCharacter} = \GaussSumTorusCharacter{\transfer{T}}{\transfer{\theta}}{\chi}{\fieldCharacter_{\quadraticExtension}} \RTGTheta{T_H}{H}{\theta_H}\left(1\right) = \GaussSumTorusCharacter{\transfer{T}}{\transfer{\theta}}{\chi}{\fieldCharacter_{\quadraticExtension}} \left(\sum_{\pi} c_{\pi}\right).$$
Subtracting the equality for non-cuspidal $\pi$ established by \ref{item:follows-for-non-cuspidal}, we are left with the equality
$$\sum_{\pi \textnormal{ cuspidal}} c_{\pi} \dblGammaFactorSpace{\hermitianSpace}{\pi}{\chi}{\fieldCharacter} = \GaussSumTorusCharacter{\transfer{T}}{\transfer{\theta}}{\chi}{\fieldCharacter_{\quadraticExtension}} \left(\sum_{\pi \textnormal{ cuspidal}} c_{\pi}\right).$$
But by \Cref{prop:at-most-one-cuspidal-lusztig-series} there is at most one cuspidal $\pi$ appearing in our sum.
\end{proof}
This completes the proof of \Cref{thm:doubling-method-gamma-factor-for-deligne-lusztig}.

\appendix
\section{Evaluation of some exponential sums}\label{appendix:evaluation-of-some-exponential-sums}

In this appendix, we evaluate several exponential sums.

\subsection{The case $\quadraticExtension = \finiteField$}\label{appendix:e-equals-f}
Let $\chi \colon \multiplicativegroup{\finiteField} \to \multiplicativegroup{\cComplex}$ be a non-trivial character. We compute $$\sum_{\substack{x \in \multiplicativegroup{\finiteFieldExtension{2}}\\
	\trace_{\finiteFieldExtension{2} \slash \finiteField}\left(x\right) \ne 0}} \chi\left(\FieldNorm{2}{1}\left(\frac{x}{\trace_{\finiteFieldExtension{2} \slash \finiteField} x}\right) \right).$$
As usual, we rewrite this sum as
$$\frac{1}{q} \sum_{z \in \finiteField} \sum_{y \in \multiplicativegroup{\finiteField}} \sum_{x \in \multiplicativegroup{\finiteFieldExtension{2}}} \chi\left(\FieldNorm{2}{1}\left(xy\right)\right) \fieldCharacter\left(z\left(y \trace_{\finiteFieldExtension{2} \slash \finiteField} x - 1\right)\right).$$
It is clear that if $z = 0$, then the sum over $x$ vanishes. Hence, we may reduce the sum to a sum over $z \in \multiplicativegroup{\finiteField}$. By replacing $y$ with $z^{-1} y$ and then replacing $z$ with $-z$, we arrive at the sum
$$\frac{1}{q} \sum_{y \in \multiplicativegroup{\finiteField}} \sum_{x \in \multiplicativegroup{\finiteFieldExtension{2}}} \chi\left(\FieldNorm{2}{1}\left(xy\right)\right) \fieldCharacter\left(y \trace_{\finiteFieldExtension{2} \slash \finiteField} x\right) \sum_{z \in \multiplicativegroup{\finiteField}} \chi^{-2}\left(z\right)\fieldCharacter\left(z\right).$$
Replacing $x$ with $y^{-1}x$, and using the Hasse--Davenport relation, we arrive at the sum
$$q^{\frac{1}{2}} \left(q-1\right) \tau\left(\chi^{-1}, \fieldCharacter\right)^2 \tau\left(\chi^{2}, \fieldCharacter\right).$$

\subsection{The case $\quadraticExtension \ne \finiteField$}\label{appendix:e-not-equal-f}
Let $\chi \colon \multiplicativegroup{\finiteFieldExtension{2}} \to \multiplicativegroup{\cComplex}$ be a character such that $\involutionPlusOne{\chi} \ne 1$ and let $\theta \colon \NormOneGroup{2} \to \multiplicativegroup{\cComplex}$ be a character. Our goal is to compute $$\sum_{-1 \ne x \in \NormOneGroup{2}} \theta\left(x\right) \chi\left(1+x\right).$$
Using the Hilbert 90 map this is equivalent to
$$\frac{1}{q-1} \sum_{\substack{x \in \multiplicativegroup{\finiteFieldExtension{2}}\\
x \notin \delta \multiplicativegroup{\finiteField}}} \theta\left(\frac{x}{x^q}\right) \chi\left(1+\frac{x}{x^q}\right),$$
Where $\delta \in \multiplicativegroup{\finiteFieldExtension{2}}$ is a trace zero element. 
We arrive at the sum
$$\frac{1}{q-1} \sum_{\substack{x \in \multiplicativegroup{\finiteFieldExtension{2}}\\
		x \notin \delta \multiplicativegroup{\finiteField}}} \transfer{\theta}\left(x\right) \chi^{-q}\left(x\right) \chi\left(\trace_{\finiteFieldExtension{2} \slash \finiteField} x\right),$$
	where $\transfer{\theta} \colon \multiplicativegroup{\finiteFieldExtension{2}} \to \multiplicativegroup{\cComplex}$ is given by $\transfer{\theta}\left(x\right) = \theta\left(\frac{x}{x^q}\right)$.
	
	We have that $x \in \multiplicativegroup{\finiteFieldExtension{2}}$ satisfies $x \notin \delta \multiplicativegroup{\finiteField}$ if and only if $\trace_{\finiteFieldExtension{2} \slash \finiteField}\left(x\right) \ne 0$. As usual, we may use this last condition to rewrite the sum as
	$$\frac{1}{q\left(q-1\right)} \sum_{x \in \multiplicativegroup{\finiteFieldExtension{2}}}\sum_{z \in \finiteField} \sum_{y \in \multiplicativegroup{\finiteField}}  \transfer{\theta}\left(x\right) \chi^{-q}\left(x\right) \chi\left(y\right) \fieldCharacter\left(z\left(\trace_{\finiteFieldExtension{2} \slash \finiteField} \left(x\right) - y\right)\right).$$
	Since $\chi \restriction_{\multiplicativegroup{\finiteField}} \ne 1$, the inner sum over $y$ vanishes when $z = 0$, and therefore we can reduce the sum over $z$ to $z \in \multiplicativegroup{\finiteField}$. Changing variables $x \mapsto z^{-1} x$ and $y \mapsto -z^{-1} y$, and using the fact that $z^q = z$, we arrive at the sum
	$$\frac{1}{q} \chi\left(-1\right) \sum_{x \in \multiplicativegroup{\finiteFieldExtension{2}}}\transfer{\theta}\left(x\right)  \chi^{-q}\left(x\right) \fieldCharacter\left(\trace_{\finiteFieldExtension{2} \slash \finiteField} \left(x\right)\right) \sum_{y \in \multiplicativegroup{\finiteField}}   \chi\left(y\right) \fieldCharacter\left(y\right),$$
	which by changing variables $x \mapsto x^q$ equals
	$$q^{\frac{1}{2}} \chi\left(-1\right) \GaussSumCharacter{\transfer{\theta}}{\chi}{\fieldCharacter_{2}} \tau\left(\chi^{-1} \restriction_{\multiplicativegroup{\finiteField}}, \fieldCharacter\right).$$

\section{Multiplicativity of Jacobi sums} \label{appendix:multiplicativity-of-jacobi-sums}
In \cite{GirschZelingher2026}, the following theorem, which is closely related to \Cref{thm:multiplicativity-in-terms-of-gauss-sums}, was proved.

\begin{theorem}\label{thm:identity-of-jacobi-sums}
	Let $\hermitianSpace = \xIsotropic \oplus \hermitianSpace' \oplus \yIsotropic$ be as in \Cref{subsec:multiplicativity}.
	
	Let $\pi'$ be an irreducible representation of $\IsometryGroup^0\left(\hermitianSpace'\right)$ and let $\tau$ be an irreducible representation of $\GL_k\left(\quadraticExtension\right)$. Let $P_{\IsometryGroup^0\left(\hermitianSpace\right)} \subset \IsometryGroup^0\left(\hermitianSpace\right)$ be the parabolic subgroup of $\IsometryGroup^0\left(\hermitianSpace\right)$ stabilizing the subspace $\xIsotropic$. Then for every irreducible subrepresentation $\pi \subset \Ind{P_{\IsometryGroup^0\left(\hermitianSpace\right)}}{\IsometryGroup^0\left(\hermitianSpace\right)}{\tau \overline{\otimes} \pi'}$
	$$\posDblJacobiSumScalar{\pi}{\chi} = \posDblJacobiSumScalar{\pi'}{\chi} \posDblJacobiSumScalar{\tau \otimes \minusInvolution{\chi}_{\GL_{k}}}{\involutionPlusOne{\chi}},$$
	where
	$$\posDblJacobiSum{\pi}{\chi} = \posDblJacobiSumScalar{\pi}{\chi} \cdot \idmap_{\pi} \coloneq \frac{1}{\sqrt{\sizeof{\lieAlgebra}}} \sum_{\substack{g \in \IsometryGroup^0\left(\hermitianSpace\right)\\
			\detQuadratic\left(\idmap_{\hermitianSpace} + g\right) \ne 0}}\chi\left(\detQuadratic\left(\idmap_{\hermitianSpace} + g\right)\right) \pi\left(g\right),$$
	and $\posDblJacobiSum{\pi'}{\chi}$ and $\posDblJacobiSumScalar{\pi'}{\chi}$ are defined similarly.
\end{theorem}

The goal of this appendix is to explain how to deduce \Cref{thm:multiplicativity-in-terms-of-gauss-sums} from \Cref{thm:identity-of-jacobi-sums}.

We only need to explain this in the similitude group case $H = \GroupExtension{G}$. In this case, the restriction of $\pi'$ to $G'$ is a direct sum $\bigoplus_i \sigma_i$ of irreducible representations $\sigma_i$ of $G'$. The restriction of $\posDblJacobiSum{\pi'}{\chi}$ to $\sigma_i$ is $\posDblJacobiSum{\sigma_i}{\chi}$ and thus $\posDblJacobiSumScalar{\pi'}{\chi} = \posDblJacobiSumScalar{\sigma_i}{\chi}$. Next, the restriction of $\Ind{P}{H}{\tau \overline{\otimes} \pi'}$ to $G$ is the direct sum $\bigoplus_i \Ind{G \cap P}{G}{\tau \overline{\otimes} \sigma_i}$. Let $\Pi$ be an irreducible subrepresentation of $\Ind{G \cap P}{G}{\tau \overline{\otimes} \sigma_i}$ for some $i$. We have by \Cref{thm:identity-of-jacobi-sums} that for such $\Pi$, the operator $\posDblJacobiSum{\Pi}{\chi}$ acts by multiplication by the scalar $$\posDblJacobiSumScalar{\sigma_i}{\chi}\posDblJacobiSumScalar{\tau \times \minusInvolution{\chi_{\GL_k}}}{\involutionPlusOne{\chi}} = \posDblJacobiSumScalar{\pi'}{\chi}\posDblJacobiSumScalar{\tau \times \minusInvolution{\chi_{\GL_k}}}{\involutionPlusOne{\chi}}.$$ 
Thus we showed that for every irreducible subspace $\Pi$ of the restriction of $\Ind{P}{H}{\tau \overline{\otimes} \pi'}$ to $G$, the operator $\posDblJacobiSum{\Pi}{\chi}$ acts by the scalar $\posDblJacobiSumScalar{\pi'}{\chi}\posDblJacobiSumScalar{\tau \times \minusInvolution{\chi_{\GL_k}}}{\involutionPlusOne{\chi}}$. Using again the compatibility of Jacobi sums with restriction, this implies that the Jacobi sum $\posDblJacobiSum{\pi}{\chi}$ also acts by this scalar, as required.

\section{Extension of identities}\label{appendix:extension-of-identities}
For future purposes, we extend some of the identities in the article to the singular cases where $\involutionPlusOne{\chi} = 1$.

We have the following extension of \Cref{thm:gln-doubling-gauss-sum-in-terms-of-kondo}.
\begin{theorem}Let $\tau$ be an irreducible representation of $\GL_k\left(\quadraticExtension\right)$. Suppose that $1$ is not in the cuspidal support of $\tau$. Then
	$$\posDblJacobiSumScalar{\tau}{1} = \left(-1\right)^k q^{-\frac{k}{2}} \GaussSumScalar{\tau^{\vee}}{\fieldCharacter_{\quadraticExtension}} \GaussSumScalar{\tau}{\fieldCharacter_{\quadraticExtension}}.$$
	\begin{proof}
		Similarly to the proof of \Cref{prop:doubling-for-gln-in-terms-of-kondo}, we have
		\begin{equation}\label{eq:jacobi-sum-fourier-transform-for-trivial-character}
			\posDblJacobiSumScalar{\tau}{1} = \frac{q^{-\grpIndex{\quadraticExtension}{\finiteField} k^2 / 2}}{\sizeof{\squareMatrix_k\left(\quadraticExtension\right)}}\sum_{g \in \GL_k\left(\quadraticExtension\right)} \sum_{X \in \squareMatrix_k\left(\quadraticExtension\right)} \sum_{h \in \GL_k\left(\quadraticExtension\right)} \fieldCharacter\left(\trace \left(X\left(g+\IdentityMatrix{k}-h\right)\right)\right) \pi\left(g\right).
		\end{equation}
		Let $X = h_1 \left(\begin{smallmatrix}
			\IdentityMatrix{k-r}\\
			& 0_r
		\end{smallmatrix}\right) h_2$ be a singular matrix, where $h_1, h_2 \in \GL_k\left(\quadraticExtension\right)$ and $1 \le r \le k$, and consider the sum $$\sum_{g \in \GL_k\left(\quadraticExtension\right)} \fieldCharacter\left(\trace \left(Xg\right)\right) \pi\left(g\right) = \pi\left(h_2^{-1}\right)\sum_{g \in \GL_k\left(\quadraticExtension\right)} \fieldCharacter\left(\trace \left(\begin{pmatrix}
			\IdentityMatrix{k-r}\\
			& 0_r
		\end{pmatrix}g\right)\right)\pi\left(g\right) \pi\left(h_1^{-1}\right).$$
		Then the inner sum $$T = \sum_{g \in \GL_k\left(\quadraticExtension\right)} \fieldCharacter\left(\trace \left(\begin{pmatrix}
			\IdentityMatrix{k-r}\\
			& 0_r
		\end{pmatrix}g\right)\right)\pi\left(g\right)$$
		defines a linear map $T \circ \pi \to \pi$ satisfying $$T \circ \pi\begin{pmatrix}
			\IdentityMatrix{k-1} & X\\
			& t
		\end{pmatrix} = T$$
		for every $t \in \multiplicativegroup{\finiteField}$ and every $X \in \Mat{\left(k-1\right)}{1}\left(\quadraticExtension\right)$.
		Since $1$ is not in the cuspidal support of $\tau$, we have $T = 0$. Thus we may reduce \eqref{eq:jacobi-sum-fourier-transform-for-trivial-character} to a sum over $X \in \GL_k\left(\quadraticExtension\right)$. The rest of the proof is similar to the proofs of \Cref{prop:doubling-for-gln-in-terms-of-kondo} and \Cref{thm:gln-doubling-gauss-sum-in-terms-of-kondo}, where we use the fact that $\GaussSumSingleCharacter{1}{\fieldCharacter} = q^{-\frac{1}{2}}$.
	\end{proof}
\end{theorem}
Next, we write down the analog of the first split torus computation from \Cref{subsec:split-torus-computation-1}.

$$\sum_{\substack{x \in \multiplicativegroup{\quadraticFieldExtension{k}}\\
		x \ne -1}} \alpha\left(x\right) = -\alpha\left(-1\right) + \sum_{x \in \multiplicativegroup{\quadraticFieldExtension{k}}} \alpha\left(x\right) = -\alpha\left(-1\right) + \delta_{\alpha, 1} \cdot \left(q^k - 1\right),$$
where $\delta_{\alpha, 1}$ is Kronecker's delta function.
The analog of the second split torus computation is
\begin{theorem}\label{thm:split-torus-computation-chi-1-plus-c-equals-1}
	Let $\chi \colon \multiplicativegroup{\quadraticExtension} \to \multiplicativegroup{\cComplex}$ be such that $\involutionPlusOne{\chi} = 1$. Then
	\begin{equation*}
		\begin{split}
			&\sum_{\substack{x \in \multiplicativegroup{\quadraticFieldExtension{k}}\\
					x \ne -1}} \alpha \left(x\right) \chi\left(\aFieldNorm_{\quadraticFieldExtension{k} \slash \quadraticExtension}\left(1 + x\right)\right) \chi\left(\aFieldNorm_{\quadraticFieldExtension{k} \slash \quadraticExtension}\left(1 + \minusInvolution{x}\right)\right)\\
				=& \begin{dcases}
					q^k - 2 & \alpha = \chi^{-1} \circ \aFieldNorm_{\quadraticFieldExtension{k} \slash \quadraticExtension}, \\
					-\GaussSumCharacter{\alpha}{\chi}{\fieldCharacter_{\quadraticFieldExtension{k}}} \GaussSumCharacter{\minusInvolution{\alpha}}{\chi}{\fieldCharacter_{\quadraticFieldExtension{k}}} & \text{otherwise.}
				\end{dcases}
		\end{split}
	\end{equation*}
\end{theorem}
\begin{proof}
	We have
	$$\sum_{\substack{x \in \multiplicativegroup{\quadraticFieldExtension{k}}\\
			x \ne -1}} \alpha \left(x\right) \chi\left(\aFieldNorm_{\quadraticFieldExtension{k} \slash \quadraticExtension}\left(1 + x\right)\right) \chi\left(\aFieldNorm_{\quadraticFieldExtension{k} \slash \quadraticExtension}\left(1 + \minusInvolution{x}\right)\right) = \sum_{\substack{x \in \multiplicativegroup{\quadraticFieldExtension{k}}\\
			x \ne -1}} \alpha \left(x\right) \minusInvolution{\chi}\left(\aFieldNorm_{\quadraticFieldExtension{k} \slash \quadraticExtension}\left(x\right)\right).$$
	This equals $q^k-2$ if $\alpha = \involution{\chi} \circ \aFieldNorm_{\quadraticFieldExtension{k} \slash \quadraticExtension}$ and otherwise equals $-\alpha\left(-1\right)\chi\left(-1\right)^k$. The result now follows from the identity of Gauss sums (for $\alpha \ne \chi^{-1} \circ \aFieldNorm_{\quadraticFieldExtension{k} \slash \quadraticExtension}$)
	$$\GaussSumCharacter{\alpha}{\chi}{\fieldCharacter_{\quadraticFieldExtension{k}}} \GaussSumCharacter{\alpha^{-1}}{\chi^{-1}}{\fieldCharacter_{\quadraticFieldExtension{k}}} = \alpha\left(-1\right) \chi\left(-1\right)^k,$$
	and from the identity
	$$\GaussSumCharacter{\alpha^{-1}}{\chi^{-1}}{\fieldCharacter_{\quadraticFieldExtension{k}}} = \GaussSumCharacter{\minusInvolution{\alpha}}{\minusInvolution{\chi}}{\fieldCharacter_{\quadraticFieldExtension{k}}}$$
	and the fact $\minusInvolution{\chi} = \chi$.
\end{proof}

\begin{theorem}
	Let $\theta \colon \NormOneGroup{2m} \to \multiplicativegroup{\cComplex}$ and $\chi \colon \multiplicativegroup{\finiteField} \to \multiplicativegroup{\cComplex}$ be characters, such that $\chi^2 = 1$. Consider the sum
	\begin{equation}\label{eq:exponential-sum-quadratic-case}
		\sum_{\substack{x \in \NormOneGroup{2m}\\
				x \ne -1}} \theta \left(x\right) \chi\left(\FieldNorm{2m}{1}\left(1 + x\right)\right).
	\end{equation}
	\begin{enumerate}
		\item If $\transfer{\theta} \ne \chi \circ \FieldNorm{2m}{1}$ then \eqref{eq:exponential-sum-quadratic-case} equals $\GaussSumCharacter{\transfer{\theta}}{\chi}{\fieldCharacter_{2m}}$.
		\item If $\transfer{\theta} = \chi \circ \FieldNorm{2m}{1}$ then
		$$\sum_{\substack{x \in \NormOneGroup{2m}\\
				x \ne -1}} \theta \left(x\right) \chi\left(\FieldNorm{2m}{1}\left(1 + x\right)\right) = \begin{dcases}
			q^m & \chi = 1\\
			1 + 2q^{-m}  & \chi \ne 1.
		\end{dcases}$$
	\end{enumerate}
\end{theorem}
\begin{proof}
	If $\chi = 1$ we have that
	$$\sum_{\substack{x \in \NormOneGroup{2m}\\
			x \ne -1}} \theta \left(x\right) = \begin{dcases}
		-\theta\left(-1\right) & \theta \ne 1,\\
		q^m & \theta = 1.
	\end{dcases}$$
	By \cite[Proposition A.2]{Zelingher2023} we have that $\GaussSumCharacter{\transfer{\theta}}{1}{\fieldCharacter_{2m}} = -\theta\left(-1\right)$ for $\transfer{\theta} \ne 1$.
	
	Suppose that $\chi \ne 1$. Proceeding as in \Cref{subsubsec:elliptic-torus-computation-E-equal-F} and using $\chi^{-1} = \chi$ we have that \eqref{eq:exponential-sum-quadratic-case} equals
	\begin{equation}\label{eq:exponential-sum-quadratic-case-two}
	\begin{split}
		& -q^{-m} \GaussSumSingleCharacter{\chi \circ \FieldNorm{2m}{1}}{\fieldCharacter_{2m}} \\
		& \times \frac{1}{q^m-1} \sum_{t,z \in \multiplicativegroup{\finiteFieldExtension{2m}}} \chi\left(\FieldNorm{2m}{1}\left(tz\right)\right) \transfer{\theta}\left(t\right) \fieldCharacter_{2m}\left(\trace_{\finiteFieldExtension{2m} \slash \finiteFieldExtension{m}}\left(z\right) t\right).
	\end{split}
	\end{equation}

	Let $\delta \in \multiplicativegroup{\finiteFieldExtension{2m}}$ be an element such that $\delta + \delta^{q^m} = 0$. Then we may write
	\begin{align*}
		&\sum_{t,z \in \multiplicativegroup{\finiteFieldExtension{2m}}} \chi\left(\FieldNorm{2m}{1}\left(tz\right)\right) \transfer{\theta}\left(t\right) \fieldCharacter_{2m}\left(\trace_{\finiteFieldExtension{2m} \slash \finiteFieldExtension{m}}\left(z\right) t\right) \\
		=& \sum_{z \in \multiplicativegroup{\finiteFieldExtension{m}},t \in \multiplicativegroup{\finiteFieldExtension{2m}}} \chi\left(\FieldNorm{2m}{1}\left(\delta tz\right)\right) \transfer{\theta}\left(t\right) \\
		& + \sum_{\substack{t,z \in \multiplicativegroup{\finiteFieldExtension{2m}}\\
				\trace_{\finiteFieldExtension{2m} \slash \finiteFieldExtension{m}}\left(z\right) \ne 0}} \chi\left(\FieldNorm{2m}{1}\left(tz\right)\right) \transfer{\theta}\left(t\right) \fieldCharacter_{2m}\left(\trace_{\finiteFieldExtension{2m} \slash \finiteFieldExtension{m}}\left(z\right) t\right).
	\end{align*}
	Regarding the first sum, since $\FieldNorm{2m}{1}\left(z\right) = \FieldNorm{m}{1}\left(z\right)^2 \in \multiplicativegroup{\finiteFieldExtension{m}}$ and since $\chi^2 = 1$ we have $\chi\left(\FieldNorm{2m}{1}\left(z\right)\right) = 1$. We have that $\delta^{q^m + 1} = -\delta^2$ and that $\FieldNorm{m}{1}\left(\delta^2\right)^{\frac{q-1}{2}} = \delta^{q^m - 1} = -1$ and thus $\chi\left(\FieldNorm{m}{1}\left(\delta^2\right)\right) = -1$. Therefore since $\chi\left(\FieldNorm{2m}{1}\left(\delta\right)\right) = \chi\left(\FieldNorm{m}{1}\left(-1\right)\right) \chi\left(\FieldNorm{m}{1}\left(\delta^2\right)\right)$, we have  $$\sum_{z \in \multiplicativegroup{\finiteFieldExtension{m}},t \in \multiplicativegroup{\finiteFieldExtension{2m}}} \chi\left(\FieldNorm{2m}{1}\left(\delta tz\right)\right) \transfer{\theta}\left(t\right) = -\chi\left(-1\right)^m \sum_{t \in \multiplicativegroup{\finiteFieldExtension{2m}}} \chi\left(\FieldNorm{2m}{1}\left(t\right)\right) \transfer{\theta}\left(t\right)$$
	which equals $-\chi\left(-1\right)^m \left(q^{2m} - 1\right)\delta_{\transfer{\theta}, \chi \circ \FieldNorm{2m}{1}}$.
	Regarding the second sum, by changing variables, we have that
	\begin{equation*}
		\begin{split}
			&\sum_{\substack{t,z \in \multiplicativegroup{\finiteFieldExtension{2m}}\\
					\trace_{\finiteFieldExtension{2m} \slash \finiteFieldExtension{m}}\left(z\right) \ne 0}} \chi\left(\FieldNorm{2m}{1}\left(tz\right)\right) \transfer{\theta}\left(t\right) \fieldCharacter_{2m}\left(\trace_{\finiteFieldExtension{2m} \slash \finiteFieldExtension{m}}\left(z\right) t\right) \\
			=& \sum_{\substack{z \in \multiplicativegroup{\finiteFieldExtension{2m}}\\
					\trace_{\finiteFieldExtension{2m} \slash \finiteFieldExtension{m}}\left(z\right) \ne 0}} \chi\left(\FieldNorm{2m}{1}\left(\frac{z}{\trace_{\finiteFieldExtension{2m} \slash \finiteFieldExtension{m}}\left(z\right)}\right)\right) \sum_{t \in \multiplicativegroup{\finiteFieldExtension{2m}}} \chi\left(\FieldNorm{2m}{1}\left(t\right)\right) \transfer{\theta}\left(t\right) \fieldCharacter_{2m}\left(t\right).
		\end{split}
	\end{equation*}
	which by \Cref{appendix:e-equals-f} equals \begin{equation*}
		\begin{split}
			& -q^m \left(q^m-1\right) \GaussSumSingleCharacter{\chi^{-1} \circ \FieldNorm{m}{1}}{\fieldCharacter_m}^2
			\GaussSumSingleCharacter{\chi^{2} \circ \FieldNorm{m}{1}}{\fieldCharacter_m} q^{\frac{m}{2}} \GaussSumCharacter{\transfer{\theta}}{\chi}{\fieldCharacter_{2m}} \\
			= & - q^{m} \left(q^m-1\right) \GaussSumSingleCharacter{\chi^{-1} \circ \FieldNorm{2m}{1}}{\fieldCharacter_{2m}} \GaussSumCharacter{\transfer{\theta}}{\chi}{\fieldCharacter_{2m}},
		\end{split}
	\end{equation*}
	where we used the fact $\GaussSumSingleCharacter{\chi^{2} \circ \FieldNorm{m}{1}}{\fieldCharacter_m} = q^{-\frac{m}{2}}$ and the Hasse--Davenport relation.
	Thus if $\transfer{\theta} \ne \chi \circ \FieldNorm{2m}{1}$ then \eqref{eq:exponential-sum-quadratic-case-two} equals
	\begin{equation*}
		\GaussSumCharacter{\transfer{\theta}}{\chi}{\fieldCharacter_{2m}},
	\end{equation*}
	where we used the relation $\GaussSumSingleCharacter{\chi^{-1} \circ \FieldNorm{2m}{1}}{\fieldCharacter_{2m}} \GaussSumSingleCharacter{\chi \circ \FieldNorm{2m}{1}}{\fieldCharacter_{2m}} = 1$.
	
	On the other hand, if $\transfer{\theta} = \chi \circ \FieldNorm{2m}{1}$ then \eqref{eq:exponential-sum-quadratic-case-two} equals
	\begin{align*}
		&-\frac{q^{-m}}{q^m-1} \GaussSumSingleCharacter{\chi \circ \FieldNorm{2m}{1}}{\fieldCharacter_{2m}} \left(-\chi\left(-1\right)^m \left(q^{2m} - 1\right) - \GaussSumSingleCharacter{\chi^{-1} \circ \FieldNorm{2m}{1}}{\fieldCharacter_{2m}} \left(q^m - 1\right)\right) \\
		=& 1+2q^{-m}.
	\end{align*}
	Here we used the relation
	$$\GaussSumSingleCharacter{\chi \circ \FieldNorm{2m}{1}}{\fieldCharacter_{2m}} = \GaussSumSingleCharacter{\chi \circ \FieldNorm{m}{1}}{\fieldCharacter_{m}} \GaussSumSingleCharacter{\chi^{-1} \circ \FieldNorm{m}{1}}{\fieldCharacter_{m}} = \chi\left(-1\right)^m,$$
	where the first equality uses the Hasse--Davenport relation and the fact that $\chi^{-1} = \chi$.
\end{proof}
\begin{theorem}
	Let $\chi \colon \multiplicativegroup{\finiteFieldExtension{2}} \to \multiplicativegroup{\cComplex}$ be a character such that $\involutionPlusOne{\chi} = 1$. Let $m$ be an odd integer. Let $\theta \colon \NormOneGroup{m} \to \multiplicativegroup{\cComplex}$ be a character. Consider the sum \begin{equation}\label{eq:exponential-sum-unitary-case}
		\sum_{\substack{x \in \NormOneGroup{2m}\\
				x \ne -1}} \theta \left(x\right) \chi\left(\FieldNorm{2m}{2} \left(1 + x\right)\right)
	\end{equation}
	\begin{enumerate}
		\item If $\transfer{\theta} \ne \chi^{-1} \circ \FieldNorm{2m}{2}$ then \eqref{eq:exponential-sum-unitary-case} equals $\GaussSumCharacter{\transfer{\theta}}{\chi}{\fieldCharacter_{2m}}$.
		\item If $\theta = \chi^{-1} \circ \FieldNorm{2m}{2}$ then \eqref{eq:exponential-sum-unitary-case} equals $q^m$.
	\end{enumerate}
\end{theorem}
\begin{proof}
	We first compute this for $m=1$. Proceeding as in \Cref{appendix:e-not-equal-f} we have that \eqref{eq:exponential-sum-unitary-case} equals 
	$$\frac{1}{q\left(q-1\right)} \sum_{z \in \finiteField} \sum_{x \in \multiplicativegroup{\finiteFieldExtension{2}}}\sum_{y \in \multiplicativegroup{\finiteField}}  \transfer{\theta}\left(x\right) \chi^{-q}\left(x\right) \fieldCharacter\left(z\left(\trace_{\finiteFieldExtension{2} \slash \finiteField} \left(x\right) - y\right)\right).$$
	We separate two sums: one for $z = 0$ and another for $z \in \multiplicativegroup{\finiteField}$. If $z = 0$, we arrive at the sum
	$$\frac{1}{q} \sum_{x \in \multiplicativegroup{\finiteFieldExtension{2}}} \transfer{\theta}\left(x\right) \chi^{-q}\left(x\right).$$
	which equals $\frac{q^2 - 1}{q}$ if $\transfer{\theta} = \chi^{q} = \chi^{-1}$ and otherwise equals zero. For $z \ne 0$, we have that the sum over $y$ is $-1$. Changing variables $x \mapsto z^{-1} x$ and using the fact that $\transfer{\theta}$ and $\chi$ are trivial on $\multiplicativegroup{\finiteField}$, we arrive at the sum
	$$-\frac{1}{q} \sum_{x \in \multiplicativegroup{\finiteFieldExtension{2}}} \transfer{\theta}\left(x\right) \chi^{-q}\left(x\right) \fieldCharacter_2\left(x\right),$$
	which by changing variable $x \mapsto x^q$ equals $\GaussSumCharacter{\transfer{\theta}}{\chi}{\fieldCharacter_2}$. Thus we have that for $m=1$, \eqref{eq:exponential-sum-unitary-case} equals
	$$\left(q - \frac{1}{q}\right) \delta_{\transfer{\theta}, \chi^{-1}} + \GaussSumCharacter{\transfer{\theta}}{\chi}{\fieldCharacter_2}.$$
	If $\transfer{\theta} = \chi^{-1}$ then $\GaussSumCharacter{\transfer{\theta}}{\chi}{\fieldCharacter_2} = q^{-1}$ and thus we get the result for $m=1$.
	
	For general $m$ we obtain the identity by replacing $\finiteField$ with $\finiteFieldExtension{m}$ and $\chi$ with $\chi \circ \FieldNorm{2m}{2}$ using the fact that $m$ is odd so $\chi\left(\FieldNorm{2m}{2}\left(\xi\right)\right)^{1+q^m} = \chi\left(\FieldNorm{2m}{2}\left(\xi\right)\right)^{1+q} = 1$ for $\xi \in \multiplicativegroup{\finiteFieldExtension{2m}}$.
\end{proof}

We conclude with a result analogous to \Cref{thm:computation-of-doubling-gauss-sum-scalar-for-deligne-lusztig-characters}.

\begin{theorem}
		Let $\chi \colon \multiplicativegroup{\quadraticExtension} \to \multiplicativegroup{\cComplex}$ be a character such that $\involutionPlusOne{\chi} = 1$. Suppose that $\algebraicGroup{T}_{\algebraicGroup{H}}$ is a maximal rational torus of $\algebraicGroup{H}$ that corresponds to the pair of partitions $\left(\lambda^+, \lambda^-\right)$ and let $\algebraicGroup{T}_{\algebraicGroup{G}} \coloneq \algebraicGroup{G} \cap \algebraicGroup{T}_{\algebraicGroup{H}}$. Denote $T_H = \algebraicGroup{T}_{\algebraicGroup{H}}^{\Frobenius}$ and $T = T_G = \algebraicGroup{T}_{\algebraicGroup{G}}^{\Frobenius}$. Suppose that $\theta_{H} \colon T_H \to \multiplicativegroup{\cComplex}$ is a character. Under the notation of \Cref{subsec:some-notation} suppose that 
		\begin{enumerate}
			\item $\alpha_j \ne \chi^{-1} \circ \aFieldNorm_{\quadraticFieldExtension{k_j} \slash \quadraticExtension}$ for every $j$. \item If $\quadraticExtension = \finiteField$, then $\transfer{\theta}_i \ne \chi \circ \FieldNorm{2m_i}{1}$ for every $i$.
		\end{enumerate}
		Let $$\theta = \theta_{G} \coloneq \theta_{H} \restriction_{T_G} \colon T_G \to \multiplicativegroup{\cComplex}.$$ Then
	\begin{align*}
		\dblPosVirtualJacobiSumScalar{\RTGTheta{T_H}{H}{\theta_H}}{\chi} =& \varepsilon_{\algebraicGroup{G}} q^{-\frac{1}{2} \left\lfloor\frac{ \dim_{\finiteField} \hermitianSpace}{2}\right\rfloor} \RTGTheta{T_H}{H}{\theta_H}\left(1\right) \GaussSumTorusCharacter{\transfer{T}}{\transfer{\theta}}{\chi}{\fieldCharacter}^{-1}\\
		& \cdot \begin{dcases}
			\chi\left(2\right) & \quadraticExtension = \finiteField, \text{ and } \dim_{\finiteField}\hermitianSpace \text{ is odd},\\
			1 & \text{otherwise.}
		\end{dcases}
	\end{align*}
\end{theorem}
\begin{proof}
	The proof is similar to \Cref{thm:computation-of-doubling-gauss-sum-scalar-for-deligne-lusztig-characters}. We use the computations above, noticing that $$\GaussSumCharacter{\alpha_j}{\chi}{\fieldCharacter_{\quadraticFieldExtension{k_j}}} \GaussSumCharacter{\minusInvolution{\alpha_j}}{\chi}{\fieldCharacter_{\quadraticFieldExtension{k_j}}} = \GaussSumCharacter{\alpha_j}{\chi}{\fieldCharacter_{\quadraticFieldExtension{k_j}}}^{-1} \GaussSumCharacter{\minusInvolution{\alpha_j}}{\chi}{\fieldCharacter_{\quadraticFieldExtension{k_j}}}^{-1},$$
	(both sides equal $\alpha\left(-1\right) \chi\left(-1\right)^k$, see the proof of \Cref{thm:split-torus-computation-chi-1-plus-c-equals-1}),
	that if $\quadraticExtension = \finiteField$ and $\transfer{\theta}_i \ne \chi \circ \FieldNorm{2m_i}{1}$ then
	$$	\GaussSumCharacter{\transfer{\theta}_i}{\chi}{\fieldCharacter_{2m_i}}^{-1} =  \GaussSumCharacter{\transfer{\theta}_i}{\chi}{\fieldCharacter_{2m_i}},$$
	and that if $\quadraticExtension \ne \finiteField$ then
	$$\GaussSumCharacter{\transfer{\theta}_i}{\chi}{\fieldCharacter_{2m_i}}^{-1} = \begin{dcases}
		q^m & \transfer{\theta}_i = \chi^{-1} \circ \FieldNorm{2m_i}{2},\\
		\GaussSumCharacter{\transfer{\theta}_i}{\chi}{\fieldCharacter_{2m_i}} & \text{otherwise.}
	\end{dcases}$$
\end{proof}

\bibliographystyle{abbrv}
\bibliography{references}
\end{document}